\pdfoutput=1
\documentclass[a4paper,11pt]{article}

\usepackage[subsec]{modstyle}
\usepackage{microtype}

\newcommand{\ctxt}{{F}}
\newcommand{\nrms}{{N}}
\newcommand{\SFF}{\Span_{\nrms}(\cF)}

\title{Normed equivariant ring spectra \\ and higher Tambara functors}
\author{Bastiaan Cnossen, Rune Haugseng, Tobias Lenz, and Sil Linskens}

\date{July 11, 2024}

\tikzcdset{arrow style=tikz,diagrams={>=Classical TikZ Rightarrow}}

\begin{document}
\maketitle

\begin{abstract}
  In this paper we extend equivariant infinite loop space theory
  to take into account multiplicative norms: For every finite group $G$, we
  construct a multiplicative refinement of the comparison between the \icats{} of 
  connective genuine $G$-spectra and space-valued Mackey functors,
  first proven by Guillou--May, and use this to give a
  description of connective normed equivariant ring spectra as
  space-valued Tambara functors.

  In more detail, we first introduce and study a general notion of
  homotopy\hskip0pt-coherent normed (semi)rings, and identify these with
  product-preserving functors out of a corresponding $\infty$-category
  of bispans. In the equivariant setting, this identifies space-valued
  Tambara functors with normed algebras with respect to a certain
  normed monoidal structure on grouplike $G$-commutative monoids in
  spaces. We then show that the latter is canonically equivalent to
  the normed monoidal structure on connective $G$-spectra given by the
  Hill--Hopkins--Ravenel norms.  Combining our comparison with results
  of Elmanto--Haugseng and Barwick--Glasman--Mathew--Nikolaus, we
  produce normed ring structures on equi\-variant algebraic K-theory
  spectra.
\end{abstract}

\begingroup\tosfstyle
\let\oldcontentsline=\contentsline
\def\contentsline#1#2#3#4{\oldcontentsline{#1}{#2}{\texttosf{#3}}{#4}}
\tableofcontents
\endgroup

\section{Introduction}
Infinite loop space theory has played an important role in algebraic
topology since the 1970's, giving a way to construct interesting
examples of spectra from space-level data. At its heart lies the
\emph{Recognition Theorem} \cite{MayGeomIter,BoardmanVogt, segal}, 
which in modern language describes
connective spectra as commutative group objects in the
\icat{} of spaces, i.e.
\begin{equation}
  \label{eq:SpconnisCGrp}
\Sp^{\geq 0} \simeq \catname{CGrp}(\Spc).
\end{equation}
Such commutative groups arise in nature, for instance, by
group-completing the classifying spaces of symmetric monoidal ($\infty$-)categories. As an important
example, applying this to the groupoid of finitely generated
projective modules over a ring $R$ with symmetric monoidal structure
via direct sum yields the algebraic K-theory spectrum of the ring $R$ \cite{may-perm, segal}.

In order to obtain spectra with algebraic structures, we need to
upgrade \cref{eq:SpconnisCGrp} to take into account
\emph{multiplicative} structures. For $1$-categorical inputs, such
\emph{multiplicative infinite loop spaces machines} have been
considered for example by May \cite{may-biperm,may-biperm2} and Elmendorf--Mandell \cite{elmendorf-mandell}. Working in the \icatl{}
framework, Gepner--Groth--Nikolaus \cite{GepnerGrothNikolaus} both generalized and elucidated these constructions: they show that there is a canonical symmetric monoidal structure on the \icat{}
$\CMon(\Spc)$ of commutative monoids in spaces, which localizes to
commutative groups, and that \cref{eq:SpconnisCGrp} uniquely upgrades to an
equivalence of symmetric monoidal \icats{} where we equip
$\Sp^{\geq 0}$ with the smash product. The tensor product on
$\CGrp(\Spc)$ is an \icatl{} analogue of the tensor product of abelian
groups, so it is natural to think of a
commutative algebra object of $\CGrp(\Spc)$ as a commutative ring in
$\Spc$; as a direct consequence of the multiplicative comparison, we then have an equivalence
\begin{equation}\label{eq:Sp-as-mackey}
 \catname{CAlg}(\Sp^{\geq 0}) \simeq \CAlg\big(\CGrp(\Spc)\big)\eqqcolon \catname{CRing}(\Spc)
\end{equation}
between connective commutative ring spectra and commutative ring spaces.

\subsection*{A multiplicative equivariant recognition theorem}
Our first goal is to extend the above story to equivariant
spectra over a finite group $G$. While the study of such
\emph{equivariant infinite loop space machines} began in the late 70's
(in unpublished work of Segal and Hauschild--May--Waner), the subject
has experienced a renaissance in recent years. As part of this, its
point-set level foundations have been rewritten and extended by May
and his collaborators \cite{guillou-may, gm-operad, mmo}, and new
\icatl{} approaches to the subject have been introduced by Barwick and collaborators \cite{BarwickMackey, exposeI}.

In the present paper, we will adopt the latter perspective. For this, recall from \cite{cmmn-Mackey}*{A.1} that the \icat{} $\Sp_G$ of $G$-spectra, defined classically as the Dwyer--Kan localization of a suitable model category of orthogonal or symmetric spectra with $G$-action, admits a purely \icatl{} description as
\[ \Sp_{G} \simeq \Fun^{\times}(\Span(\xF_{G}), \Sp),\]
that is, as product-preserving functors from a $(2,1)$-category of
spans of finite $G$-sets to spectra;  see also \cite{guillou-may} for a
model-categorical version.

From this equivalence and
\cref{eq:SpconnisCGrp}, one immediately obtains an \emph{Equivariant Recognition Theorem}, in the form of a space-level description of {connective}
$G$-spectra as
\begin{equation}
  \label{eq:SpGconnisCGrpG}
  \Sp_G^{\geq 0} \simeq \Fun^{\times}_{\mathrm{grp}}(\Span(\xF_{G}),
  \Spc);
\end{equation}
here the right-hand side consists of functors $F \colon \Span(\xF_{G}) \to
\Spc$ that preserve products and that are \emph{grouplike}, in the sense that for every $G$-set $X$ the
commutative multiplication given by
\[ F(X) \times F(X) \simeq F(X \amalg X) \to F(X)\]
makes $F(X)$ a commutative group. Analogously to the non-equivariant situation, such objects arise naturally from \icatl{} data, and this provides one possible approach to equivariant algebraic K-theory \cite{BarwickMackey2}.

To shed some light on this equivalence,
recall \cite{Dress} that a \emph{Mackey functor} $M$ for $G$ consists of abelian
groups $M(X)$ for every finite $G$-set $X$ together with restriction
and additive norm (or ``transfer'') maps
\[
  f^{*} \colon M(X) \to M(Y), \quad f_{\oplus} \colon M(Y) \to
  M(X) \]
for every morphism $f \colon Y \to X$ of $G$-sets, such that $M$ takes
disjoint unions to products, both restrictions and norms are
functorial, and they compose according to a double coset
formula. The zeroth homotopy groups of any $G$-spectrum form a Mackey
functor, and Mackey functors are the most general coefficients for
ordinary equivariant (co)homology \cite{lms}*{\S V.9}.
The data of a Mackey functor can be neatly organized into a
product-preserving functor
\[ \Span(\xF_{G}) \to \Ab,\] or equivalently a functor
$\Span(\xF_{G}) \to \Set$ that preserves products and such that the
induced (commutative) multiplication on the value at every $G$-set has
inverses. Thus we may think of the equivalence
\cref{eq:SpGconnisCGrpG} as saying that connective $G$-spectra are
\emph{space-valued Mackey functors}.

Just like a space-valued Mackey functor contains more information than a commutative group in the \icat{} of $G$-spaces (namely, in the form of additive norms), a multiplicative refinement of the equivalence \cref{eq:SpGconnisCGrpG} should not just take the ordinary symmetric monoidal structures on both sides into account (arising via the smash product and Day convolution, respectively), but additionally respect suitable \emph{symmetric monoidal norms}. To make this precise, note that if $\cC$ is any \icat{} with finite products, we can more generally define a
\emph{normed $G$-monoid} in $\cC$ to be a functor
\[ M \colon \Span(\xF_{G}) \to \cC\]
that preserves finite products; this amounts to specifying a commutative
monoid $M(G/H)$ with an action of the Weyl group $W_GH\coloneqq N_GH/H$ for every subgroup $H$ of $G$, together with restrictions $\Res^K_H\colon M(G/K)\to M(G/H)$ and
norm maps $\Nm^K_H\colon M(G/H) \to M(G/K)$ for all subgroups $H\leqslant K\leqslant G$ as well as various coherences. The equivalence (\ref{eq:Sp-as-mackey}) can then be restated as saying that genuine $G$-spectra are normed $G$-monoids in spectra, while the equivalence (\ref{eq:SpGconnisCGrpG}) says that connective $G$-spectra are equivalently ``grouplike'' normed $G$-monoids in spaces, or \emph{normed $G$-groups} for short.

On the other hand, taking $\cC$ to be $\CatI$, we obtain the notion of a
\emph{normed $G$-\icat{}} as the equivariant version of a symmetric monoidal \icat{}. Many important symmetric monoidal \icats{} studied in equivariant homotopy theory turn out to admit natural refinements to normed $G$-\icats{}; in particular, there is a normed $G$-\icat{}
\[
  \ul\Sp_G\colon G/H\mapsto \Sp_H
\]
whose contravariant functoriality is given by the evident
restrictions, and whose covariant functoriality encodes the smash
product of equivariant spectra together with the
\emph{Hill--Hopkins--Ravenel norms} \cite{HHR}. We then
prove:

\begin{introthm}[See \Cref{thm:recognition}]\label{introthm:Gsymmeq}
  The $G$-\icat{}
  \[
    \ul\NMon_{G}(\Spc)\colon G/H\mapsto \Fun^\times(\Span(\xF_H),\Spc)
  \]
   of $G$-normed monoids in spaces has a canonical normed structure
   that localizes to $\ul\NGrp_{G}(\Spc)$. Furthermore, the equivalence
  \cref{eq:SpGconnisCGrpG} upgrades to an equivalence
  \[ \ul\Sp_G^{\geq 0} \simeq \ul\NGrp_{G}(\Spc)\]
  of normed $G$-\icats{}, where the left-hand side carries the restriction of the normed structure described above.
\end{introthm}

\subsection*{Normed \boldmath $G$-ring spectra}
As a direct consequence of Theorem~\ref{introthm:Gsymmeq}, we obtain equivalences between connective $G$-spectra equipped with extra ``parametrized algebraic structure'' and $G$-commutative groups equipped with the same structure. To make this precise, note that by straightening--unstraightening we can equivalently regard a normed $G$-\icat{} $\cC$ as a cocartesian
fibration
\[ \cC^{\otimes} \to \Span(\xF_{G}).\]
Following Bachmann--Hoyois \cite{norms} we define a \emph{$G$-normed
  algebra} in $\cC$ as a section $\Span(\xF_{G}) \to \cC^{\otimes}$
that takes the backward maps to cocartesian morphisms in
$\cC^{\otimes}$; this amounts to specifying a commutative algebra $A$ in the underlying \icat{} $\cC(G/G)$ together with suitably coherent \emph{normed multiplications} $\mathop{\Nm}^G_H\mathop{\Res}^G_H A\to A$ for every $H\leqslant G$.

In particular, we get a notion of $G$-normed
algebras in $\ul\Sp_{G}$, or \emph{normed $G$-spectra} for short, which are generally expected to be equivalent to
the objects obtained as strict commutative algebras in the $1$-categories of symmetric or
orthogonal $G$-spectra. Theorem~\ref{introthm:Gsymmeq} then shows that \emph{connective} normed $G$-spectra can equivalently be described as normed algebras in $\ul\NGrp_G(\Spc)$, i.e.~as ``normed $G$-rings.'' As our second main result, we then build on this comparison to give a space-level description of connective normed $G$-spectra, generalizing the result for $G=1$ proven in \cite{CHLL1}:

\begin{introthm}[See Theorem~\ref{thm:Main_Theorem}]\label{introthm:tambara}
  There is an equivalence of \icats{}
  \begin{equation}
    \label{eq:normedTambaraeq}
  \NAlg_{G}(\Sp_G^{\geq 0}) \simeq
    \Fun^{\times}_{\mathrm{grp}}(\Bispan(\xF_{G}), \Spc).
  \end{equation}
\end{introthm}
Here $\Bispan(\xF_{G})$ is the $(2,1)$-category of \emph{bispans} of finite
$G$-sets in the sense of \cite{bispans}: its objects are finite $G$-sets, and morphisms are
given by diagrams
\begin{equation}\label{eq:generic-bispan}
  A\xlongleftarrow{R} B\xlongrightarrow{N} C\xlongrightarrow{T} D;
\end{equation}
the composition law in $\Bispan(\xF_G)$ is somewhat involved and encodes both the Mackey double coset formulas for commuting restrictions past norms and transfers, as well as a distributivity relation between norms and transfers. Moreover,
$\Fun^\times_{\textup{grp}}$ again denotes the full subcategory of
those product-preserving functors that are \emph{grouplike} in a
suitable sense (see Definition~\ref{def:tambara} for details).

Recall \cite{tambara-tnr} that a \emph{Tambara functor} $X$ for a
finite group $G$ is an assignment of an abelian group $X(G/H)$ for
every subgroup of $G$ together with compatible restriction, transfer,
and norm maps for every subgroup inclusion. A Tambara functor is thus a
multiplicative enhancement of a Mackey functor, and this is precisely
the structure existing on the zeroth equivariant homotopy groups of a
strictly commutative $G$-ring spectrum, see \cite{Brun}*{\S 7.2} and \cite{ullman}. Tambara functors can equivalently be described 
\cite{StricklandTambara} as
grouplike product-preserving functors $\Bispan(\xF_G)\to\Set$ (with restrictions, transfers, and norms corresponding to the functoriality in the components $R$, $T$, and $N$ of the bispan (\ref{eq:generic-bispan}), respectively), and
we can therefore think of the equivalence \cref{eq:normedTambaraeq} as
identifying connective normed $G$-spectra with \emph{space-valued
  Tambara functors}.

In fact, we deduce \Cref{introthm:tambara} from a much more general result: following Bachmann \cite{BachmannMotivic}, we consider \emph{normed \icats{}} as functors from suitable span \icats{} into $\CatI$, and we give a general description of \emph{normed (semi)rings} in this context in terms of product-preserving functors out of an \icat{} of bispans, see Theorems~\ref{thm:cmon-bispan-model} and~\ref{thm:ring-comparison}. This in particular allows us to deduce a version of Theorem~\ref{introthm:tambara} with fewer normed multiplications, in which case we can describe the corresponding connective normed algebras as a space-valued version of the \emph{incomplete Tambara functors} considered by Blumberg--Hill \cite{blumberg-hill}.

\subsection*{Multiplicative equivariant K-theory}
As a concrete application of \cref{introthm:tambara}, we can construct
normed multiplicative structures on equivariant algebraic K-theory
spectra: Recall that Elmanto and Haugseng \cite{bispans}*{\S 4.3} show that if $E$ is a
normed $G$-spectrum, then the assignment
\[ H \mapsto \Mod_{E^{H}}(\Sp_{H})\]
extends naturally to a functor
\[ \Bispan(\xF_{G}) \to \CatI \]
that preserves products and takes values in the subcategory of stable \icats{} and polynomial
functors. Combining this with the polynomial functoriality of
(connective) algebraic K-theory of Barwick, Glasman, Mathew, and
Nikolaus~\cite{polynomials}, we obtain a space-valued Tambara functor
given by
\[ H \mapsto \Omega^{\infty}K(\Mod_{E^{H}}(\Sp_{H})).\]
Now \cref{introthm:tambara} identifies this with a normed
$G$-spectrum; as the constructions involved are functorial, we obtain:
\begin{introcor}
  Connective equivariant algebraic K-theory can be enhanced to a functor
  \[ K \colon \NAlg_{G}(\Sp_{G}) \to \NAlg_{G}(\Sp_G^{\geq 0}).\]
\end{introcor}

More generally, we obtain normed $G$-spectra from suitable normed stable $\infty$-categories \cite{bispans}*{4.3.2}. Specializing this as in \cite{bispans}*{4.3.9} we in particular obtain a refinement of connective equivariant algebraic $K$-theory of stable \icats{} to a functor from symmetric monoidal stable \icats{} to normed $G$-spectra. In the case where $G$ is a finite $2$-group, an entirely different approach to
such a refinement has previously been worked out by
Hilman~\cite{Hilman-Kth}; to the best of our knowledge, ours is the first construction in this generality.

\subsection*{Related work}
During the long history of equivariant infinite loop space theory, a wide range of notions of ``$G$-commutative monoids'' have been introduced and studied, for example Shimakawa's \emph{special $\Gamma$-$G$-spaces} \cite{shimakawa-equivariant}, the operadic models of Guillou--May \cite{gm-operad}, various ``ultra-commutative'' models \cite{g-global,global-star}, and the $\infty$-categorical model \cite{guillou-may,BarwickMackey} used in this paper. All of these notions are known to be equivalent to each other \cite{mmo, genuine-naive, marckey}, and in particular each of them comes with an equivariant recognition theorem relating the corresponding grouplike objects to connective $G$-spectra.

Since the early days of the subject, much effort went into the search
for multiplicative refinements of these comparisons, with several
breakthroughs in the last couple of years. In particular,
Guillou--May--Merling--Osorno studied multiplicative properties of the
operadic machine, culminating in the article \cite{gmmo-mult} where
they refine equivariant infinite loop space theory to an enriched
multifunctor, allowing the construction of \emph{non-commutative}
$G$-ring spectra from space-level or categorical data. Yau
\cite{yau-mult} recently improved this to a \emph{symmetric}
enriched multifunctor, which then in particular can also be used to produce
\emph{commutative} $G$-ring spectra.

In contrast to our approach, the aforementioned authors work with
\emph{strict} (non-parametrized) algebraic structures on the level of
$1$-categorical models. While commutative structures on
symmetric/orthogonal $G$-spectra are expected to model
 \icatl{} $G$-normed spectra, a symmetric monoidal or symmetric multifunctor
structure on the functor from $G$-commutative monoids to connective
$G$-spectra does \emph{not} induce any such structure on the inverse
functor, and accordingly there is no analogue of our
Theorem~\ref{introthm:tambara} known in these settings. In fact, there
are serious obstructions to achieving a complete space-level
description of connective commutative $G$-ring spectra along these
lines: for example, Lawson \cite{lawson} proved that even for $G=1$
not all connective commutative ring spectra arise from strictly
commutative algebras in $\Gamma$-spaces.

\subsection*{Outline}
In Section~\ref{sec:bg} we recall some necessary background about \icats{} of spans and bispans. We then introduce the framework of \emph{normed monoids}, \emph{normed \icats{}}, and \emph{normed algebras} in Section~\ref{sec:normed-icats} as a very mild generalization of work of Bachmann and Hoyois. As the main new result of that section, we construct a \emph{Day convolution} normed structure on certain \icats{} of product-preserving functors and give a description of normed algebras in it, see \cref{propn:FunTsymmon}.

Specializing this, we then construct normed \icats{} of {normed monoids} in Section~\ref{sec:commring}, which in particular allows us to define \emph{normed (semi)rings}. Combining the description of normed algebras with respect to the Day convolution structure with our results in \cite{CHLL1}, we then show that normed rings can be equivalently described as certain higher Tambara functors (\Cref{thm:ring-comparison}).

In Section~\ref{sec:Normed_Spectra} we introduce and compare various normed \icats{} related to equivariant homotopy theory, in particular proving \Cref{introthm:Gsymmeq}. Combining this with the results of the previous section, we then finally deduce \Cref{introthm:tambara}.

The paper ends with a short appendix on a \emph{Borel construction} due to Hilman \cite{borel-normed} that builds normed $G$-\icats{} from ordinary symmetric monoidal \icats{} with $G$-action, which is used in various constructions in Section~\ref{sec:Normed_Spectra}.

\subsection*{Notations and conventions}
\begin{itemize}
  \item We write $\xF$ for the category of finite sets, and $\bfu{n} \coloneqq \{1,\dots,n\}$
  for the standard set with $n$ elements. For an $\infty$-category $\cC$, we write $\xF[\cC]$ for the finite coproduct completion of $\cC$. When $\cC$ is the orbit category $\bO_G$ of a finite group $G$, we denote the category of finite $G$-sets by $\xF_G=\xF[\bO_G]$.
  \item Functors that preserve \textit{finite} products will be referred to as \emph{product-preserving} for short, and similarly for coproducts. We will never speak about arbitrary products and coproducts.
  \item We write $\CatI$ for the \icat{} of \icats{} and $\Spc$ for the
  \icat{} of spaces (a.k.a. \igpds{}).
  \item If $\cC$ is an \icat{}, then we denote its underlying \igpd{} by $\cC^{\simeq}$ or $\cC_{\eq}$, depending on context.
  \item We write $\Ar(\cC) \coloneqq \Fun([1],
  \cC)$ for the \emph{arrow \icat{}} of $\cC$.
  \item Throughout, we use the word \emph{subcategory} to refer to what is sometimes called a \emph{replete subcategory}: that is, for us a subcategory $\cC_0\subseteq\cC$ is always required to contain all equivalences between its objects. A subcategory is called \emph{wide} if it in addition
  contains all objects.
\end{itemize}

\subsection*{Acknowledgments}
The first author is an associate member of the SFB 1085 Higher Invariants. The fourth author is a member of the Hausdorff Center for Mathematics,
supported by the DFG Schwerpunktprogramm 1786 ``Homotopy Theory and
Algebraic Geometry'' (project ID SCHW 860/1-1). We thank Maxime Ramzi for helpful discussions regarding \cref{propn:LKEprodpres}.

\section{Spans and bispans}\label{sec:bg}
The goal of this section is to recall some basic properties of $\infty$-categories of spans and bispans, and to formulate conditions guaranteeing that a (bi)span $\infty$-category admits (co)products.

\subsection{Spans}\label{sec:spans}
Let us begin by recalling some basic definitions and results concerning $\infty$-{\hskip0pt}categories of spans. Our main references for this are Barwick's original article \cite{BarwickMackey} and the more recent treatment in
\cite{HHLN2}.

\begin{defn}
	A \emph{span pair} $(\cC, \cC_{F})$
	consists of an \icat{} $\cC$ together with a wide
	subcategory $\cC_{F}$ of ``forward'' maps, such that base
        changes of morphisms in $\cC_F$ exist in $\cC$ and are again contained in $\cC_F$. We write $\SPair$ for the \icat{} of span pairs; a morphism $(\cC, \cC_{F}) \to (\cD, \cD_{F})$ here is a functor $\cC \to \cD$ that preserves the forward maps  as well as pullbacks along forward maps.
\end{defn}

\begin{remark}
	Both \cite{BarwickMackey} and \cite{HHLN2} work more generally with so-called \emph{adequate triples} $(\cC,\cC_F,\cC_B)$. Span pairs correspond to the special case $\cC_B=\cC$.
\end{remark}

\begin{ex}
	For any \icat{} $\cC$, we always have the \emph{minimal} span pair $(\cC,
	\cC^{\simeq})$. If $\cC$ has all pullbacks, then we also have the \emph{maximal} span pair $(\cC, \cC)$.
\end{ex}

\cite{HHLN2}*{2.12} constructs a functor
\[ \Span \colon \SPair \to \CatI,\]
sending a span pair $(\cC, \cC_{F})$ to its \emph{span \icat{}}
\[\Span_F(\cC) := \Span(\cC, \cC_{F}).\] This $\infty$-category has
the same objects as $\cC$, and a map in $\Span_{F}(\cC)$ from $x$ to
$y$ is given by a \emph{span}
\[
\begin{tikzcd}
	{} & z \ar[dl,"b"{swap}] \ar[dr, "f"] \\
	x  & & y,
\end{tikzcd}
\]
where $f$ is in $\cC_{F}$ and $b$ is arbitrary; composition is
given by taking pullbacks in $\mathcal{C}$. If $\cC$ has all pullbacks, we abbreviate $\Span(\cC)$ for the span category associated to the span pair $(\cC, \cC)$.

\begin{ex}[\cite{HHLN2}*{2.15}]
	We have $\Span_{\eq}(\cC) = \Span(\cC, \cC^{\simeq}) \simeq \cC^{\op}$.
\end{ex}

\begin{remark}
  The \icat{} $\SPair$ has limits, which are computed in
	$\CatI$ \cite{HHLN2}*{2.4}, and the functor $\Span$ preserves these \cite{HHLN2}*{2.18}.
\end{remark}

We will need the following special case of Barwick's ``unfurling''
theorem:

\begin{propn}\label{unfurl radj}
  Suppose $(\cB, \cB_{F})$ is a span pair and $\Phi \colon \cB \to \CatI$ is a functor such that
  \begin{itemize}
  \item for every morphism $f \colon b \to b'$ in $\cB_{F}$, the functor $f_{!} = \Phi(f)$ has a right adjoint $f^{*}$,
  \item and for every pullback square
    \[
      \begin{tikzcd}
       a' \ar[r, "f'"]\drpullback \ar[d, "\alpha"'] & b'  \ar[d, "\beta"] \\
       a \ar[r, "f"'] & b
      \end{tikzcd}
    \]
    in $\cB$ with $f$ in $\cB_{F}$, the induced Beck--Chevalley transformation
    \[ \alpha_{!}f'^{*} \to f^{*}\beta_{!}\]
    is an equivalence.
  \end{itemize}
  Let $p \colon \cE \to \cB$ be the cocartesian fibration for $\Phi$, and write $\cE_{\fcart}$ for the subcategory containing the morphisms that are $p$-cartesian over morphisms in $\cB_{F}$. Then $(\cE, \cE_{\fcart})$ is a span pair, $p$ is a morphism of span pairs, and
  \[ \Span(p)^{\op} \colon \Span_{\fcart}(\cE)^{\op} \to
    \Span_{F}(\cB)^{\op}\] is the cocartesian fibration for a functor
  $\Span_{F}(\cB)^{\op} \to \CatI$ that restricts to $\Phi$ on $\cB$
  and to the functor obtained by passing to right adjoints from $\Phi$
  on $\cB_{F}^{\op}$.
\end{propn}

\begin{proof}
It follows from \cite{BarwickMackey}*{11.6} that $(\cE, \cE_{\fcart})$ is a span pair, that $p$ is a morphism of span pairs, and that $\Span(p)^{\op}$ is a cocartesian fibration. For the convenience of the reader we recall the proof, giving some additional details. It follows from \cite{BarwickMackey}*{11.2} that $(\cE, \cE_{\fcart})$ is a span pair, and the pullback
  \[
    \begin{tikzcd}
     w \ar[r, "\bar{f}'"] \ar[d, "\bar{g}'"']\drpullback & z  \ar[d, "\bar{g}"] \\
     x \ar[r, "\bar{f}"'] & y
    \end{tikzcd}
  \]
  of $\bar{f} \colon x \to y$ in $\cE_{\fcart}$ over $f \colon a \to b$ along a morphism $\bar{g} \colon z \to y$ over $g \colon c \to b$ is obtained by taking the pullback
  \[
    \begin{tikzcd}
     d \drpullback \ar[r, "f'"] \ar[d, "g'"'] & c  \ar[d, "g"] \\
     a \ar[r, "f"'] & b
    \end{tikzcd}
  \]
  in $\cB$, picking a $p$-cartesian morphism $\bar{f}' \colon w \to z$ over $f'$, and letting $\bar{g}'$ be the unique factorization of $\bar{g}\bar{f}'$ through the $p$-cartesian morphism $\bar{f}$.

  To show that $\Span(p)^{\op}$ is a cocartesian fibration, it suffices to show that any span
  \[ x \xfrom{f} y \xto{g} z \]
  in $\Span_{\fcart}(\cE)^{\op}$, where $f$ is $p$-cartesian over $\cB_{F}$ and $g$ is $p$-cocartesian, is a cocartesian morphism, since then $\Span_{\fcart}(\cE)^{\op}$ has all cocartesian lifts of morphisms in $\Span_{F}(\cB)^{\op}$. To see this we apply \cite{BarwickMackey}*{12.2}, in the guise of \cite{HHLN2}*{3.1}:
  \begin{itemize}
  \item Condition (1) is immediate since $p$ is a cocartesian fibration.
  \item Unwinding the definitions, condition (2) says that given a pullback square
    \[
      \begin{tikzcd}
       a\ar[r, "g'"] \ar[d, "f'"']\drpullback & b' \ar[d, "f"] \\
       b \ar[r, "g"'] & c
      \end{tikzcd}
    \]
    in $\cB$ with $f$ in $\cB_{F}$ and a commutative square
    \[
      \begin{tikzcd}
        f'^{*}x \ar[r, "\gamma"] \ar[d, "\bar{f}'"'] & y  \ar[d, "\phi"] \\
        x \ar[r, "\bar{g}"'] & g_{!}x
      \end{tikzcd}
    \]
    where $\bar{g}$ is $p$-cocartesian over $g$ and $\bar{f}'$ is
    $p$-cartesian over $f'$, then $\gamma$ is $p$-cocartesian \IFF{}
    $\phi$ lies in $\cE_{\fcart}$ and the square is a
    pullback. Indeed, in the former case $\phi$ factors as the
    canonical map $g'_{!}f'^{*}x \to f^{*}g_{!}x$ followed by a
    cartesian morphism over $f$, while in the latter case $\gamma$
    factors as a cocartesian morphism over $g'$ followed by the same
    map. Since this Beck--Chevalley map is by assumption invertible,
    the two conditions are equivalent.
  \end{itemize}
  It remains to identify the fibrations we get over $\cB$ and
  $\cB_{F}^{\op}$. Since the functor $\Span(\blank)$ is compatible
  with pullbacks, we see that over $\cB$ we recover
  $p \colon \cE \to \cB$, while over $\cB_{F}$ we get
  $\Span(\cE_{F}, \cE_{F,\fw}, \cE_{\fcart}) \to \cB_{F}^{\op}$, where
  $\cE_{F} \coloneqq \cE \times_{\cB} \cB_{F}$ and $\cE_{F,\fw}$ denotes the
  subcategory of morphisms that map to equivalences in
  $\cB_{F}^{\op}$. This is the cocartesian fibration that describes
  the same functor as the cartesian fibration $\cE_{F} \to \cB_{F}$,
  by \cite{BarwickGlasmanNardinCart}*{1.4} or \cite{HHLN2}*{3.18}.
\end{proof}

\subsection[{Products in span \icats{}}]{\boldmath Products in span \icats{}}\label{sec:prodinspan}

 In this subsection we provide criteria for \icats{} of spans to have products and coproducts.

 \begin{defn}
 	Recall that an \icat{} $\cC$ is called \emph{extensive} if $\cC$ has finite coproducts
 	and the coproduct functor
 	\[ \amalg \colon \prod_{i=1}^{n} \cC_{/x_{i}} \to \cC_{/\coprod_{i}
 		x_{i}}\] is an equivalence for all objects $x_1, \dots, x_n \in \cC$. A span pair
 	$(\cC, \cC_{F})$ is called \textit{extensive} if
  the following conditions are satisfied:
 	\begin{itemize}
    \item $\cC$ is extensive,
 		\item the morphisms in $\cC_F$ are closed under finite coproducts,
 		\item and for every $x \in \cC$, the maps $\emptyset \to x$ and $x \amalg x \to x$ are in $\cC_F$.
 	\end{itemize}
 	More generally, we say that $(\cC,\cC_F)$ is \textit{weakly extensive} if
 	\begin{itemize}
 		\item $\cC$ has finite coproducts,
 		\item the morphisms in $\cC_{F}$ are closed under
                  finite coproducts,
 		\item and the coproduct functor
 		\[ \amalg \colon \prod_{i=1}^{n} \cC_{/x_{i}}^{F} \to
 		\cC_{/\amalg_{i} x_{i}}^{F}\]
 		is an equivalence for all $n\ge0$ and $x_1, \dots, x_n \in \cC$. Here $\cC^F_{/y}$ denotes the full subcategory of $\cC_{/y}$ spanned by those maps $z\to y$ that belong to $F$.
 	\end{itemize}
 	If $\cC$ is an extensive $\infty$-category, then a wide subcategory $\cC_F \subseteq \cC$ is called a \textit{(weakly) extensive subcategory} if the pair $(\cC, \cC_F)$ is (weakly) extensive span pair.
 \end{defn}

\begin{remark}
	Note that a span pair $(\cC,\cC_F)$ is extensive if and only if $\cC$ and $\cC_F$ are both extensive \icats{} and the inclusion $\cC_F \hookrightarrow \cC$ preserves finite coproducts. Also note that every extensive span pair is weakly extensive.
\end{remark}

\begin{remark}
	\label{rmk:Reformulation_Extensive}
	Let $(\cC,\cC_F)$ be a span pair such that the morphisms in $\cC_F$ are closed under finite coproducts. Then the coproduct functor $\prod_{i=1}^{n} \cC^{F}_{/x_{i}} \to \cC^{F}_{/\mathop\amalg_{i} x_{i}}$ admits a right adjoint $\cC^{F}_{/\mathop\amalg_{i} x_{i}} \to \prod_{i=1}^{n} \cC^{F}_{/x_{i}}$ given by pullback along the maps $x_i \to \mathop\amalg_i x_i$, see \cite{HTT}*{5.2.5.1}, and it follows that $(\cC,\cC_F)$ is weakly extensive if and only if this functor is an equivalence. In this case, the canonical squares
	\[
	\begin{tikzcd}
		y_i \rar \dar[swap]{f_i} & \mathop\amalg_i y_i \dar{\mathop\amalg_i f_i} \\
		x_i \rar & \mathop\amalg_i x_i
	\end{tikzcd}
	\]
	are pullback squares for all morphisms $f_i\colon y_i \to x_i$ in $\cC_F$.

	One may similarly characterize extensiveness of $\cC$ by means of the right adjoint $\cC_{/\amalg_{i} x_{i}} \to \prod_{i=1}^{n} \cC_{/x_{i}}$ to the coproduct functor; in this case we need to assume that the morphism $\emptyset \to x$ is in $\cC_F$ for each object $x$ to guarantee that the relevant pullbacks exist.
\end{remark}

\begin{propn}
	\label{propn:Characterization_Extensive}
	\label{obs:Fext}
	A span pair $(\cC, \cC_F)$ is weakly extensive if and only if the following conditions hold:
	\begin{itemize}
		\item the \icat{} $\cC$ has finite coproducts,
		\item the coproduct functor $\cC \times \cC \to \cC$ is a morphism of span pairs (\ie{} morphisms in $\cC_{F}$ are closed
		under coproducts and coproducts of pullbacks of morphisms in
		$\cC_{F}$ are again pullbacks),
		\item and the commutative squares
		\[
		\begin{tikzcd}
			x \amalg x \ar[r] \ar[d] & y \amalg y \ar[d] \\
			x \ar[r] & y
		\end{tikzcd}
		\qquad
		\begin{tikzcd}
			\emptyset  \ar[r] \ar[d] & \emptyset  \ar[d] \\
			x \ar[r] & y
		\end{tikzcd}
		\]
		are pullbacks for all morphisms $x \to y$ in $\cC_{F}$.
	\end{itemize}
	The pair $(\cC, \cC_F)$ is extensive if and only if in addition we have:
	\begin{itemize}
		\item the above two squares are pullbacks for \emph{any} morphism $x \to y$ in $\cC$,
		\item the maps $\emptyset \to x$ and $x \amalg x \to x$ are in $\cC_F$ for all $x \in \cC$.
	\end{itemize}
\end{propn}
\begin{proof}
	First assume that $(\cC,\cC_F)$ is weakly extensive. By assumption, $\cC$ has finite coproducts and morphisms in $\cC_F$ are closed under finite coproducts. For the second condition, consider pullback squares
	\[
	\begin{tikzcd}
		x_i' \drpullback\rar{h_i} \dar[swap]{f_i'} & x_i \dar{f_i} \\
		y_i' \rar{g_i} & y_i
	\end{tikzcd}
	\]
	for $i = 1, \dots, n$, with $f_i \in \cC_F$. We need to show that their coproduct is again a pullback square, or, equivalently, that the map
	\[
		\mathop{\amalg}\nolimits_i x_i' \to (\mathop{\amalg}\nolimits_i x_i) \times_{{\amalg}_i y_i} (\mathop{\amalg}\nolimits_i y_i')
	\]
	is an equivalence in $\cC^F_{/\amalg_i y_i'}$. But since $(\cC,\cC_F)$ is weakly extensive, this may be checked after pulling back along each of the maps $y_i' \to \mathop\amalg_i y_i'$, where it becomes clear. For the third condition, we must show that the maps $x \amalg x \to (y \amalg y) \times_y x$ in $\cC^F_{/y \amalg y}$ and $\emptyset \to \emptyset \times_y x$ in $\cC^F_{/\emptyset}$ are equivalences. The latter is clear since $\cC^F_{/\emptyset} \isoto *$. For the former, it again suffices to check this after pulling back along the two inclusions $y \to y \amalg y$, where it is also clear.

	For the converse, assume that the first three conditions in the proposition are satisfied. We show that $(\cC,\cC_F)$ is a weakly extensive span pair. It suffices to prove
	that the pullback functor
	\[ \cC_{/{\coprod_{i=1}^{n}} x_{i}}^{F} \to \prod_{i=1}^{n} \cC_{/x_{i}}^{F} \]
	is an equivalence when $n = 0$ and $n = 2$. For $n = 0$ we want to show that $\cC_{/\emptyset}^{F} \simeq *$, which follows because for $x \to \emptyset$ in $\cC_{F}$ we have a pullback square
	\[
	\begin{tikzcd}
		\emptyset \drpullback \ar[r] \ar[d] & \emptyset  \ar[d] \\
		x \ar[r] & \emptyset,
	\end{tikzcd}
	\]
	so that $\emptyset \to x$ is an equivalence. For $n = 2$, the coproduct functor determines a left adjoint $\amalg\colon \cC^F_{/x_1} \times \cC^F_{/x_2} \to \cC^F_{/x_1 \amalg x_2}$ of the pullback functor by \cite{HTT}*{5.2.5.1}, and it will suffice that both the unit and the counit are equivalences. For the counit, consider a map $y \to x_{1} \amalg x_{2}$ in $\cC_{F}$, and let $y_{i} \to x_{i}$ be the pullback of $y$ along the inclusion of $x_{i}$ in the coproduct, which is again in $\cC_{F}$; we must show that the canonical map $y_{1} \amalg y_{2} \to y$ is an equivalence. To see this, consider the commutative diagram
	\[
	\begin{tikzcd}
		y_{1} \amalg y_{2} \ar[r] \ar[d] & y \amalg y \ar[r]\ar[d] & y \ar[d] \\
		x_{1} \amalg x_{2} \ar[r] & (x_{1}\amalg x_{2}) \amalg (x_{1} \amalg x_{2}) \ar[r] & x_{1} \amalg x_{2}.
	\end{tikzcd}
	\]
	Here the left square is cartesian since it's a coproduct of two
	pullback squares along morphisms in $\cC_{F}$, and the right square
	is cartesian since it's a square of fold maps for a morphism in
	$\cC_{F}$. The composite square is then cartesian, and the bottom
	horizontal composite is the identity, which implies that the top horizontal composite is indeed an equivalence.

	We now show that the unit of the adjunction is an equivalence. Given morphisms $y_1 \to x_1$ and $y_2 \to x_2$ in $\cC_F$, this amounts to showing that the canonical squares
	\[
	\begin{tikzcd}
		y_{i}  \ar[r] \ar[d] & y_{1} \amalg y_{2} \ar[d] \\
		x_{i} \ar[r] & x_{1} \amalg x_{2}
	\end{tikzcd}
	\]
	are pullback squares. Writing $x_1 = x_1 \amalg \emptyset$ and similarly for $x_2$, $y_1$ and $y_2$, these squares can be expressed as a coproduct of squares we know are pullbacks along morphisms in $\cC_{F}$, hence are pullback squares by assumption. This finishes the proof of the characterization of being weakly extensive.

	The proof for extensive span pairs is identical; the additional assumption that $\cC_F$ contains the maps $\emptyset \to x$ and $x \amalg x \to x$ is to ensure that all the relevant pullbacks that appear in the proof exist in $\cC$.
\end{proof}

We will now show that the extensiveness properties on a span pair imply good behavior of products and coproducts in the associated span $\infty$-category.

\begin{propn}[cf.\ {\cite{norms}*{C.3}}]\label{obs:spanprod}
	Suppose $(\cC, \cC_{F})$ is a span pair.
	\begin{enumerate}
		\item If $(\cC, \cC_{F})$ is weakly extensive, then the
		coproduct in $\cC$ gives a product in $\Span_{F}(\cC)$.
		\item If $(\cC, \cC_{F})$ is extensive, then the coproduct in
		$\cC$ is also a coproduct in $\Span_{F}(\cC)$. Moreover, the
		\icat{} $\Span_{F}(\cC)$ is semiadditive.
	\end{enumerate}
\end{propn}
For the last statement, recall that an $\infty$-category $\cD$ is called \textit{semiadditive} if it admits finite products and coproducts, the unique morphism $\emptyset \to *$ is an equivalence, and for all $x_1,x_2 \in \cD$, the morphism
\[
\begin{pmatrix}
	\id_{x_1} & 0 \\
	0 & \id_{x_2}
\end{pmatrix}\colon x_1 \amalg x_2 \to x_1 \times x_2
\]
is an equivalence, where $0$ denotes the unique map that factors through $*$.
\begin{proof}
	For the first part apply \cite{norms}*{C.21(2)} together with the characterization from \Cref{propn:Characterization_Extensive} to the adjunctions
	\[ \amalg : \cC \times \cC \rightleftarrows \cC : \Delta \qquad
	\text{and} \qquad \{\emptyset\}\colon * \rightleftarrows \cC :
	p.\] For the second part apply part (1) of the same corollary, to see that $\emptyset$ is also initial and $\Span(\amalg)$ is also \emph{left} adjoint to the restriction, so that $\Span_F(\cC)$ has finite coproducts. It is then clear that $\Span_F(\cC)$ is pointed. To see that it is semiadditive, we now observe that in any pointed \icat{} with finite (co)products the canonical comparison map $x\amalg y\to x\times y$ factors as
	\[
	x\amalg y\simeq (x\times0)\amalg(0\times y)\to (x\amalg 0)\times(0\amalg y)\simeq x\times y,
	\]
	so it is an equivalence in the case of $\Span_F(\cC)$ as the
        coproduct functor is a right adjoint by the above, and hence preserves products.
\end{proof}

\begin{remark}
  Our definition of ``extensive span
  pairs'' is closely related to Barwick's \emph{disjunctive triples}
  \cite{BarwickMackey}*{5.2}. Thus, \cref{obs:spanprod} is essentially
  a variant of the proof of semiadditivity in
  \cite{BarwickMackey}*{4.3 and 5.8}.
\end{remark}

\subsection{Bispans}\label{sec:bispans}
Finally, let us recall \icats{} of \emph{bispans} from
\cite{bispans}.

\begin{defn}\label{def:bispantrip}
  A \emph{bispan triple} $(\cC, \cC_{F}, \cC_{L})$ consists of an
  \icat{} $\cC$ together with two wide subcategories $\cC_{F},\cC_{L}\subseteq\cC$ such that both $(\cC,\cC_F)$ and $(\Span_F(\cC)^\op,\cC_L)$ are span pairs. In this case, we define
   \[ \Bispan_{F,L}(\cC) \coloneqq \Span_{L}(\Span_{F}(\cC)^{\op}).\]
  For $\cC_{L} = \cC$ we abbreviate this to $\Bispan_{F}(\cC)$, and if
  moreover also $\cC_{F} = \cC$, we will simply write $\Bispan(\cC)$.
\end{defn}

\begin{remark}
  By \cite{bispans}*{2.5.2(1)}, a triple $(\cC,\cC_F,\cC_L)$ is a bispan triple if and only if it satisfies the following more explicit conditions:
  \begin{enumerate}[(1)]
  \item Both $(\cC, \cC_{F})$ and $(\cC, \cC_{L})$ are span pairs.
  \item\label{it:bispanradj} Let $\cC_{/x}^{L} \subseteq \cC_{/x}$ again denote the full subcategory
    spanned by the maps to $x$ that lie in $\cC_{L}$. Then the functor
    $f^{*} \colon \cC_{/y}^{L} \to \cC_{/x}^{L}$ given by pullback along $f$ has a right adjoint
    $f_{*}$ for every map $f$ in $\cC_F$.
  \item\label{it:bispanmate} For any pullback square
        \[
      \begin{tikzcd}
        x' \ar[r, "f'"]\ar[d, "\xi"{swap}]\drpullback  & y' \ar[d, "\eta"] \\
        x \ar[r, "f"'] & y
      \end{tikzcd}
    \]
    with $f$ a map in $\cC_{F}$, the commutative square
    \[
      \begin{tikzcd}
        \cC^{L}_{/y} \ar[r, "f^{*}"] \ar[d, "\eta^{*}"'] & \cC^{L}_{/x}
        \ar[d, "\xi^{*}"] \\
        \cC^{L}_{/y'} \ar[r, "f'^{*}"'] & \cC^{L}_{/x'}
      \end{tikzcd}
    \]
    is \emph{right adjointable}, \ie{} the Beck--Chevalley transformation $\eta^{*}f_{*} \to
    f'_{*}\xi^{*}$ is invertible.
  \end{enumerate}

  Note that if $\cC_{L} = \cC$, then condition \ref{it:bispanradj} precisely says
   that $\cC$ is locally cartesian closed. In this case, condition
  \ref{it:bispanmate} is actually automatic as it can be checked after passing to left adjoints.
\end{remark}

\begin{defn}\label{def:bispanmor}
  Let $(\cC, \cC_{F}, \cC_{L})$ and $(\cD, \cD_{F}, \cD_{L})$ be bispan triples. A
  \emph{morphism of bispan triples} is a functor $\Phi \colon \cC \to \cD$ that induces morphisms of span pairs
  $(\cC, \cC_{F}) \to (\cD, \cD_{F})$, $(\cC, \cC_{L}) \to (\cD, \cD_{L})$, and
  \[(\Span_{F}(\cC)^{\op}, \cC_{L}) \to (\Span_{F}(\cD)^{\op},
    \cD_{L}).\]
\end{defn}

\begin{remark}
  In order to unpack the final condition, let us describe pullbacks in $\Span_F(\cC)^\op$ along morphisms in $\cC_L$ more concretely, for which it will be enough to describe pullbacks of backwards and forwards maps individually:
  \begin{itemize}
    \item Given a forward map $x \xleftarrow{=} x \xto{g} y$, its pullback along a map $l \colon z \to y$ in $\cC_{L}$ is given by
      \[
      \begin{tikzcd}
        d & d \arrow{r} \arrow[equals]{l} & z \\
        d \arrow[equals]{u} \arrow{d} & d \arrow{r} \arrow[equals]{l}
        \arrow[equals]{u} \arrow{d} \drpullback \urpullback \dlpullback & z \arrow[equals]{u} \arrow{d}{l} \\
        x  & x \arrow{r}{g} \arrow[equals]{l} & y\rlap,
      \end{tikzcd}
      \]
      see \cite{bispans}*{2.5.10}.
      \item Given a backwards map $x \xleftarrow{f} y \xto{=} y$ with $f$ in $\cC_{F}$,  \cite{bispans}*{2.5.12} shows that its pullback along $l \colon z \to y$ is of the form
      \[
      \begin{tikzcd}
        e & d \arrow{r}{\epsilon} \arrow{l} & z \\
        e \arrow[equals]{u} \arrow[swap]{d}{f_{*}l} & d \arrow{r}{\epsilon} \arrow{l}
        \arrow[equals]{u} \arrow{d}{f^{*}f_{*}l} \urpullback \dlpullback & z \arrow[equals]{u} \arrow{d}{l} \\
        x & y \arrow[equals]{r}\arrow{l}{f} & y\rlap,
      \end{tikzcd}
      \]
      where $\epsilon$ is the counit map $f^{*}f_{*}l \to l$.
  \end{itemize}
  In particular, we see that if $\Phi\colon\cC\to\cD$ is such that $(\cC,\cC_L)\to(\cD,\cD_L)$ and $(\cC,\cC_F)\to(\cD,\cD_F)$ are maps of span pairs, then $\Phi$ is a a map of bispan triples if and only if
  the Beck--Chevalley map
  \[ \Phi \circ f_{*} \to \Phi(f)_{*} \circ \Phi\]
  induced by the commutative square
  \[
    \begin{tikzcd}
      \cC^{L}_{/y} \ar[r, "f^{*}"] \ar[d, "\Phi"'] &[1em] \cC^{L}_{/x} \ar[d, "\Phi"] \\
      \cD^{L}_{/\Phi(y)} \ar[r, "\Phi(f)^{*}"'] & \cD^{L}_{/\Phi(x)}
    \end{tikzcd}
  \]
  is an equivalence.
\end{remark}

\begin{propn}\label{propn:bispanprod}
  Suppose $(\cC, \cC_{F}, \cC_{L})$ is a bispan triple such that both
  of the span pairs $(\hskip0pt minus .5pt\cC\hskip0pt minus .5pt,\hskip0pt minus 1pt \cC_{F})$ and $(\hskip0pt minus .5pt\cC\hskip0pt minus .5pt,\hskip0pt minus 1pt\cC_{L})$ are
  weakly extensive. Then the span pair $(\Span_{F}\hskip0pt minus .5pt(\hskip0pt minus .5pt\cC)\hskip0pt minus 1pt{}^{\op}\hskip0pt minus 1pt,\hskip0pt minus .5pt \cC_{L})$
  is also weakly extensive. In particular, $\Bispan_{F,L}(\cC)$ has finite products, and these are given by coproducts in $\cC$.
\end{propn}
\begin{proof}
  By \cref{obs:spanprod}(1), the coproduct in $\cC$ gives a product in
  $\Span_{F}(\cC)$ and hence a coproduct in $\Span_{F}(\cC)^{\op}$.

  We now claim that pullback squares along $\cC_{L}$ are closed under finite coproducts. Using the explicit description of pullbacks from the previous remark, the only non-obvious part of this is that, given $f_{i} \colon y_{i} \to x_{i}$ in $\cC_{F}$ and $l_{i} \colon z_{i} \to y_{i}$ in $\cC_L$ for $i = 1, 2$, we have
  \[ (f_{1} \amalg f_{2})_{*}(l_{1} \amalg l_{2}) \simeq f_{1,*}(l_{1}) \amalg f_{2,*}(l_{2}).\]
  To see this we use \cref{obs:Fext} with the argument from \cite{bispans}*{2.6.12}:
  Given $g \colon a \to x_{1} \amalg x_{2}$ in $\cC_{L}$, if $g_{i} \colon a_{i} \to x_{i}$ for $i = 1,2$ are the pullbacks along the summand inclusions, we get
  \[
  \begin{split}
    \Map_{\cC^{L}_{/x_{1} \amalg x_{2}}}(g, f_{1,*}(l_{1}) \amalg f_{2,*}(l_{2})) & \simeq \Map_{\cC^{L}_{/x_{1}}}(g_{1}, f_{1,*}(l_{1})) \times \Map_{\cC^{L}_{/x_{2}}}(g_{2}, f_{2,*}(l_{2})) \\
    & \simeq \Map_{\cC^{L}_{/y_{1}}}(f_{1}^{*}g_{1}, l_{1}) \times \Map_{\cC^{L}_{/y_{2}}}(f_{2}^{*}g_{2}, l_{2}) \\
    & \simeq \Map_{\cC^{L}_{/y_{1} \amalg y_{2}}}(f_{1}^{*}g_{1} \amalg f_{2}^{*}g_{2}, l_{1} \amalg l_{2}) \\
    & \simeq \Map_{\cC^{L}_{/y_{1} \amalg y_{2}}}((f_{1} \amalg f_{2})^{*}(g_{1} \amalg g_{2}), l_{1} \amalg l_{2}) \\
    & \simeq \Map_{\cC^{L}_{/x_{1} \amalg x_{2}}}(g, (f_{1} \amalg f_{2})_{*}(l_{1} \amalg l_{2})).
  \end{split}
  \]
  The remaining part of the conditions for a weakly extensive span pair hold
  because they by assumption hold for $(\cC, \cC_{L})$. With this established, the final statement is another instance of \Cref{obs:spanprod}(1).
\end{proof}

\section{\boldmath Normed $\infty$-categories}\label{sec:normed-icats}
We recall the definition of normed \icats{} and normed algebras from \cite{norms, BachmannMotivic} and give various examples of normed \icats{}.

\subsection[Normed monoids]{\boldmath Normed monoids}\label{sec:cmon}
Our starting point is the following generalization of the notion of a commutative monoid:

\begin{defn}
  \label{defn:tcommon}
  Let $\ctxt = (\cF, \cF_{\nrms})$ be a weakly extensive span pair and let $\cC$ be an \icat{}
  with finite products. An \emph{$\ctxt$-normed monoid in $\cC$} is a product-preserving functor
  \[ M \colon \SFF \to \cC.\]
  We denote its contravariant functoriality by $f^*\colon M(Y) \to M(X)$ for morphisms $f\colon X \to Y$ in $\cF$, and refer to these maps as \textit{restriction maps}. We denote its covariant functoriality by either $n_{\oplus}\colon M(X) \to M(Y)$ or $n_{\otimes}\colon M(X) \to M(Y)$ for morphisms $n\colon X \to Y$ in $\cF_{\nrms}$, and refer to these maps as \textit{(additive/multiplicative) norm maps}.
	We write \[ \NMon_{\ctxt}(\cC) \coloneqq \FunT(\SFF, \cC)\] for the full
	subcategory of $\Fun(\SFF, \cC)$ spanned by the
    $\ctxt$-normed monoids.
\end{defn}

\begin{observation}\label{obs:levelwise-monoid}
  If $\ctxt$ is actually extensive (and not only weakly so), then semiadditivity of $\SFF$ implies that all its objects carry unique commutative
  monoid structures, and so the values $M(X)$ of an $\ctxt$-normed
  monoid at $X \in \cF$ inherit commutative monoid structures in
  $\cC$. In fact, by \cite{GepnerGrothNikolaus}*{2.5} we get an
  equivalence
  \begin{equation}
    \label{eq:cmonincmon}
    \NMon_{\ctxt}(\cC) \simeq \NMon_{\ctxt}(\CMon(\cC))
 \end{equation}
 inverse to the forgetful functor.
\end{observation}

\begin{observation}
  If $\cC$ is presentable and $\cF$ is small, then $\NMon_{\ctxt}(\cC)$ is an accessible localization of $\Fun(\SFF, \cC)$, and so is a presentable \icat{}.
\end{observation}

Let us discuss various examples of normed monoids:

\begin{ex}\label{ex:g-mackey}
	Our main example of an extensive span pair is the pair
	$\ctxt = (\xF_{G}, \xF_{G})$ where $\xF_{G}$ is the category of finite
	$G$-sets for a finite group $G$. In this case, $\ctxt$-normed
    monoids in $\cC$ are also known as $\cC$-valued \textit{$G$-Mackey
    functors}:
    \[
    	\Mack_G(\cC) := \Fun^{\times}(\Span(\xF_G),\cC).
    \]
    More generally, we obtain a notion of \textit{incomplete $G$-Mackey functors} by taking $\ctxt = (\xF_G,I)$ for some weakly extensive subcategory $I \subseteq \xF_G$:
    \[
    \Mack_G^I(\cC) := \NMon_{(\xF_G,I)}(\cC) = \Fun^{\times}(\Span_I(\xF_G), \cC).
    \]
    These are most typically considered when $I \subseteq \xF_G$ is in fact an extensive subcategory of $\xF_G$ (and not only weakly extensive), in which case $I$ is usually called an \textit{indexing system for $G$} \cite[1.2 and 1.4]{blumberg-hill}.
\end{ex}

\begin{remark}\label{rk:mackey-functors}
  To see how our approach relates to classical equivariant infinite loop space theory, consider an indexing system $I\subseteq\xF_G$. By the discussion after
  \cite{rubin}*{3.9}, we can associate to $I$ a so-called
  \emph{$N_\infty$-operad} $\cO$ in $G$-spaces, and all
  $N_\infty$-operads arise this way; see also~\cite{gutierrez, bonventre}. The main result of \cite{marckey} connects space-valued Mackey-functors to $N_{\infty}$-algebras by showing that $\Mack_G^I(\Spc)$ is equivalent to the Dwyer--Kan localization of the $1$-category of $\cO$-algebras in $G$-spaces at a certain class of equivariant weak equivalences.
\end{remark}

\begin{ex}\label{ex:classical-cmon}
  Specializing example \cref{ex:g-mackey} to the trivial group, we obtain the extensive span pair $(\xF, \xF)$, where $\xF$ is the category of finite sets. By \cite{norms}*{C.1} there is an equivalence
  \[
  	\NMon_{\xF}(\cC) \simeq \CMon(\cC)
  \]
  between $\xF$-normed monoids in $\cC$ and commutative monoids in $\cC$, defined as functors $\xF_*\to\cC$ satisfying the Segal condition.
\end{ex}

\begin{ex}\label{ex:fold_maps}
	Let $\cF$ be an extensive $\infty$-category and let $\cF_{\mathrm{fold}}$ be the wide subcategory whose morphisms are finite coproducts of fold maps $\coprod_n x \to x$ for $x \in \cF$ and $n \ge 0$. Then the pair $(\cF, \cF_{\mathrm{fold}})$ is an extensive span pair, and \cite{norms}*{C.5} provides an equivalence
	\[
		\NMon_{(\cF,\cF_{\mathrm{fold}})}(\cC) \simeq \Fun^{\times}(\cF^{\op}, \CMon(\cC)).
	\]
\end{ex}

\begin{ex}\label{ex:extslice}
	Given a span pair $\ctxt = (\cF,\cF_{\nrms})$ and an object $x\in \cC$, we may consider the wide subcategory $\cF_{/x,\nrms} \coloneqq \cF_{/x} \times_{\cF} \cF_{\nrms}$ of the slice $\cF_{/x}$ consisting of those morphisms over $x$ that are contained in $\cF_{\nrms}$. We note that $\ctxt_{/x} := (\cF_{/x},\cF_{/x,\nrms})$ is again a span pair, which is (weakly) extensive if $(\cF,\cF_\nrms)$ is so. In particular we may speak of $\ctxt_{/x}$-normed monoids in $\cC$.
\end{ex}

\begin{ex}\label{ex:orbital}
  Let $T$ be any small \icat{}, and let $\xF[T]$ be the \icat{}
  obtained by freely adjoining finite coproducts to $T$, i.e.~$\xF[T]$
  is the full subcategory of the \icat{} of presheaves spanned by
  finite coproducts of representables. An \emph{orbital subcategory of
    $T$} \cite{CLL_Global}*{4.2.2} is a wide subcategory $P\subseteq
  T$ such that $(\xF[T],\xF[P])$ is a span pair. In this case,
  $(\xF[T],\xF[P])$ is always extensive: indeed, pullbacks in $\xF[T]$
  are also pullbacks in $\Fun(T^\op,\Spc)$ as $\xF[T]$ contains all
  representables, whence it suffices to check the compatibility axioms
  between coproducts and pullbacks in $\Spc$, which is
  straightforward.

  In particular, if $T$ is any \emph{orbital category} in the sense of \cite{NardinStab}*{4.1} (i.e.~$T$ is orbital as a subcategory of itself), then $(\xF[T],\xF[T])$ is extensive. Note that for $T=\textbf{O}_G$ the \emph{orbit category} of $G$ (i.e.~the $1$-category of finite transitive $G$-sets), we precisely recover Example~\ref{ex:g-mackey}.
\end{ex}

\begin{remark}
  If $\ctxt_{T}\coloneqq(\xF[T],\xF[T])$ is the extensive span pair arising from an
  orbital \icat{} $T$, then our definition of $\ctxt_{T}$-normed
  monoids fits into the framework for algebraic structures defined by
  Segal conditions from \cite{patterns1}: We can endow
  $\Span(\xF[T])$ with the structure of an \emph{algebraic pattern}
  where the inert--active factorization system is that given by the
  backwards and forwards maps, and the objects of $T$ are the
  elementary objects. Then a Segal $\Span(\xF[T])$-object in $\cC$ is
  a functor $M\colon \Span(\xF[T]) \to \cC$ such that
  \[ \textstyle M\big({\coprod_{i} t_{i}}\big) \isoto \prod_{i=1}^{n} M(t_{i})\]
  for all $t_{i}\in T$, which is equivalent to $M$ preserving finite products.
\end{remark}

\subsection[Normed \icats{} and normed algebras]{\boldmath Normed \icats{} and normed algebras}\label{sec:tsymmoncats}

In this subsection we fix a weakly extensive span pair $\ctxt = (\cF,
\cF_{\nrms})$. Specializing \cref{defn:tcommon} to $\CatI$ leads to the following definition:
\begin{defn}[Bachmann]
	\label{def:Normed_Category}
  An \emph{$\ctxt$-normed \icat{}} is an $\ctxt$-normed monoid in $\CatI$, i.e.\ a product-preserving
  functor $\cC\colon\Span_{\nrms}(\cF)\to\CatI$. We denote its contravariant functoriality by $f^*\colon \cC(y) \to \cC(x)$ for morphisms $f\colon x \to y$ in $\cF$, and denote its covariant functoriality by $f_{\otimes}\colon \cC(x) \to \cC(y)$ whenever $f$ is in $\cF_{\nrms}$.
\end{defn}

\begin{remark}
When $\ctxt$ is actually an extensive span pair, our definition of a normed $\infty$-category is identical to that of Bachmann \cite{BachmannMotivic}*{3.1}.
\end{remark}

\begin{remark}
	\label{rmk:Normed_Terminology}
	Since the product in $\SFF$ is the coproduct in $\cF$, a
	functor $\SFF \to \CatI$ is an $\ctxt$-normed \icat{} if and only if its restriction to $\cF^{\op}$
	preserves finite products. We will sometimes refer to product-preserving functors $\cF^{\op} \to \CatI$ as \emph{$\cF$-\icats{}}, and refer to the restriction of an $\ctxt$-normed \icat{} $\cC$ to $\cF^{\op}$ as the \emph{underlying $\cF$-\icat{}} of $\cC$. Similarly, we may sometimes refer to $\ctxt$-normed \icats{} as \textit{$\nrms$-normed $\cF$-\icats{}} whenever we wish to emphasize the collection of morphisms $\cF_{\nrms}$ along which we have norms.

	Note that for $\ctxt_{T} = (\xF[T], \xF[T])$, $T$ some orbital $\infty$-category, an $\xF[T]$-\icat{} is equivalently a functor $T^{\op} \to \CatI$ by the universal property of finite coproduct completion. This is the definition of a \emph{$T$-\icat} used e.g.~in \cite{exposeI,NardinStab,CLL_Global}.
\end{remark}

\begin{notation}
  Given an $\ctxt$-normed structure on an \icat{}
  $\cC$, we will denote the corresponding cocartesian and cartesian
  fibrations by 
  \[ \cC^{\otimes} \to \SFF, \quad \cC_{\otimes} \to
    \SFF^{\op}.\]
  We say that a morphism in $\cC^{\otimes}$ is \emph{inert} if
  it is cocartesian over a backwards morphism in $\SFF$;
  similarly, a morphism in $\cC_{\otimes}$ is \emph{inert} if it is
  cartesian over a (reversed) backwards morphism in
  $\SFF^{\op}$.
\end{notation}

\begin{defn}
  Suppose $\cC^{\otimes}, \cD^{\otimes} \to
  \SFF$ are $\ctxt$-normed \icats{}. An \emph{$\ctxt$-normed functor} from $\cC$ to $\cD$ is a
  commutative triangle
  \[
    \begin{tikzcd}
      \cC^{\otimes} \ar[rr, "\Phi"] \ar[dr] & &
      \cD^{\otimes} \ar[dl] \\
       & \SFF
    \end{tikzcd}
  \]
  where $\Phi$ preserves cocartesian morphisms. We say that $\Phi$ is
  \emph{lax $\ctxt$-normed} if it instead only preserves inert
  morphisms. We write
  \[
    \Funlax_{/\SFF}(\cC^{\otimes}, \cD^{\otimes})\subseteq\Fun_{/\SFF}(\cC^{\otimes}, \cD^{\otimes})
  \]
  for the full subcategory spanned by the lax $\ctxt$-normed functors.
\end{defn}

\begin{remark}
In the non-parametrized case, i.e.~the case $\ctxt = (\xF,\xF)$, it follows from \cite{envelopes}*{5.1.15}
  that this definition of lax symmetric monoidal functors agrees with
  the more standard one, with $\xF_{*}$ in place of $\Span(\xF)$. 
\end{remark}

\begin{defn}
  An \emph{$\ctxt$-normed algebra} in an $\ctxt$-normed \icat{}
  $\cC$ is a lax $\ctxt$-normed functor from
  $*^{\otimes} = \SFF$ to $\cC$; in other words, it is a section of the cocartesian fibration
  \[ \cC^{\otimes} \to \SFF\] that takes backward maps in $\SFF$ to
  cocartesian morphisms.  We write
  \[ \NAlg_{\ctxt}(\cC) \coloneqq \Funlax_{/\SFF}(\SFF,
    \cC^{\otimes}) \]
  for the \icat{} of $\ctxt$-normed algebras in $\cC$.
\end{defn}

\begin{remark}
  By an easy extension of \cite{envelopes}*{5.2.14}, our definitions
  of $\ctxt$-normed \icats{} and lax $\ctxt$-normed
  functors are equivalent to those of Nardin and
  Shah~\cite{NardinShah} in the case where $\ctxt = \ctxt_{T}$ for a
  so-called ``atomic'' orbital \icat{} $T$. In particular, the
  $\infty$-categories of $\ctxt$-normed algebras are equivalent, cf.~\cite{envelopes}*{5.3.17}. For extensive $\ctxt$, our $\ctxt$-normed algebras are also studied in \cite{BachmannMotivic}, as a generalization of the normed spectra introduced in \cite{norms}*{7.1}.
\end{remark}

We end this section by considering a construction of normed structures on spans:
\begin{construction}\label{const:sym_mon_spans}
  Since the functor $\Span \colon \SPair \to \CatI$ preserves limits,
  hence in particular finite products, any $\ctxt$-normed monoid in
  $\SPair$ gives rise to an $\ctxt$-normed \icat{} by applying $\Span$ pointwise. Observe that an $\ctxt$-normed monoid in $\SPair$ is
  an $\ctxt$-normed $\infty$-category
  \[
    \cC\colon \SFF \to \CatI
  \]
  equipped with a subfunctor $\cC_{Q} \subseteq \cC$ such
  that $(\cC(X),\cC_Q(X))$ is a span pair for every $X \in \cF$ and the induced
  functor $m_{\otimes}f^*\colon \cC(X) \to \cC(Y)$ is a map of span pairs for
  every morphism $X \xleftarrow{\raisebox{-.3ex}[0pt][0pt]{\ensuremath{\scriptstyle f}}} Z \xrightarrow{\raisebox{-.3ex}[0pt][0pt]{\ensuremath{\scriptstyle m}}} Y$ in $\Span_{\nrms}(\cF)$. In this case, the composite
  \[
    \Span_{Q}(\cC)\coloneqq \Span \circ (\cC,\cC_Q) \colon \SFF \to \CatI
  \]
  endows $\Span_{Q}(\cC)$ with an $\ctxt$-normed structure inherited from that of $\cC$.
\end{construction}

The following result provides an explicit description of the cocartesian fibrations associated to such normed structures:
\begin{propn}\label{propn:spansmon}
  Let $p \colon \cC_{\otimes} \to \SFF^{\op}$ be a
  cartesian fibration corresponding to an $\ctxt$-normed monoid
  in $\SPair$. Then the cocartesian fibration
  $\Span_{Q}(\cC)^{\otimes} \to \SFF$ for the induced
  $\ctxt$-normed structure on spans from
  \cref{const:sym_mon_spans} is given by
  \[ \Span_{Q}(\cC)^{\otimes} \simeq
    \Span_{Q\dfw}(\cC_{\otimes}),\]
  where $(\cC_{\otimes})_{Q\dfw}$ denotes the subcategory
  of maps that go to equivalences under $p$ and fiberwise lie in $\cC^{\otimes}(\blank)_{Q}$.
\end{propn}
\begin{proof}
	This is a special case of \cite{HHLN2}*{3.9}.
\end{proof}

\subsection{Norms on product-preserving functors}
In this subsection we will construct a (low-tech) version of
``parametrized Day convolution'' for \icats{} of product-preserving
functors. More precisely, we will show the following:
\begin{propn}\label{propn:FunTsymmon}
  Let $\cX$ be a cocomplete \icat{} with finite products, where the
  product functor preserves colimits in each variable. Suppose
  $\ctxt = (\cF, \cF_{\nrms})$ is a weakly extensive span pair, and consider an
  $\ctxt$-normed \icat{} $\cC \colon \SFF \to \CatI$
  such that $\cC(X)$ has finite products for every $X \in \cF$ (but the
  morphisms in the diagram need not preserve these).
  \begin{enumerate}[(i)]
  \item There is a functor
    \[ \cQ = \FunT(\cC(\blank), \cX) \colon \SFF \to \LCatI \]
    obtained by left Kan extension from $\cC$. This preserves finite
    products, and so defines another $\ctxt$-normed \icat{}.
\item If $\cC^{\otimes} \to \SFF$ is the cocartesian fibration for $\cC$, then $\ctxt$-normed algebras in $\cQ = \FunT(\cC(\blank), \cX)$ are equivalent to functors
  \[ A \colon \cC^{\otimes} \to \cX\]
  such that
  \begin{itemize}
  \item for every $X \in \cF$ the restriction
    \[ A_{X} \colon \cC(X) \to \cX\]
    of $A$ to the fiber over $X$ is a product-preserving functor,
  \item and for every morphism $f \colon X \to Y$ in $\cF$, viewed as a backward morphism in $\SFF$, the natural transformation
      \[
    \begin{tikzcd}
      \cC(Y) \ar[rr, "f^{*}"{above}] \ar[dr, "A_{Y}"', ""{above,name=A, inner sep=4pt}] & & \cC(X) \ar[dl, "A_{X}", ""{above,name=B, inner sep=4pt}] \\
      & \cX
      \ar[from=A, to=B, Rightarrow]
    \end{tikzcd}
  \]
  exhibits $A_{X}$ as a left Kan extension of $A_{Y}$ along $f^{*}$.
  \end{itemize}
  More precisely, $\NAlg_{\ctxt}(\cQ)$ is equivalent to the full
  subcategory $\cal A\subseteq\Fun(\cC^{\otimes}, \cX)$ of functors that satisfy
  these conditions, and for every $A\in\cF$ this equivalence fits into
  a commutative diagram
  \begin{equation*}
    \begin{tikzcd}[column sep=small]
      \NAlg_\ctxt(\cQ)\arrow[dr, bend right=10pt, "\ev_A"']\arrow[rr,"\simeq"] && \cal A\rlap.\arrow[dl, bend left=10pt, "\textup{res}"]\\
      & \Fun^\times(\cC(A),\cX)
    \end{tikzcd}
  \end{equation*}
  \end{enumerate}
\end{propn}

The key input to the construction is the following observation about
left Kan extensions of product-preserving functors:
\begin{propn}\label{propn:LKEprodpres}
  Suppose $\cA$ and $\cB$ are small \icats{} with finite products, and
  $\cC$ is an \icat{} with small colimits and finite products such
  that the cartesian product preserves colimits in each variable. If
  $F \colon \cA \to \cC$ is a product-preserving functor and
  $g \colon \cA \to \cB$ is an arbitrary functor, then the left Kan
  extension $g_{!}F$ also preserves finite products. In other words, left Kan extension restricts to a functor
  \[ g_{!} \colon \Fun^{\times}(\cA, \cC) \to \Fun^{\times}(\cB, \cC).\]
\end{propn}

\begin{remark}
  For $1$-categories, a version of this result apparently goes back to
  Lawvere's thesis \cite{LawvereThesis}. See also for instance
  \cite{DayThesis}*{Appendix 2} or \cite{BorceuxDay} for
  generalizations to enriched categories and \cite{StreetProduct} for
  another variant and a historical discussion.
\end{remark}

\begin{proof}[Proof of \cref{propn:LKEprodpres}]
  Our assumptions guarantee that $g_{!}F$ is computed by the pointwise formula,
  \[ g_{!}F(b) \simeq \colim_{\cA_{/b}} F.\]
  In particular, $g_{!}F(*)$ is a colimit over $\cA \times_{\cB} \cB_{/*} \simeq \cA$; since this has a terminal object $*\to*$, we see \[g_{!}F(*) \simeq F(*) \simeq *.\]
  For objects $b_{1},b_{2} \in \cB$, consider the functor
  \[ \pi_{b_{1},b_{2}} \colon \cA_{/b_{1} \times b_{2}} \to \cA_{/b_{1}} \times \cA_{/b_{2}}\]
  given by composition with the projections $b_{1} \times b_{2} \to b_{i}$. We claim that this functor has a right adjoint $R = R_{b_{1},b_{2}}$, given on a pair $(\Phi_1,\Phi_2)$ of objects $\Phi_{i} \coloneqq (a_{i}, \phi_{i}\colon g(a_{i}) \to b_{i})$ in $\cA_{/b_i}$ by $R(\Phi_1,\Phi_2) \,=\,(a_{1} \times a_{2}, r(\phi_1,\phi_2))$, where $r(\phi_1,\phi_2)$ is defined as the composite
  \[
  	g(a_{1} \times a_{2}) \rightarrow g(a_{1}) \times g(a_{2}) \xrightarrow{\phi_{1} \times \phi_{2}} b_{1} \times b_{2}.
  \]
  To see this, observe that for an object $\Psi = (x, \psi\colon g(x) \to b_{1} \times b_{2})$ of $\cA_{/b_{1} \times b_{2}}$ the mapping space $\Map_{\cA_{/b_1 \times b_2}}(\Psi,R(\Phi_1,\Phi_2))$ sits in a pullback diagram as follows:
	\[
	\hspace{-10.70pt}\hfuzz=10.70pt
	\begin{tikzcd}[column sep = small]
		{\Map_{\cA_{/b_1 \times b_2}}(\Psi,R(\Phi_1,\Phi_2))} \drpullback &[-.5em] {\Map_{\cB_{/b_1 \times b_2}}(\psi,r(\phi_1,\phi_2))} \drpullback
		&[2em] {\{\psi\}} \\
		{\Map_{\cA}(x,a_1\times a_2)} & {\Map_{\cB}(g(x),g(a_1\times a_2))}
		& {\Map_{\cB}(g(x),b_1 \times b_2)}\rlap.
		\arrow[from=1-1, to=1-2]
		\arrow[from=1-1, to=2-1]
		\arrow[from=1-2, to=1-3]
		\arrow[from=1-2, to=2-2]
		\arrow[from=1-3, to=2-3]
		\arrow[from=2-1, to=2-2, "g"]
		\arrow[from=2-2, to=2-3, "{r(\phi_1,\phi_2) \circ -}"]
	\end{tikzcd}\]
	Under the identification of $\Map_{\cA}(x,a_1\times a_2)$ with $\Map_{\cA}(x,a_1) \times \Map_{\cA}(x,a_2)$ and of $\Map_{\cB}(g(x),b_1 \times b_2)$ with $\Map_{\cB}(g(x),b_1) \times \Map_{\cB}(g(x),b_2)$, the bottom map turns into a product of the two maps
	\[
		\Map_{\cA}(x,a_i) \xto{g} \Map_{\cB}(g(x),g(a_i)) \xto{\phi_i \circ -} \Map_{\cB}(g(x),b_i),
	\]
	and so by passing to fibers we obtain a natural equivalence
	\[
	\Map_{\cA_{/b_1 \times b_2}}(\Psi,R(\Phi_1,\Phi_2)) \isoto \Map_{\cA_{/b_1}}((x,\pr_1 \psi),\Phi_1) \times \Map_{\cA_{/b_2}}((x,\pr_2 \psi),\Phi_2).
	\]
	Since the target is canonically identified with $\smash{\Map_{\cA_{/b_1} \times \cA_{/b_2}}(\pi_{b_1,b_2}\Psi,(\Phi_1,\Phi_2))}$, this shows that $R_{b_1,b_2}$ is the desired right adjoint.

  Since right adjoints are cofinal, composition with $R_{b_{1},b_{2}}$ therefore induces an equivalence
  \[\hskip-14.2pt\hfuzz=14.2pt g_{!}F(b_{1}) \times g_{!}F(b_{2}) \simeq \colim_{(a,a') \in \cA_{/b_{1}} \times \cA_{/b_{2}}} F(a \times a') \to \colim_{x \in \cA_{/b_{1} \times b_{2}}} F(x) \simeq g_{!}F(b_{1} \times b_{2}).\]
  Moreover, these right adjoints are compatible with composition in $\cB$, so for maps $b_{1} \to c_{1}, b_{2} \to c_{2}$ we get a commutative square
  \[
    \begin{tikzcd}
      g_{!}F(b_{1}) \times g_{!}F(b_{2}) \ar[r, "\sim"] \ar[d] & g_{!}F(b_{1} \times b_{2}) \ar[d] \\
      g_{!}F(c_{1}) \times g_{!}F(c_{2}) \ar[r, "\sim"] & g_{!}F(c_{1} \times c_{2}).
    \end{tikzcd}
  \]
  Taking $c_{1} = b_{1}$ and $c_{2} = *$, we see in particular that
  projection to $g_{!}F(b_{1})$ on the left corresponds to composition
  with $b_{1} \times b_{2} \to b_{1}$ on the right, so that the
  canonical map
  $g_{!}F(b_{1} \times b_{2}) \to g_{!}F(b_{1}) \times g_{!}F(b_{2})$
  is an equivalence. In other words, the functor $g_{!}F$ is
  product-preserving, as required.
\end{proof}

\begin{lemma}\label{lem:prodpresprod}
  Suppose $\cA_{1},\ldots, \cA_{n}$ and $\cB$ are \icats{} with finite products. If $\cA \coloneqq \prod_{i} \cA_{i}$, then left Kan extension along the projections $\pi_{i} \colon \cA \to \cA_{i}$ gives an equivalence
  \[ \FunT(\cA, \cB) \isoto \prod_{i} \FunT(\cA_{i}, \cB),\]
  with inverse given by
  \[ \left(F_{i} \colon \cA_{i} \to \cB\right) \,\, \mapsto \,\, \left(\prod_{i}F_{i}\circ \pi_{i} \colon \cA \to \cB \right)\]
\end{lemma}
\begin{proof}
  For $i = 0$ we indeed have $\FunT(*, \cB) \simeq *$ as the only
  product-preserving functor is the one constant at the terminal
  object. Suppose therefore that $i > 1$.  The pointwise left Kan
  extension of $F \colon \cA \to \cB$ along $\pi_{i}$, if it exists,
  is given at $x \in \cA_{i}$ by taking a colimit over
  \[ \cA_{/x} \simeq \cA_{i/x} \times \prod_{j \neq i} \cA_{j}\]
  This has a terminal object, so the colimit (and hence the pointwise Kan extension) always exists, and is given by
  \[ (\pi_{i,!}F)(x) \simeq F(*,\dots,*,x,*,\dots,*).\]
  The functor we claim is an equivalence is the composite
  \[ \FunT(\cA, \cB) \to \prod_{i} \FunT(\cA, \cB) \xto{\prod_{i} \pi_{i,!}} \prod_{i} \FunT(\cA_{i}, \cB).\]
  Since $\FunT(\cA, \cB)$ has finite products (computed pointwise), this functor has a right adjoint, given by
  \[ \prod_{i} \FunT(\cA_{i}, \cB) \xto{\prod \pi_{i}^{*}} \prod_{i} \FunT(\cA, \cB) \xto{\times} \FunT(\cA, \cB).\]
  To see that this adjunction is in fact an equivalence, it suffices to observe that for $F_{i} \in \FunT(\cA_{i},\cB)$ we have
  \[ \Big(\pi_{j,!}\Big(\prod_{i} F_{i}\circ \pi_{i}\Big)\Big)(x) \simeq F_{j}(x) \times \prod_{i \neq j} F_{i}(*) \simeq F_{j}(x)\]
  and that for $F \in \FunT(\cA,\cB)$ we have
  \[ F(x_{1},\ldots,x_{n}) \isoto \prod_{i} (\pi_{i,!}F)(x_{i}),\]
  since $(x_{1},\ldots,x_{n})$ is the finite product
  \[ (x_{1},*,\dots*) \times (*,x_{2},*,\dots,*) \times \cdots (*,\dots,*,x_{n}) \]
  in $\cA$.
\end{proof}

\begin{remark}
  Let $\mathcal R$ be any collection of diagram shapes containing both the empty set and the two-point set. Then the same argument shows that the categories of $\mathcal R$-shaped limit preserving functors satisfy $\Fun^{\mathcal R\text{-lim}}(\prod_{i=1}^n\cA_i,\cB)\simeq\prod_{i=1}^n\Fun^{\mathcal R\text{-lim}}(\cA_i,\cB)$.
\end{remark}

We now come to our main construction:

\begin{construction}\label{constr:funprodfib}
  Let $\cX$ be a cocomplete \icat{} with finite products, such that
  the cartesian product preserves colimits in each variable, and let
  $F \colon \cI \to \CatI$ be a functor such that $F(i)$ has finite
  products for all $i \in \cI$ (but these are not necessarily
  preserved by the morphisms in the diagram).

  Let $p \colon \cE \to \cI$ be the cartesian fibration for the functor
  \[ \Fun(F(\blank), \cX) \colon \cI^{\op} \to \LCatI,\]
  and note that by \cite{freefib}*{7.3} there is a natural equivalence
  \begin{equation}
    \label{eq:fibfuntoconst}
    \Fun_{/\cI}(\cK, \cE) \simeq \Fun(\cK \times_{\cI} \cF, \cX),
\end{equation}
  where $\cF \to \cI$ is the cocartesian fibration for $F$. Here
  $p$ is also a cocartesian fibration, since we can left Kan extend functors to $\cX$. Moreover, if we define $\cE'$
  as the full subcategory containing the functors $F(i) \to \cX$ that
  preserve products for all $i$, then \cref{propn:LKEprodpres} implies
  that $\cE' \to \cI$ is again a cocartesian fibration. Note that for $f \colon i \to j$ in $\cI$, a morphism $\phi$ in $\cE$ over $f$ corresponds under the equivalence \cref{eq:fibfuntoconst} to a functor $[1] \times_{\cI} \cF \to \cX$. Here the source is the cocartesian fibration over $[1]$ encoding the functor $F(f) \colon F(i) \to F(j)$ and so can be described as the pushout $F(i) \times [1] \amalg_{F(i) \times \{1\}} F(j)$, see \cite{freefib}*{3.1}. We can thus identify the morphism $\phi$ with a natural transformation
  \[
    \begin{tikzcd}
      F(i) \ar[rr, "F(f)"{above}] \ar[dr,  ""{above,name=A, inner sep=4pt}] & & F(j) \ar[dl,  ""{above,name=B, inner sep=4pt}] \\
      & \cX\rlap,
      \ar[from=A, to=B, Rightarrow]
    \end{tikzcd}
  \]
  and $\phi$ is a cocartesian morphism if and only if this diagram exhibits $F(j) \to \cX$ as a left Kan extension of $F(i) \to \cX$ along $F(f)$.
\end{construction}

\begin{proof}[Proof of \cref{propn:FunTsymmon}]
  To prove that $\cQ$ is $\ctxt$-normed we must show that given a finite coproduct $X \simeq \coprod_{i} X_{i}$ in $\cF$, with $\iota_{j} \colon X_{j} \to X$ the summand inclusions, the functor
  \[ (\pi_{j,!})_{j} \colon \FunT(\cC(X), \cX) \to \prod_{j} \FunT(\cC(X_{j}), \cX),\]
  where $\pi_{j} \coloneqq \cC(\iota_{j})$,
  is an equivalence. This is the content of \cref{lem:prodpresprod}.

  Part (ii) follows immediately from \cref{constr:funprodfib} specialized to $\cI=\Span_{\nrms}(\cF)$: note that the straightening of the cocartesian fibration $\cE\to\cI$ dicussed there agrees by definition with the functor $X\mapsto\Fun(\cC(X),\cX)$ with functoriality via left Kan extension, so that the cocartesian subfibration $\cE'\to\cI$ classifies the functor $\FunT(\cC(-),\cX)$ in question.
\end{proof}

\begin{observation}\label{obs:radjsimplify}
  In the situation above, suppose the functor
  $f^{*} \colon \cC(Y) \to \cC(X)$ has a right adjoint $f_{*}$ for every
  backwards map $f$. Then the condition for $A\colon \cC^{\otimes} \to \cX$ to define an
  $\ctxt$-normed algebra in $\cQ$ can be rephrased as requiring an
  equivalence
  \[ A_{X} \simeq A_{Y} \circ f_{*}.\] In this case $\cC^{\otimes}$
  also has \emph{cartesian} morphisms over backwards maps, and we can
  phrase this condition more precisely as: If $\bar{X}$ is in $\cC^{\otimes}_{X}$ and
  $\phi \colon \bar{Y} \to \bar{X}$ is cartesian over a backwards map in
  $\SFF$, then $A(\phi)$ is an equivalence.
\end{observation}

In the special case $\ctxt = (\xF,\xF)$, the resulting normed structure on $\Fun^{\times}(\cC,\cX)$ corresponds by \Cref{ex:classical-cmon} to a symmetric monoidal structure. We will now compare it to the Day convolution monoidal structure:

\begin{propn}\label{prop:dayconvolution}
  Let $\cX$ be a cocomplete \icat{} with finite products, where the
  product functor preserves colimits in each variable. Suppose
  $\cC \colon \Span(\xF) \to \CatI$ is a symmetric monoidal \icat{} whose underlying \icat{} has finite products.
  \begin{itemize}
  	\item The symmetric
  	monoidal structure on $\FunT(\cC, \cX)$ from \cref{propn:FunTsymmon}
  	is a full symmetric monoidal subcategory of the Day convolution on
  	$\Fun(\cC, \cX)$.
  	\item If $\cX$ is presentable and the tensor product on $\cC$ preserves finite products in each variable, it is moreover a
  	symmetric monoidal localization.
  \end{itemize}
\end{propn}

The proof will require some preparations.

\begin{lemma}
  Let $\cC,\cD\to\cI$ be cocartesian fibrations, and let $F\colon\cC\to\cD$ be a functor over $\cI$. Then the following are equivalent:
  \begin{enumerate}[(i)]
    \item $F$ preserves cocartesian edges.
    \item For every cocartesian edge $[1]\to\cC$ the composite $[1]\to\cD$ is the relative left Kan extension (over $\cI$) of its restriction to $0$.
    \item For every $i\in\cI$ and every $\cI_{i/}\to\cC$ over $\cI$ landing in the subcategory of cocartesian edges, the composite $\cI_{i/}\to\cD$ is relatively left Kan extended from $\id_{i}\in\cI_{i/}$.
   \end{enumerate}
   \begin{proof}
    Recall that if $\cJ\to\cI$ is arbitrary and $\cJ$ has an initial object $\emptyset$, then the relative left Kan extension along $\{\emptyset\}\hookrightarrow\cJ$ exists for every cocartesian fibration $\mathcal E\to\cI$, and $\cJ\to\mathcal E$ is relatively left Kan extended if and only if it factors through cocartesian edges \cite[\href{https://kerodon.net/tag/043G}{Tag \textsc{043g}}]{kerodon}. The equivalence between (1) and (2) follows immediately, while for the equivalence between (1) and (3) it suffices to observe in addition that every cocartesian edge $x\to y$ of $\cC$ is contained in the image of some cocartesian $\cI_{i/}\to\cC$: namely, if $i$ is the image of $x$, then the relative left Kan extension of $x$ along $\{\emptyset\}\hookrightarrow\cI_{i/}$ has the required properties.
   \end{proof}
\end{lemma}

\begin{propn}\label{propn:technical}
  Let $\cC\colon\Span(\xF)\to\CatI$ be a symmetric monoidal \icat{}, let $\cX$ be a cocomplete category with finite products such that the product preserves colimits in each variable, and let $\cE'\to\Span(\xF)$ denote the cocartesian fibration classifying the functor $\FunT(\cC(-),\cX)$ with functoriality via left Kan extension. 

  If $\cO^\otimes\to\Span(\xF)$ is any symmetric monoidal \icat{}, then $\cC^\otimes\times_{\Span(\xF)}\cO^\otimes$ has finite products, and a functor $\cO^\otimes\to\cE'$ over $\Span(\xF)$ is lax symmetric monoidal if and only if the functor $\tilde F\colon\cC^\otimes\times_{\Span(\xF)}\cO^\otimes\to\cX$ associated to $F$ via (\ref{eq:fibfuntoconst}) preserves finite products.
\end{propn}
\begin{proof}
  We first recall from \cite{anmnd2}*{2.2.6} that $\cC^\otimes$ and $\cO^\otimes$ have finite products and that the maps $\cC^\otimes\to\Span(\xF)$ and $\cO^\otimes\to\Span(\xF)$ preserve them; thus, $\cC^\otimes\times_{\Span(\xF)}\cO^\otimes$ again has finite products, which are computed componentwise. The cited reference moreover shows that any $X\in\cO_{\textbf n}$ is the product of its cocartesian pushforwards along the backwards maps $\textbf n\gets\bfone=\bfone$; thus, we see that a functor $\cC^\otimes\times_{\Span(\xF)}\cO^\otimes\to\cX$ preserves products if and only if its restriction to $\cC^\otimes\times_{\Span(\xF)}{(\xF_{/\textbf n})}^\op$ does so for every map ${(\xF_{/\textbf n})}^\op\to\cO^\otimes$ over $\Span(\xF)$ landing in cocartesian edges.

  Write now $\cE\to\Span(\xF)$ for the cocartesian fibration from Construction~\ref{constr:funprodfib} classifying $\Fun(\cC(-),\cX)$, of which $\cE'\to\Span(\xF)$ is a subfibration. We then have for every $\textbf{n}\in\xF$ and $(\xF_{/\textbf n})^\op\to\cO^\otimes$ over $\Span(\xF)$ a commutative diagram
  \begin{equation}\label{diag:naturality-cE-repr}
    \begin{tikzcd}
      \Fun_{/\Span(\xF)}\big(\cO^\otimes,\cE\big)\arrow[r,"\sim"]\arrow[d] &
      \Fun\big(\cC^\otimes\times_{\Span(\xF)}\cO^\otimes,\cX\big)\arrow[d]\\
      \Fun_{/\Span(\xF)}\big((\xF_{/\textbf n})^\op, \cE\big)\arrow[r,"\sim"]\arrow[d] &
      \Fun\big(\cC^\otimes\times_{\Span(\xF)}(\xF_{/\textbf n})^\op,\cX\big)\arrow[d]\\
      \Fun\big(\{O\},\Fun(\cC_{\textbf n},\cX)\big)\arrow[r,"\sim"] &
      \Fun\big(\vrule width0pt height0pt depth11pt\smash{\underbrace{\smash{\cC^\otimes\times_{\Span(\xF)}\{O\}}}_{{}=\cC_{\textbf n}}},\cX\big)
    \end{tikzcd}
  \end{equation}
  where the horizontal equivalences are as in Construction~\ref{constr:funprodfib} and the vertical maps are the restrictions. By the previous lemma applied to $\mathcal I=\xF^\op$, $F\colon\cO^\otimes\to\cE$ preserves inert edges if and only if restriction to $(\xF_{/\textbf n})^\op$ is contained in the image of the left adjoint of the lower left vertical map for every $(\xF_{/\textbf n})^\op\to\cO^\otimes$ factoring through cocartesian edges. It follows formally from commutativity of (\ref{diag:naturality-cE-repr}) that this is equivalent to the restriction of the corresponding functor $\tilde F$ to $\cC^\otimes\times_{\Span(\xF)}(\xF_{/\textbf n})^\op$ being left Kan extended from $\cC_{\textbf n}$; it remains to show that if $F$ factors through $\cE'\subseteq\mathcal E$ (i.e.~if the restriction of $\tilde F$ to $\cC^\otimes\times_{\Span(\xF)}\{P\}$ preserves products for every $P\in\cO^\otimes$), then the latter condition is in turn equivalent to the restriction of $\tilde F$ to $\cC^\otimes\times_{\Span(\xF)}{(\xF_{/\textbf n})}^\op\to\cX$ preserving products.

  Proposition~\ref{propn:LKEprodpres} shows that the left Kan extension of any product-preserving functor $\cC_{\textbf n}\to\cX$ to $\cC^\otimes\times_{\Span(\xF)}(\xF_{/\textbf n})^\op$ is again product-preserving, so the above condition is indeed sufficient for $\tilde F$ to preserve products. To see that it is also necessary, it will suffice to show that any product-preserving functor $G\colon\cC^\otimes\times_{\Span(\xF)}(\xF_{/\textbf n})^\op\to\cX$ {whose restriction to $\cC_{\textbf{\textup n}}$ again preserves products} is left Kan extended from its restriction to $\cC_{\textbf n}$.

  If we let $j\colon\cC_{\textbf n}\hookrightarrow\cC^\otimes\times_{\Span(\xF)}(\xF_{/\textbf n})^\op$ denote the inclusion, then the counit $j_!j^*G\to G$ is an equivalence for any $(X,\id_{n})\in\cC^\otimes\times_{\Span(\xF)}(\xF_{/\textbf n})^\op$ by full faithfulness of $j$. We claim that it is in fact also an equivalence for every $(X,i\colon{\bfone\to\textbf n})$; with this established, the claim will follow as both $G$ (by assumption) and $j_!j^*G$ (by the above) preserve products and every object of $\cC^\otimes\times_{\Span(\xF)}(\xF_{/\textbf n})^\op$ decomposes as a finite product of objects $(X,\bfone\to\textbf n)$.

  To prove the claim, note that for any $X_1,\dots,X_n\in\cC_{\bfone}$
  \[
    \cC_{\textbf n}\ni(X_1,\dots,X_n;\id_{\textbf n})\simeq\prod_{k=1}^n(X_k, k\colon\bfone\to\textbf n);
  \]
  thus, the product of the counits $j_!j^*G(X_k, k\colon\bfone\to\textbf n)\to G(X_k, k\colon\bfone\to\textbf n)$ is an equivalence as $G$ and $j_!j^*G$ preserve products. Specializing to $X_k=*$ for $k\not=i$, it will therefore suffice that $G(*,k\colon\bfone\to\textbf n)\simeq*\simeq j_!j^*G(*,k\colon\bfone\to\textbf n)$ for every $1\le k\le n$. For this, we further set $X_i=*$ to see that 
  \[
    \prod_{k=1}^n G(*,k\colon\bfone\to\textbf n)\simeq G(*,\id_{\textbf n})\simeq j_!j^*G(*,\id_{\textbf n})\simeq\prod_{k=1}^n j_!j^*G(*,k\colon\bfone\to\textbf n).
  \]
  As the restriction of $G$ to $\cC_{\textbf n}$ preserves finite products, $G(*,\id_{\textbf n})$ is terminal; since a product is terminal if and only if all of its factors are, this completes the proof of the claim and hence of the proposition.
\end{proof}

\begin{proof}[Proof of Proposition~\ref{prop:dayconvolution}]
  Recall first that the Day convolution of $F,G\colon\cC\to\cX$ is given by the left Kan extension of
  \begin{equation}\label{eq:external-product}
    \cC\times\cC\xrightarrow{F\times G}\cX\times\cX\xrightarrow{\prod}\cX
  \end{equation}
  along $\otimes\colon\cC\times\cC\to\cC$ \cite{HA}*{2.2.6.17}. If $F$ and $G$ preserve products, so does (\ref{eq:external-product}), whence so does the Day convolution by Proposition~\ref{propn:LKEprodpres}. On the other hand, the unit is given by the left Kan extension of $*$ along $\{\bbone\}\hookrightarrow\cC$, and the same argument shows that this is again product preserving, i.e.~the Day convolution structure indeed restricts to $\Fun^\times(\cC,\cX)$. Moreover, we saw in \cite{CHLL1}*{3.3.4} that this is also a symmetric monoidal localization
  when $\cX$ is presentable and the tensor product on $\cC$ preserves products in each variable.

  It remains to compare this symmetric monoidal structure to the one above, for which we will show that both represent the same functor in the \icat{} of symmetric monoidal \icats{} and \emph{lax} symmetric monoidal functors.

  For this we first recall that Day convolution is defined in such a way that lax symmetric monoidal functors $\cO^\otimes\to\Fun(\cC,\cX)^\otimes_\textup{Day}$ correspond bijectively to lax symmetric monoidal functors $\cC^\otimes\times_{\Span(\xF)}\cO^\otimes\to\cX^\times$, which are in turn identified with product-preserving functors $\cC^\otimes\times_{\Span(\xF)}\cO^\otimes\to\cX$ \cite{CHLL1}*{2.4.5 and 2.4.6}. Thus, functors into the restricted Day convolution on $\Fun^{\times}(\cC,\cX)$ correspond bijectively to product-preserving functors $\cC^\otimes\times_{\Span(\xF)}\cO^\otimes\to\cX$ such that in addition the restriction $\cC_\textbf{n}\simeq\cC^\otimes\times_{\Span(\xF)}\{O\}\to\cX$ preserves products for all $n\ge0$, $O\in\cO_\textbf{n}$.

  On the other hand, we have seen in Construction~\ref{constr:funprodfib} that
  functors $F\colon\cO^\otimes\to\Fun^\times(\cC,\cX)^\otimes$ over $\Span(\xF)$ correspond to functors
  \begin{equation}\label{eq:adjointed}
    \tilde F\colon\cC^\otimes\times_{\Span(\xF)}\cO^\otimes\to\cX
  \end{equation}
  that preserve products when restricted to each $\cC^\otimes\times_{\Span(\xF)}\{O\}$, and Proposition~\ref{propn:technical} shows that such a functor $F$ is indeed lax symmetric monoidal if and only if $\tilde F$ preserves finite products.
\end{proof}

\subsection[The cartesian normed structure on $\cF$]{\boldmath The cartesian normed structure on $\cF$}

Consider an extensive \icat{} $\cF$ with pullbacks, taken to be fixed throughout this subsection.
We will see that we may equip $\cF$ with a ``cartesian'' normed structure
whenever $\cF$ is suitably locally cartesian closed.

\begin{notation}\label{nota:f-ul}
  We define a parametrized version $\ulF$ of $\cF$ as the functor
  \[\ulF\colon \cF^{\op}\rightarrow \Cat_\infty,\quad  X \mapsto \cF_{/X},\] with functoriality coming from
  pullbacks. Since $\cF$ is extensive, this functor preserves products, and hence defines an $\cF$-\icat{}.
\end{notation}

\begin{defn}
	\label{def:Locally_Cartesian_Closed}
  Given a weakly extensive subcategory $\cF_{\nrms} \subseteq \cF$, we say that $\cF$ is
  \emph{$\nrms$-locally cartesian closed} if the functor
  $f^{*} \colon \cF_{/Y} \to \cF_{/X}$ given by pullback along
  $f \colon X \to Y$ in $\cF_{\nrms}$ has a right adjoint $f_{*}$.
\end{defn}

\begin{propn}\label{propn:cFtimes}
  Let $\cF_{\nrms} \subseteq \cF$ be a weakly extensive subcategory and suppose $\cF$ is $\nrms$-locally cartesian
  closed. Then the following hold:
  \begin{itemize}
  	\item The pair $(\Ar(\cF), \Ar(\cF)_{\nrms\dpb})$ is a span pair, where
  	$\Ar(\cF)_{\nrms\dpb}$ consists of the pullback squares along morphisms
  	in $\cF_{\nrms}$.
  	\item The functor $\ev_{1} \colon \Ar(\cF) \to \cF$ is a morphism of
  	span pairs.
  	\item The functor
  	\[ \Span(\ev_{1})^{\op} \colon \Span_{\nrms\dpb}(\Ar(\cF))^{\op} \to \SFF^{\op} \]
  	is the cartesian fibration for an $\nrms$-normed structure on $\ulF$.
  \end{itemize}
\end{propn}
\begin{proof}
  Consider the cocartesian fibration \[\ev_{1} \colon \Ar(\cF) \to \cF\] classified by the functor
  $\cF_{/(\blank)} \colon \cF \to \CatI$, with functoriality given by
  composition. By assumption we have right adjoints (given by
  pullback) for morphisms in $\cF_{\nrms}$, and by unpacking the
  definitions and applying the pasting lemma for pullbacks we see that
  these satisfy base change. Applying \cref{unfurl radj} to this situation, we obtain the first two bullet points, and we get that $\Span(\ev_{1})^{\op}$ is a cocartesian fibration. To see that it is also a cartesian fibration, it suffices by \cite{HTT}*{5.2.2.4${}^\op$} to show that it is a locally cartesian fibration, which we can check separately over $\cF$ and $\cF_{\nrms}^{\op}$. Over $\cF_{\nrms}$ we get the cocartesian fibration for the functor $\cF_{/(\blank)} \colon \cF_{\nrms}^{\op} \to \CatI$, with functoriality given by pullback; since these pullback functors have right adjoints due to $\nrms$-locally cartesian closedness of $\cF$, it is also a cartesian fibration over $\cF_{\nrms}^{\op}$. On the other hand, over $\cF$ we get the functor
  $\ev_{1} \colon \Ar(\cF) \to \cF$. Since this is the cartesian fibration for $\ulF$, we indeed get the cartesian fibration for an $\ctxt$-normed structure on $\ulF$.
\end{proof}

\begin{notation}
  In the context of \cref{propn:cFtimes}, we write
  \[\ulF_{\times} \coloneqq \Span_{\nrms\dpb}(\Ar(\cF))^{\op}\] and refer to it as
  the \emph{cartesian} normed structure on $\ulF$. In
  the non-parametrized case, this construction indeed gives the
  cartesian fibration for the cartesian symmetric monoidal structure
  on $\xF$ by \cite{CHLL1}*{3.1.4}. We expect that our construction
  more generally agrees with \cite{NardinShah}*{2.4.1} whenever the
  two frameworks overlap; however, as this won't be relevant for the
  purposes of this paper, we will not prove this here.
\end{notation}

\section{Normed rings}
\label{sec:commring}
In this section, we introduce the notion of a \textit{normed ring} and show it may equivalently be encoded as a \emph{space-valued Tambara functor}.

\subsection[Normed semirings]{\boldmath Normed semirings}

We want to consider notions of normed semirings where we have
two potentially different families of (``additive'' and
``multiplicative'') norms, generalizing the addition and multiplication operations that exist in an ordinary semiring. To capture such structures, we introduce the following definition:
\begin{defn}
  \label{defn:rigctxt}
  A \emph{semiring context} $\ctxt = (\cF, \cF_{M}, \cF_{A})$ consists of
  an extensive \icat{} $\cF$ together with two weakly extensive subcategories $\cF_{M}$ and
  $\cF_{A}$ such that:
  \begin{enumerate}
  \item $\cF$ has pullbacks.
  \item For $m \colon X \to Y$ in $\cF_{M}$, the pullback functor
    $m^{*} \colon \cF_{/Y} \to \cF_{/X}$ has a right adjoint $m_{*}\colon \cF_{/X} \to \cF_{/Y}$ which preserves morphisms whose image in $\cF$ lies in $\cF_A$.
  \end{enumerate}
  We write $\ctxt_{M} \coloneqq (\cF, \cF_{M})$ and $\ctxt_{A} \coloneqq (\cF,\cF_{A})$ for the resulting two weakly extensive span pairs.
\end{defn}

\begin{observation}
	\label{obs:rigctxt_bispantriple}
  Every semiring context is a bispan triple:
  \begin{itemize}
    \item The Beck--Chevalley condition for the functors $m_*$ is automatically satisfied, since it may be checked after passing to left adjoints.
    \item For $m\colon X \to Y$ in $\cF_M$, the functor $m_*\colon \cF_{/X} \to \cF_{/Y}$ preserves terminal objects, hence it sends $\cF_{/X}^A$ into $\cF_{/Y}^A$.
  \end{itemize}
\end{observation}

\begin{ex}\label{ex:trivialsemictx}
  When $\cF_A = \cF^{\simeq}$, the triple $(\cF, \cF_{M}, \cF^{\simeq})$ is
  a semiring context if and only if $\cF$ is extensive, admits pullbacks, and is $M$-locally cartesian
  closed, in the sense of \Cref{def:Locally_Cartesian_Closed}.
\end{ex}

\begin{ex}
  Let $\cF$ be an extensive $\infty$-category that is locally
  cartesian closed. Then the triple $(\cF, \cF,\cF)$ is a semiring
  context.
\end{ex}

\begin{ex}\label{ex:F_G-rigctxt}
  For a finite group $G$, the category $\xF_G$ of finite $G$-sets is
  extensive and locally cartesian closed, so that the triple
  $(\xF_G, \xF_G,\xF_G)$ is a semiring context. More generally we obtain a semiring context $(\xF_G,I,\xF_G)$ for every weakly extensive subcategory $I \subseteq \xF_G$.
\end{ex}

We fix a semiring context $\ctxt = (\cF, \cF_{M}, \cF_{A})$. Our goal in
the rest of this subsection is to construct the $\infty$-category $\NRig_{\ctxt}(\cX)$ of \textit{$\ctxt$-normed semirings in $\cX$} for suitable choices of $\infty$-categories $\cX$. We start by constructing a certain $\ctxt_M$-normed $\infty$-category $\Span_{A}(\ulF)$ of spans in $\cF$.

\begin{construction}
  \label{cons:SpanA}
  As $\cF$ admits pullbacks, the evaluation map $\ev_1\colon \Ar(\cF) \to \cF$
  is a cartesian fibration, classifying the functor $\ul{\cF}\colon \cF^{\op} \to \CatI$ from Notation~\ref{nota:f-ul}. Since
  morphisms in $\cF_{A}$ are closed under base change, we obtain a functor
  \[
    \cF^{\op} \to \SPair, \qquad X \mapsto (\cF_{/X}, \cF_{/X,A})
  \]
  where $(\cF_{/X},\cF_{/X,A})$ is the span pair from \cref{ex:extslice}; we define the
  $\cF$-$\infty$-category $\Span_{A}(\ul{\cF})$ by composing with the
  functor $\Span\colon \SPair \to \CatI$. Since limits in $\SPair$ are computed in $\CatI$ and $\Span(\blank)$ preserves limits, this is indeed an $\cF$-\icat{}.
\end{construction}

\begin{lemma}
  The $\ctxt_M$-normed \icat{} $\ulF$ from \cref{propn:cFtimes}
  induces, via \cref{const:sym_mon_spans}, an $\ctxt_M$-normed
  structure on $\Span_{A}(\ulF)$.
\end{lemma}

\begin{proof}
  Equipping each $\cF_{/X}$ with the span pair structure
  $(\cF_{/X}, \cF_{/X,A})$ from \cref{cons:SpanA}, the functors
  $m_*f^*\colon \cF_{/X} \to \cF_{/Y}$ are morphisms of span pairs by
  our assumptions on $\ctxt$, hence we obtain an $\ctxt_M$-normed
  $\infty$-category $\Span_A(\ul{\cF})^{\otimes}$ using
  \cref{const:sym_mon_spans}.
\end{proof}

\begin{defn}
	\label{def:monoidal_CMon}
Let $\cX$ be a cocomplete \icat{} with finite products, such that
the cartesian product preserves colimits in each variable. We define
$\uNMon_{\ctxt_{A}}(\cX)$ to be the $\ctxt_M$-normed \icat{} $\FunT(\Span_{A}(\ulF), \cX)$ obtained by applying
\cref{propn:FunTsymmon} to the $\ctxt_M$-normed structure
on $\Span_{A}(\ulF)$ from \cref{cons:SpanA}.
\end{defn}

Note that the value of $\uNMon_{\ctxt_{A}}\hskip0pt minus 1.25pt(\cX\hskip0pt minus .5pt)$ at $X\hskip0pt minus .25pt\in\hskip0pt minus .25pt \cF$ is the $\infty$-category of $(\cF_{/X}\hskip0pt minus .25pt,\hskip0pt minus .25pt \cF_{/X,A})$-normed monoids in $\cX$.

\begin{ex}\label{ex:CMonisGGN}
  Combining \cref{prop:dayconvolution} with \cite{CHLL1}*{3.3.5}, we
  see that when $\cX$ is a cartesian closed presentable \icat{}, then
  the symmetric monoidal structure on $\NMon_{\xF}(\cX)\simeq \CMon(\cX)$ from
  \cref{def:monoidal_CMon} agrees with the ``standard'' one constructed in
  \cite{GepnerGrothNikolaus}.
\end{ex}

\begin{remark}\label{rk:extend-wrong-way}
Let $f\colon X\to Y$ be any map in $\cF$. By \cite{norms}*{C.21(2)}, the adjunction $f_!\colon\cF_{/X}\rightleftarrows\cF_{/Y}:\!f^*$ induces a ``wrong way'' adjunction \[f^*\colon\Span_A(\cF_{/Y})\rightleftarrows\Span_A(\cF_{/X}) :\!f_!.\]
The underlying $\cF$-\icat{} of $\ul\NMon_{\ctxt_A}(\cX)$ therefore admits the following alternative description: it is given by the composite
\[
\cF^\op\xrightarrow{\Span_A(\cF_{/\blank})}(\CatI^\times)^\op\xrightarrow{\Fun^\times(\blank,\cX)}\CatI,
\]
i.e.~its functoriality is given via restriction along pushforwards.
\end{remark}

\begin{remark}
If $\cX$ is presentable and $(\cF,\cF_A)=(\xF[T],\xF[P])$ for a small
\icat{} $T$ and a left-cancellable orbital subcategory $P\subseteq T$
consisting of truncated maps, then the $\cF$-\icat{}
$\ul\NMon_{\ctxt_A}(\cX)$ is studied in \cite{CLL_Spans}*{\S9.2} under the
name $\ul\Mack^{P}_T(\cX)$. Corollary~9.9 of said article establishes
a universal property for this $\cF$-\icat{}, and shows that whenever $P$ is a
so-called \emph{atomic} orbital subcategory it agrees with the
$\cF$-\icat{} $\ul{\CMon}^P_T(\ul\cX_T)$ of
\cite{CLL_Global}*{\S4.8}. In particular, if in addition $P=T$, this
further agrees with Nardin's $\ul{\CMon}_T(\ul{\cX}_T)$
\cite{NardinStab}*{4.9}.
\end{remark}

\begin{defn}
  Let $\cX$ be as in Definition~\ref{def:monoidal_CMon}.
An \emph{$\ctxt$-normed semiring} in $\cX$ is an $\ctxt_M$-normed
algebra in $\uNMon_{\ctxt_{A}}(\cX)$; we write
\[ \NRig_{\ctxt}(\cX) \coloneqq \NAlg_{\ctxt_M}\big(\uNMon_{\ctxt_{A}}(\cX)\big).\]
\end{defn}

\begin{ex}\label{ex:ggn}
  Let $\ctxt=(\xF,\xF,\xF)$. Combining Examples~\ref{ex:classical-cmon} and~\ref{ex:CMonisGGN}, we see that $\NRig_{\ctxt}(\cX)$ agrees with the \icat{} ${\mathcal R}\kern-1pt\textup{ig}_{\mathbb E_\infty}(\cX)$ of $\mathbb E_\infty$-semirings considered by Gepner, Groth, and Nikolaus \cite{GepnerGrothNikolaus}*{7.1}.
\end{ex}

\subsection[The Lawvere theory of normed semirings]{\boldmath The Lawvere theory of normed semirings}
Let $\ctxt = (\cF, \cF_{M}, \cF_{A})$ be a semiring context. Since $\ctxt$ is in particular a bispan triple by \Cref{obs:rigctxt_bispantriple}, we may form its bispan category $\Bispan_{M,A}(\cF)$. Our goal in this subsection is to show that the $\infty$-category $\Bispan_{M,A}(\cF)$ is the \textit{Lawvere theory for $\ctxt$-normed semirings}: for any $\infty$-category $\cX$ satisfying the conditions from \Cref{def:monoidal_CMon}, the $\infty$-category of $\ctxt$-normed semi\-rings in $\cX$ is equivalent to the $\infty$-category of product-preserving functors $\Bispan_{M,A}(\cF) \to \cX$. This in particular allows us to think of an $\ctxt$-normed semiring $R$ in $\cX$ as a family of objects $R(X)$ for all $X \in \cF$ equipped with restrictions $f^*\colon R(Y) \to R(X)$, additive norms $a_{\oplus}\colon R(X) \to R(Y)$, and multiplicative norms $m_{\otimes}\colon R(X) \to R(Y)$, which satisfy various compatibility relations exhibited by the respective composition laws in $\Bispan_{M,A}(\cF)$.

We start with some preliminary statements.

\begin{propn}\label{propn:spanFotimes}
  The $\ctxt_M$-normed structure on $\Span_{A}(\ulF)$ induced
  by the cartesian $\ctxt_M$-normed structure on $\ulF$ is
  given by
  \[ \Span_{A}(\ulF)^{\otimes} \simeq
    \Bispan_{M\dpb,A\dfw}(\Ar(\cF)) =
    \Span_{A\dfw}(\Span_{M\dpb}(\Ar(\cF))^{\op}) \]
  where $\Ar(\cF)_{M\dpb}$ denotes the wide subcategory of $\Ar(\cF)$ whose morphisms are pullback squares over $\cF_{M}$, and
  $\Span_{M\dpb}(\Ar(\cF))_{A\dfw}$ denotes the subcategory of
  morphisms whose image under $\ev_{1}$ is an equivalence in $\Span_{M}(\cF)$ and whose forward part lives over $\cF_{A}$.
\end{propn}
\begin{proof}
  This follows by combining \cref{propn:spansmon} with the
  description of $\ulF_{\times}$ from \cref{propn:cFtimes}.
\end{proof}

Explicitly this means a morphism in $\Span_{A}(\ulF)^{\otimes}$ is represented by a diagram
\[
\begin{tikzcd}
	X \arrow{d} & Y \arrow{l} \arrow{r} \arrow{d} \drpullback & X' \arrow{d} \arrow{r}{a} & X'' \arrow{d} \\
	S & T \arrow{l} \arrow{r}[swap]{m} & S'\arrow[r,equals] & S'
\end{tikzcd}
\]
with $m$ in $\cF_{M}$ and $a$ in $\cF_{A}$; the cocartesian fibration
to $\Span_M(\cF)$ is given by restricting to the bottom row.

\begin{ex}\label{ex:triv2}
	For the semiring context $\ctxt = (\cF, \cF_{M},
	\cF^{\simeq})$ from \cref{ex:trivialsemictx}, we
	obtain the $\ctxt_M$-normed structure on $\ulF^{\op}$ given by
	\[ \Span_{M\dpb}(\Ar(\cF))^{\op} \simeq (\ulF_{\times})^{\op}.\]
	Moreover, the underlying $\cF$-\icat{} of $\uNMon_{T_{A}}(\cX)$ is the $\cF$-\icat{}
  \[
    \ul{\cX}_{\cF}\coloneqq\Fun^{\times}(\ulF^{\op}, \cX)=\Fun^\times(-,\cX)\circ\ulF^\op
  \]
  of $\cF$-objects in $\cX$.
\end{ex}

\begin{cor}\label{cor:crigbispanarf}
  The \icat{} $\NRig_{\ctxt}(\cX)$ of $\ctxt$-normed semirings in $\cX$ is naturally equivalent to the
  full subcategory $\cR\subseteq\Fun(\Bispan_{M\dpb,A\dfw}(\Ar(\cF)), \cX)$ spanned by functors
  \[ \Phi \colon \Bispan_{M\dpb,A\dfw}(\Ar(\cF)) \to \cX\]
  such that
  \begin{enumerate}[(1)]
  \item For every object $E \to X$ in
    $\Bispan_{M\dpb,A\dfw}(\Ar(\cF))$, where $E$ decomposes as a
    coproduct $\coprod_{i=1}^{n} E_{i}$ in $\cF$,
    evaluating $\Phi$ at the morphisms
    \begin{equation}
      \label{eq:srigprod}
      \begin{tikzcd}
        E \ar[d] & E_{i} \ar[l] \ar[d] \ar[r,equals] \drpullback & E_i \ar[r, equals] \ar[d] & E_i \ar[d] \\
        X & X \ar[l, equals] \ar[r,equals] & X \ar[r,equals] & X
      \end{tikzcd}
    \end{equation}
    gives an equivalence
    \[ \Phi(E \to X) \isoto \prod_{i=1}^{n} \Phi(E_{i} \to X).\]
  \item $\Phi$ takes morphisms of the form
    \begin{equation}
      \label{eq:sriginvert}
           \begin{tikzcd}
     	E \ar[d] & E \ar[l,equals] \ar[d] \ar[r,equals] \drpullback & E \ar[r, equals] \ar[d] & E \ar[d] \\
     	Y & X \ar[l] \ar[r,equals] & X \ar[r,equals] & X
     \end{tikzcd}
    \end{equation}
    in $\Bispan_{M\dpb,A\dfw}(\Ar(\cF))$ to equivalences in $\cX$.
  \end{enumerate}
  Moreover, for every $X\in\cF$ this equivalence fits into a commutative diagram
  \begin{equation}\label{diag:prelawvere-pw}
    \begin{tikzcd}[column sep=small]
      \NRig_\ctxt(\cX)\arrow[dr, bend right=15pt, "\ev_A"']\arrow[rr,"\simeq"] &&[1.5em] \cR\arrow[dl, bend left=15pt, "\iota_X^*"]\\
      & \Fun^\times(\Span_A(\cF_{/X}),\cX)
    \end{tikzcd}
  \end{equation}
  where $\iota_X$ is induced by $(\cF_{/X},(\cF_{/X})^\simeq,\cF_{/X,A})\hookrightarrow (\Ar(\cF),\Ar(\cF)_{M\dpb},\Ar(\cF)_{A\dfw})$.
\end{cor}
\begin{proof}
  In light of \Cref{propn:spanFotimes}, we only have to show that the description of the full subcategory $\cR$ is equivalent to the description given in \cref{propn:FunTsymmon}.

  Since products in $\Span_{A}(\cF_{/X})$ are given by coproducts in $\cF$, the first condition amounts to asking for
  the restriction
  \[ \Phi_{X} \colon \Span_{A}(\cF_{/X}) \to \cX\]
  of $\Phi$ to preserve products for every $X \in \cF$, which is the first condition formulated in \cref{propn:FunTsymmon}.

  We now claim that (\ref{eq:sriginvert}) defines a cartesian edge over $Y\gets X=X$; \cref{obs:radjsimplify} will then immediately show that our second condition is indeed equivalent to the second condition of \cref{propn:FunTsymmon}. For this we observe that restricting in the target to $\cF^\op=\Span_{\eq}(\cF)$ recovers the map $\Span_{A\dfw}(\Ar(\cF))\to\cF^\op$ classifying $\Span_A(\ulF)$. The subfibration $\Ar(\cF)^\op\to\cF^\op$ is both cartesian and cocartesian, with cartesian edges given by the squares in question (the \emph{co}cartesian edges of $\Ar(\cF)\to\cF$). We therefore want to show that this is still cartesian in $\Span_{A\dfw}(\Ar(\cF))$. However, this simply means that the adjunction $f^*\colon (\cF_{/X})^\op\rightleftarrows(\cF_{/Y})^\op :\! f_!$ ought to extend to $\Span_A(\cF_{/X})\rightleftarrows\Span_A(\cF_{/Y})$, which was observed in Remark~\ref{rk:extend-wrong-way} above.
\end{proof}

\begin{thm}\label{thm:cmon-bispan-model}
  Composition with $\ev_{0}\hskip0pt minus .5pt \colon\hskip0pt minus 1pt \Bispan_{M\dpb,\,A\dfw}\hskip 0pt minus .75pt(\hskip 0pt minus .5pt\Ar(\cF)\hskip 0pt minus .25pt)\hskip0pt minus 1.75pt \to\hskip0pt minus 1.5pt \Bispan_{M,A}\hskip 0pt minus .75pt(\hskip 0pt minus .25pt\cF)$ induces an equivalence
  \[ \NRig_{\ctxt}(\cX) \isoto \FunT(\Bispan_{M,A}(\cF), \cX)\]
  fitting into commutative diagrams
  \begin{equation}\label{diag:lawvere-equ-pw}
    \begin{tikzcd}[column sep=small]
      \NRig_{\ctxt}(\cX)\arrow[rr,"\simeq"]\arrow[dr, bend right=15pt, "\ev_{X}"'] && \FunT(\Bispan_{M,A}(\cF), \cX)\arrow[dl,bend left=15pt, "p_X^*"]\\
      & \FunT(\Span_A(\cF_{/X}),\cX)
    \end{tikzcd}
  \end{equation}
  where $p_X$ is induced by the forgetful map $(\cF_{/X},(\cF_{/X})^\simeq,\cF_{/X,A})\to(\cF,\cF_M,\cF_A)$.
\end{thm}

\begin{proof}
  By \cite{CHLL1}*{4.2.2}, the functor $\ev_{0}$ on bispans is a localization. Let $W$ be the class of morphisms it takes to equivalences, which we can immediately simplify to those of the form
  \[
    \begin{tikzcd}
      X \arrow{d} & X \arrow[equals]{l} \arrow[equals]{r} \arrow{d} \drpullback & X \arrow{d} \arrow[r, equals] & X\arrow[d]\\
      S & T \arrow{l} \arrow{r} & S' \arrow[r, equals] & S'\llap,
    \end{tikzcd}
  \]
  and let $S \subseteq W$ be the morphisms of the form
  \cref{eq:sriginvert}. Arguing as in the proof of
  \cite{CHLL1}*{4.3.1}, we see that a functor that inverts $S$ must
  invert all of $W$.  From \cref{cor:crigbispanarf} we know that an
  $\ctxt$-normed semiring inverts the morphisms in $S$, and it
  therefore factors through the localization. All that remains for the equivalence $\NRig_\ctxt(\cX)\simeq\FunT(\Bispan_{M,A}(\cF),\cX)$ is to show that a functor $\Phi \colon \Bispan_{M,A}(\cF) \to \cX$ preserves products \IFF{} the composite $\Psi \coloneqq \Phi \circ \ev_{0}$ satisfies
  \[ \Psi(E \to X) \isoto \prod_{i}\Psi(E_{i} \to X)\]
  for $E \simeq \coprod_{i} E_{i}$.
  Since the product in $\Bispan_{M,A}(\cF)$ is given by the coproduct in $\cF$ by \cref{propn:bispanprod}, the condition for $\Psi$ is immediate if $\Phi$ preserves products. Conversely, for any coproduct decomposition $E \simeq \coprod E_{i}$, we can apply the condition for $\Psi$ with $X = E$ to conclude that $\Phi$ preserves this product.

  Finally, the commutativity of (\ref{diag:lawvere-equ-pw}) follows at once from the commutativity of (\ref{diag:prelawvere-pw}) and the observation $p_X=\ev_0\circ\iota_X$.
\end{proof}

\begin{remark}
	\label{rmk:Distributivity_Relation}
	By \Cref{thm:cmon-bispan-model}, an $\ctxt$-normed semiring in $\cX$ may be identified with a product-preserving functor $R\colon \Bispan_{M,A}(\cF) \to \cX$, and thus gives rise to maps $f^*\colon R(Y) \to R(X)$, $a_{\oplus}\colon R(X) \to R(Y)$ and $m_{\otimes}\colon R(X) \to R(Y)$ for morphisms $f,a,m \colon X \to Y$ in $\cF$ such that $a \in \cF_A$ and $m \in \cF_M$. Each of these classes of maps are compatible with composition in $\cF$. Furthermore, given pullback squares
	\[
	\begin{tikzcd}
		X' \dar[swap]{a'} \rar{g} \drpullback & X \dar{a} \\
		Y' \rar{f} & Y
	\end{tikzcd}
	\qquad \text{ and } \qquad
	\begin{tikzcd}
		X' \dar[swap]{m'} \rar{g} \drpullback & X \dar{m} \\
		Y' \rar{f} & Y
	\end{tikzcd}
	\]
	with $a \in \cF_A$ and $m \in \cF_M$, the composition relations in $\Bispan_{M,A}(\cF)$ give rise to relations $f^*a_{\oplus} \simeq a'_{\oplus} g^*$ and $f^*m_{\otimes} \simeq m'_{\otimes} g^*$ of maps $R(X) \to R(Y')$. Finally, given morphisms $a\colon X \to Y$ in $\cF_A$ and $m\colon Y \to Z$ in $\cF_M$, we may consider their associated ``distributivity diagram'' \cite{bispans}*{2.4.1}
	\[
	\begin{tikzcd}
		X \drar[swap]{a} & m^*m_*(X) \dar{b'} \rar{m'} \lar[swap]{e} \drpullback & m_*(X) \dar{b = m_*(a)} \\
		& Y \rar{m} & Z,
	\end{tikzcd}
	\]
	where $e$ is the counit of the adjunction $m^* \dashv m_*$; we then obtain a distributivity relation $m_{\otimes}a_{\oplus} \simeq b_{\oplus}m'_{\otimes}e^*$ of maps $R(X) \to R(Z)$.
\end{remark}

\begin{ex}
	Specializing \Cref{thm:cmon-bispan-model} to $\ctxt = (\xF, \xF, \xF)$ as in Example~\ref{ex:ggn}, we obtain equivalences ${\mathcal R}\kern-1pt\textup{ig}_{\mathbb E_\infty}(\cX)\simeq\NRig_\ctxt(\cX)\simeq\Fun^\times(\Bispan(\xF),\cX)$, recovering the main result of \cite{CHLL1}.
\end{ex}

\begin{ex}
	Our main interest is the case $\ctxt = (\xF_G, \xF_G, \xF_G)$ for a finite group $G$, which will be discussed extensively in \Cref{sec:Normed_Spectra}.
\end{ex}

\begin{ex}\label{ex:triv3}
  Applying \Cref{thm:cmon-bispan-model} to the case for trivial additive norms from \cref{ex:triv2}, we deduce that $\ctxt$-normed monoids in $\cX$ admit an interpretation as $\ctxt$-normed algebras:
  \[ \NAlg_{\ctxt}(\Fun^{\times}(\ulF^{\op}, \cX)) \simeq
    \Fun^{\times}(\Span_{\nrms}(\cF), \cX)  = \NMon_\ctxt(\cX).\]
    We may think of this as the normed analogue of the statement that commutative monoids in $\cX$ are the commutative algebras with respect to the cartesian symmetric monoidal structure on $\cX$.
\end{ex}

\subsection{Normed rings and Tambara functors}
Fix a cocomplete $\infty$-category $\cX$ with finite products such that the cartesian product preserves colimits in each variable. Among the normed semirings in $\cX$, we are especially interested in those that behave like \textit{rings}, in the sense that their underlying additive monoid is in fact a group:

\begin{defn}\label{defn:grouplike}
  Suppose $\ctxt_A = (\cF, \cF_{A})$ is an extensive span pair (and not only a \emph{weakly} extensive one). We then
  say an $\ctxt_A$-normed monoid $M \colon \Span_{A}(\cF) \to \cX$ is
  \emph{group\-like} if the induced commutative
  monoid structure on $M(X)$ from \Cref{obs:levelwise-monoid} is grouplike in the usual sense for every $X \in \cF$. We also refer to grouplike
  $\ctxt_A$-normed monoids as \emph{$\ctxt_A$-normed groups} and write
  $\NGrp_{\ctxt_A}(\cX) \subseteq \NMon_{\ctxt_A}(\cX)$ for the full subcategory
  of these. Note that under the equivalence \cref{eq:cmonincmon}, the full subcategory $\NGrp_{\ctxt_A}(\cX)$ corresponds to the subcategory
  \[ \NMon_{\ctxt_A}(\CGrp(\cX)) \subseteq \NMon_{\ctxt_A}(\CMon(\cX)) \simeq \NMon_{\ctxt_A}(\cX).\]
  If $\cX$ is presentable, then $\NGrp_{\ctxt_A}(\cX)$ is an accessible localization of $\NMon_{\ctxt_A}(\cX)$, and so is again presentable.
\end{defn}

\begin{defn}\label{defn:commring}
	A \textit{ring context} is a semiring context $\ctxt = (\cF, \cF_{M}, \cF_{A})$ such that the span pair $\ctxt_A = (\cF,\cF_A)$ is extensive. An \textit{$\ctxt$-normed ring in $\cX$} is an $\ctxt$-normed semiring $R\colon \Span_M(\cF)\to\ul\NMon_{\ctxt_A}(\cX)^\otimes$ such that for all $X\in \cF$ the resulting $\ctxt_{/X,A}$-normed monoid $R_X\colon \Span_A(\cF_{/X})\to\cX$ is an $\ctxt_{/X,A}$-normed group; here $\ctxt_{/X,A} = (\cF_{/X},\cF_{/X,A})$ denotes the span pair from \cref{ex:extslice}.
\end{defn}

\begin{ex}
The semiring contexts arising in equivariant mathematics, discussed in \cref{ex:F_G-rigctxt}, are always ring contexts.
\end{ex}

\begin{warning}\label{warn:cgrp-not-sym-mon}
  In the generality of our setup, the $\ctxt_M$-normed structure on
  $\ul\NMon_{\ctxt_A}(\cX)$ need not descend to grouplike objects: the
  latter may only form what should be called a
  \emph{$\ctxt_M$-$\infty$-operad}, and the previous definition could then
  be more succinctly phrased as saying that an $\ctxt$-normed ring is an
  $\ctxt_M$-normed algebra in this parametrized \iopd{}.

  However, we will show in the next section that such a normed structure \emph{does} exist in the setting of equivariant homotopy theory, which is the main case of interest to us.
\end{warning}

\begin{defn}\label{def:tambara}
  A product-preserving functor $\Bispan_{M,A}(\cF)\to\cX$ is called an ($\cX$-valued) \emph{$\ctxt$-Tambara functor} if its restriction to $\Span_A(\cF)\simeq\Bispan_{\eq,A}(\cF)$ is grouplike in the sense of Definition~\ref{defn:grouplike}. We write \[\Tamb_\ctxt(\cX)\subseteq\Fun^\times(\Bispan_{M,A}(\cF),\cX)\] for the full subcategory spanned by the Tambara functors.
\end{defn}

\begin{thm}\label{thm:ring-comparison}
  Let $\ctxt=(\cF,\cF_M,\cF_A)$ be a ring context. Then the equivalence $\NRig_\ctxt(\cX)\simeq\Fun^\times(\Bispan_{M,A}(\cF),\cX)$ constructed in Theorem~\ref{thm:cmon-bispan-model} restricts to
  \[\NRing_\ctxt(\cX)\simeq\Tamb_\ctxt(\cX).\]
  \begin{proof}
    Write $\Phi\colon\NRig_\ctxt(\cX)\to\FunT(\Bispan_{M,A}(\cF),\cX)$ for the equivalence from Theorem~\ref{thm:cmon-bispan-model}; we have to show that an $\ctxt$-normed semiring $R\in\NRig_\ctxt(\cX)$ is an $\ctxt$-normed ring if and only if $\Phi(R)$ is grouplike.

    By commutativity of (\ref{diag:lawvere-equ-pw}), $R$ is an $\ctxt$-normed ring if and only if $\Phi(R)\circ p_X$ is grouplike for every $X\in\cF$; we have to show that this is in turn equivalent to the composite $\phi\colon\Span_A(\cF)\to\Bispan_{M,A}(\cF)\to\cX$ being grouplike. But indeed, if $\hat\phi\colon\Span_A(\cF)\to\CMon(\cX)$ is the unique lift of $\phi$, its restriction along $\Span_A(\cF_{/X})\to\Span_A(\cF)$ is a lift of $\Phi(R)\circ p_X$; the claim follows as the functors $\Span_A(\cF_{/X})\to\Span_A(\cF)$ for varying $X\in\cF$ are jointly surjective.
  \end{proof}
\end{thm}

\section{Normed equivariant spectra}
\label{sec:Normed_Spectra}
In this section, we will prove the main results of our paper: In particular, we will define the $G$-normed \icat{} of $G$-spectra, compare it to $G$-commutative groups, and then finally specialize the results of the previous sections to describe connective normed $G$-spectra in terms of Tambara functors.

\begin{convention}
We will fix the ring context $\ctxt = (\xF_G, \xF_G,\xF_G)$ throughout the whole section, write ``$G$-\icat{}'' instead of ``$\xF_G$-\icat{}'', and write ``normed'' instead of \kern1pt``\kern-1pt$\ctxt$-normed.'' It will be convenient to think of normed $G$-\icats{} as functors $\Span(\xF_G)\to\CMon(\CatI)$, via Observation~\ref{obs:levelwise-monoid}. We will furthermore repurpose the notation $\cC^\otimes$ to refer to a normed $G$-\icat{} $\Span(\xF_G)\to\CMon(\CatI)$ with underlying $G$-\icat{} $\cC\colon\xF_G^\op\to\CatI$.
\end{convention}

\subsection[Presentable $G$-\icats{}]{\boldmath Presentable $G$-\icats{}}
Unlike in the rest of this paper, we will need a fair bit of parametrized higher category theory \cite{exposeI,MW} in this section, and we begin by recalling some of the basic terminology. For simplicity, we will restrict to the case of $G$-\icats{} here, although our references work in much greater generality.

\begin{construction}
  The \icat{} $\Fun^\times(\xF_G^\op,\CatI)\simeq\Fun(\textbf{O}^\op_G,\CatI)$ of $G$-\icats{} is cartesian closed. We write $\ul\Fun_G$ for the internal hom, and $\Fun_G=\ev_{G/G}\circ\ul\Fun_G$ for its underlying ordinary \icat.
\end{construction}

Using $\Fun_G$, we can view the \icat{} of $G$-$\infty$-categories as an $(\infty,2)$-category; all that we will need below is that this enhances the homotopy $1$-category to a $2$-category. In particular, we obtain a natural notion of \emph{adjunctions} between $G$-\icats. The following recognition principle will be useful:

\begin{lemma}[\cite{MW}*{3.2.9 and 3.2.11}]\label{lemma:adj-recognition-principle}
  A functor $F\colon\cC\to\cD$ of $G$-\icats{} admits a right adjoint if and only if the following hold:
  \begin{enumerate}
    \item For each $X\in\textbf{\textup O}_G$ (or equivalently for each $X\in\xF_G$), the functor $F_X\colon\cC(X)\to\cD(X)$ admits a right adjoint $G_X$ in the usual sense.
    \item For each $f\colon X\to Y$ in $\textbf{\textup O}_G$ (or equivalently for $f$ in $\xF_G$) the Beck--Chevalley transformation $f^*G_Y\to G_Xf^*$ is invertible.\qed
  \end{enumerate}
\end{lemma}

\begin{defn}
  A $G$-\icat{} $\cC\colon\xF_G\to\CatI$ is called \emph{presentable} if it satisfies all of the following conditions:
  \begin{enumerate}
    \item It factors through $\Pr^\textup{L}$.
    \item For each $g\colon C\to D$ in $\xF_G$ the functor $g^*\colon\cC(D)\to\cC(C)$ admits a left adjoint $g_!$, and for any pullback square
    \[
      \begin{tikzcd}
        A\arrow[r, "f"]\arrow[dr,phantom,"\lrcorner"{very near start}]\arrow[d,"p"'] & B\arrow[d, "q"]\\
        C\arrow[r, "g"'] & D
      \end{tikzcd}
    \]
    the Beck--Chevalley map $f_!p^*\to q^*g_!$ is invertible.
  \end{enumerate}
\end{defn}

\begin{cor}[\cite{MW_presentable}*{6.3.1}]\label{cor:G-SAFT}
  A functor $F\colon\cC\to\cD$ of presentable $G$-\icats{} is a left adjoint if and only if it \emph{\boldmath preserves $G$-colimits} in the following sense: $F$ is left adjointable, and for every $X\in\textbf{\textup O}_G$ (or equivalently $\xF_G$) the functor $\cC(X)\to\cD(X)$ preserves ordinary colimits.
  \begin{proof}
    By the Adjoint Functor Theorem the latter condition is
    equivalent to each $\cC(X)\to\cD(X)$ admitting a right adjoint
    $G_X$. By passing to total mates, the former condition is then
    equivalent to the Beck--Chevalley maps $f^*G_Y\to G_Xf^*$ being
    invertible. Thus, the claim follows from
    Lemma~\ref{lemma:adj-recognition-principle}.
  \end{proof}
\end{cor}

\begin{defn}
  We denote the \icat{} of presentable $G$-\icats{} and left adjoint (equivalently: $G$-cocontinuous) functors by $\Pr^\textup{L}_G$.
\end{defn}

\begin{remark}
  For every $\cC,\cD\in\Pr^\textup{L}_G$, there is a $G$-subcategory $\ul\Fun_G^\textup{L}(\cC,\cD)\subseteq\ul\Fun_G(\cC,\cD)$ of the internal hom, given in degree $G/G$ by the full subcategory of left adjoint functors $\cC\to\cD$; we refer the reader to \cite{CLL_Global}*{2.3.22} or \cite{MW}*{discussion before~3.3.6} for details. As we will recall in the proof of Theorem~\ref{thm:unique-naive} below, this is the internal hom for a parametrized analogue of the Lurie tensor product.
\end{remark}

\subsection[The $G$-\icat{} of $G$-spaces]{\boldmath The $G$-\icat{} of $G$-spaces}
As a warm-up and an ingredient for the construction of the normed $G$-\icat{} of $G$-spectra, we will recall two equivalent constructions of the $G$-\icat{} of \emph{$G$-spaces} in this subsection, and show that it admits a unique normed structure interacting suitably with (pointwise) colimits.

We begin with a construction via classical equivariant homotopy theory:

\begin{construction}
  Let $\SSet$ be the $1$-category of simplicial sets. Applying Construction~\ref{constr:borel-g-cat}, we obtain a \emph{Borel $G$-category} $\SSet^\flat$, given slightly informally as follows: $\SSet^\flat$ sends $G/H$ to the category of simplicial sets with (strict) $H$-action, with contravariant functoriality via restricting the action.

  We now equip each $\SSet^\flat(G/H)=\Fun(BH,\SSet)$ with the $H$-equivariant weak equivalences, i.e.~those maps $f$ such that $f^K$ is a weak homotopy equivalence for every $K\leqslant H$. As these are clearly stable under restriction, this defines a lift of $\SSet^\flat$ to a functor from $\textbf{O}^\op_G$ to relative categories. Postcomposing with Dwyer--Kan localization, we therefore obtain a $G$-$\infty$-category $\mS_G$.

  We moreover write $\mF_G$ for the Borel $G$-category $\xF^\flat$; equivalently, this is the full subcategory of $\mS_G$ spanned in degree $G/H$ by the finite $H$-sets.
\end{construction}

Next, let us compare this to a purely $\infty$-categorical construction.

\begin{construction}
  We write $\ul{\xF}_G$ for the $G$-\icat{} $X\mapsto (\xF_G)_{/X}$ and $\ul\Spc_G$ for the $G$-\icat{} $\cP_\Sigma(\ul\xF_G)\coloneqq\Fun^{\times}(\ul{\xF}_G^\op,\Spc)$, with functoriality via left Kan extension (cf.~Proposition~\ref{propn:LKEprodpres}). By \cite{HHLN2}*{8.1} the Yoneda embeddings assemble into a $G$-functor $\ul{\xF}_G\hookrightarrow\ul{\Spc}_G$, exhibiting the target as the pointwise sifted cocompletion.
\end{construction}

\begin{remark}
  Equivalently, $\ul{\Spc}_G$ is obtained from the ``co-$G$-\icat'' $\xF_G\to\CatI,X\mapsto (\xF_G)_{/X}$ (functoriality via pushforward) by applying the \emph{contravariant} functor $\Fun^\times(\blank,\Spc)$. As the inclusion $(\textbf{O}_G)_{/X}\hookrightarrow(\xF_G)_{/X}$ is the finite coproduct completion for every $X\in\textbf{O}_G$, we can also describe this as the functor $\textbf{O}_G^\op\to\CatI, X\mapsto\Fun((\textbf{O}_G)_{/X}^\op,\Spc)$. The latter description serves as the definition of $\ul\Spc_G$ in \cite{CLL_Clefts}.
\end{remark}

\begin{thm}\label{thm:unstable-comparison}
  There are unique equivalences $\mS_G\simeq\ul\Spc_G$ and $\mF_G\simeq\ul{\xF}_G$ of $G$-\icats. Moreover, these equivalences fit into a commutative diagram
  \begin{equation*}
    \begin{tikzcd}
      \mF_G\arrow[r, hook]\arrow[d,"\simeq"'] &[2em] \mS_G\arrow[d,"\simeq"]\\
      \ul{\xF}_G\arrow[r, hook, "\textup{Yoneda}"'] &[2em] \ul\Spc_G\rlap.
    \end{tikzcd}
  \end{equation*}
  \begin{proof}
    By \cite{CLL_Clefts}*{5.12} there exists a unique equivalence $\Phi\colon\mS_G\isoto\ul\Spc_G$. By virtue of being an equivalence, this preserves (levelwise) terminal objects and $G$-colimits.

    Let now $K\leqslant H\leqslant G$. Then both the restriction $i^*\colon\Fun(BH,\SSet)\to\Fun(BK,\SSet)$ as well as its $1$-categorical left adjoint $i_!$ are homotopical; thus, the $\infty$-categorical left adjoint $i_!\colon\mS_G(G/K)\to\mS_G(G/H)$ can simply be computed by the $1$-categorical left adjoint. In the same way, we see that terminal objects and coproducts in $\mS_G$ can be computed in the $1$-category $\SSet$. In particular, we have $H/K=i_!i^*(\bfone_{K})$ in $\mS_G$, whence \[\Phi(H/K)\simeq i_!i^*(\id_{G/H})\simeq (G/K\twoheadrightarrow G/H)\in(\xF_G)_{/(G/H)}.\] It follows by direct inspection that $\Phi$ maps $\textbf{O}_H\subseteq\mF_G(G/H)$ essentially surjectively into $(\textbf{O}_{G})_{/(G/H)}\subseteq\ul\xF_G(G/H)$, and closing up under finite coproducts we see that the equivalence $\Phi$ restricts to a functor $\phi\colon\ul{\xF}_G\to\mF_G$ that is essentially surjective, and hence itself an equivalence.

    Finally, $\ul{\xF}_G$ has no non-trivial automorphisms by \cite{CLL_Global}*{4.2.17}, which completes the proof of the proposition.
  \end{proof}
\end{thm}

\begin{remark}\label{rk:Spc-universal}
  As recalled in \cite{CLL_Global}*{2.4.11}, $\ul\Spc_G$ is the \emph{free presentable $G$-\icat{} on a point} in the following sense: for any $\cC\in\Pr^\textup{L}_G$ evaluation at the terminal object of $\Spc_G(G/G)$ defines an equivalence $\ul\Fun_G^\textup{L}(\ul\Spc_G,\cC)\isoto\cC$.
\end{remark}

We can also give a pointed version of the above comparison:

\begin{construction}
  Consider the category $\SSet_*$ of pointed simplicial sets. As before, we can associate to this a $G$-$1$-category $\SSet_*^\flat, G/H\mapsto\Fun(BH,\SSet_*)$, which we then localize at the (underlying) equivariant weak equivalences to obtain a $G$-$\infty$-category $\mS_{G,*}$. As the equivariant weak equivalences are part of a left proper model structure, we get a natural equivalence $\mS_{G,*}\simeq(\mS_G)_*$ compatible with the forgetful functors.

  We further write $\mF_{G,*}$ for the full subcategory spanned in degree $G/H$ by the finite pointed $H$-sets, so that $\mF_{G,*}\simeq(\mF_G)_*$.
\end{construction}

\begin{cor}\label{cor:based-comparison}
  There is a commutative diagram
  \begin{equation*}
    \begin{tikzcd}
      \mF_{G,*}\arrow[r, hook]\arrow[d,"\simeq"'] &[2em] \mS_{G,*}\arrow[d,"\simeq"]\\
      (\ul{\xF}_G)_*\arrow[r, hook] &[2em] (\ul\Spc_G)_*\rlap.
    \end{tikzcd}
  \end{equation*}
  in which the vertical maps are equivalences and the top and bottom vertical arrow exhibit their targets as sifted cocompletion of the respective sources.
  \begin{proof}
    In light of Theorem~\ref{thm:unstable-comparison}, the only non-trivial statement is that the horizontal maps define sifted cocompletions. For this it will be enough to consider the bottom arrow, where this is an immediate consequence of \cite{norms}*{4.1} as every $(Y\to X\to Y)\in \big((\xF_G)_{/Y}\big)_*$ is disjointly based.
  \end{proof}
\end{cor}

Next, we turn our attention to normed structures on these $G$-\icats; we restrict to the pointed case here (as this is the only one we will need below), although the unbased case is analogous.

\begin{propn}\label{prop:F-G-unique-sm}
  There exists a unique normed structure on $\mF_{G,*}$ with unit $S^0$ such that the symmetric monoidal product on $\mF_{G,*}(G/e)=\xF_*$ preserves finite coproducts in each variable.
  \begin{proof}
    By Corollary~\ref{cor:borel-unique-single}, it will suffice that $\xF_*$ (with trivial $G$-action) has a unique lift to $\Fun(BG,\CMon(\Cat_\infty))$ with unit $S^0$ and for which the tensor product preserves finite coproducts in each variable. Consider for this the version of the Lurie tensor product on $\Cat^\amalg$ representing functors that preserve finite coproducts in each variable. Then $(\xF_*,S^0)$ is an idempotent object for this tensor product by e.g.~\cite{CLL_Global}*{4.7.6}, whence it is also idempotent in $\Fun(BG,\Cat^\amalg)$ with the levelwise symmetric monoidal structure. The claim now follows from \cite{HA}*{4.8.2.9}.
  \end{proof}
\end{propn}

Combining this with Corollary~\ref{cor:based-comparison} and  the universal property of sifted cocompletion, we get:

\begin{cor}\label{cor:pted-unique-G-sym-mon}
  There exists a unique $G$-normed structure on $\mS_{G,*}$ together with a lift of $\mF_{G,*}\hookrightarrow\mS_{G,*}$ to a normed $G$-functor such that the following two conditions are satisfied:
  \begin{enumerate}
    \item For each $X\in\xF_G$, $\mS_{G,*}(X)$ is presentably symmetric monoidal.
    \item For each $f\colon X\to Y$ in $\xF_G$, the functor $f_\otimes\colon\mS_{G,*}(X)\to\mS_{G,*}(Y)$ preserves sifted colimits.
  \end{enumerate}
  The analogous statement for $\ul{\xF}_{G,*}\hookrightarrow\ul\Spc_{G,*}$ holds, and for these
  normed structures there is a unique normed equivalence $\mS_{G,*}^\otimes\simeq \ul\Spc_{G,*}^\otimes$.\qed
\end{cor}

Let us make the normed structure from \Cref{cor:pted-unique-G-sym-mon} explicit for our favorite model:

\begin{construction}
  We equip $\SSet_*$ with the symmetric monoidal structure coming from the smash product. This then yields a normed structure on the Borel $G$-category $\SSet_*^\flat$ via Proposition~\ref{prop:borel-G-sym-mon}. The symmetric monoidal structure on the individual categories $\SSet_*^\flat(G/H)$ is then given by the usual smash product (Observation~\ref{obs:pw-mon-borel}), while Corollary~\ref{cor:borel-norms-classical} shows that the map $i_\otimes\colon\Fun(BK,\SSet_*)\to\Fun(BH,\SSet_*)$ for subgroups $K\leqslant H\leqslant G$ is given by the classical \emph{symmetric monoidal norm}, i.e.~it sends a $K$-simplicial set $X$ to $X^{\wedge n}$ where $n=|H/K|$ and $H$ acts on $X^{\wedge n}$ by restricting the natural $\Sigma_n\wr K$-action along a certain homomorphism~$H\to\Sigma_n\wr K$; see Construction~\ref{constr:classical} for details.
\end{construction}

\begin{propn}\label{prop:spc-*-sm-loc}
  The previous construction localizes to a normed structure on $\mS_{G,*}$, and this is the normed structure from Corollary~\ref{cor:pted-unique-G-sym-mon}.

  \begin{proof}
    Since $\SSet_*^\flat$ comes with a normed $G$-functor $\xF_*^\flat\to\SSet_*^\flat$ by construction, the only non-trivial statement is that this localizes to a normed structure satisfying the assumptions (1) and (2) of \Cref{cor:pted-unique-G-sym-mon}.

    It is clear that the smash product of pointed $H$-simplicial sets preserves weak equivalences in each variable and is a left Quillen bifunctor (with respect to the model structures where cofibrations are levelwise injections). Thus, it descends to make each $\mS_{G,*}(G/H)$ into a presentably symmetric monoidal \icat{}. It remains to show that for every $K\leqslant H$ the symmetric monoidal norm functor $i_\otimes\colon\Fun(BK,\SSet_*)\to\Fun(BH,\SSet_*)$ preserves weak equivalences, and that the resulting functor on localizations preserves sifted colimits.

    For the first statement, let $f$ be a $K$-equivariant weak equivalence; we have to show that for any $j\colon L\hookrightarrow H$ the map $(j^*i_\otimes f)^L$ is a weak equivalence. Rewriting the cospan $G/H\to G/K \gets G/L$ as a span, we see that this splits as a smash product of maps $(i'_\otimes j^{\prime*}f)^L$, i.e.~after renaming we are reduced to showing that $(i_\otimes f)^H$ is a weak equivalence. But this map agrees with $f^K$ by direct inspection.

    For the second statement, we claim that the functor of $1$-categories $X\mapsto X^{\wedge n}$ preserves filtered colimits and geometric realization up to \emph{isomorphism}: the first statement is clear, while the second one follows from the fact that geometric realization is given by taking the diagonal of the associated bisimplicial set. As both of these operations are homotopical by \cite{g-global}*{1.1.2 and 1.2.57}, it follows that $i_\otimes\colon\mS_{G,*}(G/K)\to\mS_{G,*}(G/H)$ commutes with filtered colimits and $\Delta^\op$-shaped colimits, hence with all sifted colimits as claimed.
  \end{proof}
\end{propn}

\subsection[Norms on $G$-Mackey functors]{\boldmath Norms on
  $G$-Mackey functors}
  As a next step, we will show that also the $G$-\icat{} of normed $G$-monoids/$G$-Mackey functors from \Cref{ex:g-mackey} admits a unique normed structure. As a special case of \Cref{def:monoidal_CMon}, we obtain one such normed structure
   \[
  	   \ul\NMon_G\coloneqq\ul\NMon_G(\Spc) =\Fun^\times(\Span(\ul{\xF}_G),\Spc).
   \] 
   We begin by relating it to the unstable world:

\begin{propn}\label{prop:mathbb-P-sym-mon}
  There exists a unique $G$-left adjoint $\mathbb P\colon\ul\mS_{G,*}\to\ul\NMon_G$ sending $S^0$ to $\hom(\bfone,\blank)$. Moreover, this functor upgrades (canonically) to a normed $G$-functor.
  \begin{proof}
    For the first statement, we may equivalently consider
    $\ul\Spc_{G,*}=\ul\Spc_G\otimes\Spc_*$ in lieu of $\mS_{G,*}$. In
    this case, the existence and uniqueness of the $G$-left adjoint
    $\mathbb P$ follows via \cite{CLL_Spans}*{7.39} from the universal
    property of $\ul\Spc_G$ (Remark~\ref{rk:Spc-universal}) and the fact
    that the non-parametrized presentable \icat{} $\Spc_*$ is the mode
    for pointed presentable \icats{} \cite{HA}*{4.8.2.11}.

    To complete the proof, we will now construct a $G$-left adjoint
    normed $G$-functor $\mS_{G,*}^\otimes\to\ul\NMon_G^\otimes$. For
    this, note first that by \cite{CLL_Global}*{4.7.6} the inclusion
    $\xF\hookrightarrow\Span(\xF)$ extends (uniquely) to a
    coproduct-preserving functor $j\colon\xF_{*}\to\Span(\xF)$, and as
    both sides are idempotents in $\Cat^\amalg$ (see
    \cite{harpaz}*{5.3} for the target) this uniquely upgrades to a
    symmetric monoidal functor. Passing to Borel $G$-\icats{} we
    obtain a normed $G$-functor
    \[\mF_{G,*}=\ul\xF_{G}^\flat\to\Span(\xF)^\flat\simeq\Span(\ul\xF_G)\]
    sending $S^0$ to $\bfone$; here the final equivalence uses that
    $\Span(\ul\xF_G)\simeq\Span\circ\mF_{G,*}$ is a Borel $G$-\icat{}
    as postcomposing with the limit preserving functor $\Span$
    preserves right Kan extensions. Passing to sifted cocompletions and using that \[\Fun^\times(\Span(\ul\xF_G),\Spc)\simeq \Fun^\times(\Span(\ul\xF_G)^\op,\Spc)=\cP_\Sigma(\Span(\xF_G)),\] we then get a normed $G$-functor
    $\mS_{G,*}^\otimes\to\ul\NMon_G^\otimes$, whose underlying
    $G$-functor agrees up to equivalence with
    $\cP_\Sigma(j^\flat)\colon\cP_\Sigma(\xF_*^\flat)\to\cP_\Sigma(\Span(\xF)^\flat)$,
    i.e.~it is the restriction of the left Kan extension along
    $(j^\flat)^\op$ to product-preserving functors.

    To see that this is a $G$-left adjoint, we first note that each $\cP_\Sigma(j^\flat)(G/H)$ admits a right adjoint (given by restriction); it therefore only remains to check the Beck--Chevalley condition of \Cref{cor:G-SAFT}. For this we observe that for any inclusion $i\colon K\hookrightarrow H$ of subgroups of $G$, the functors $i^*\colon\Fun(BH,\xF_*)\to\Fun(BK,\xF_*)$ and $\Fun(BH,\Span(\xF))\to\Fun(BK,\Span(\xF))$ admit left adjoints $i_!$, given non-equivariantly by an $|H/K|$-fold coproduct. Thus, we may check the Beck--Chevalley condition before passing to sifted cocompletions, i.e.~we want to show that $i_!\circ\Fun(BK,j)\to\Fun(BH,j)\circ i_!$ is an equivalence of functors $\Fun(BK,\xF_*)\to\Fun(BH,\Span(\xF))$. But this may be checked after forgetting to $\Span(\xF)$, where this follows from the fact that $j$ preserves finite coproducts by construction.
  \end{proof}
\end{propn}

Restricting, we in particular get a normed structure on the $G$-functor $\mF_{G,*}\to\ul\NMon_G$. In fact, this once again uniquely characterizes the normed structure if we in addition impose compatibility with colimits:

\begin{propn}
  There exists a unique pair of a normed structure on $\ul\NMon_G$ and a normed structure on the $G$-functor $\mF_{G,*}\to\ul\NMon_G$ such that the following conditions are satisfied:
  \begin{enumerate}
    \item For each $H\leqslant G$, the symmetric monoidal \icat{} $\ul\NMon_G(G/H)$ is presentably symmetric monoidal.
    \item For each $K\leqslant H\leqslant G$ the norm $\ul\NMon_G(G/H)\to\ul\NMon_G(G/K)$ preserves sifted colimits.
  \end{enumerate}
  \begin{proof}
    We will first prove this statement with $\mF_{G,*}$ replaced by $\Span(\mF_{G})$. As $\ul\NMon_G$ is defined as the sifted cocompletion of the latter, the same argument as in Corollary~\ref{cor:pted-unique-G-sym-mon} reduces this to showing that $(\Span(\xF),\bfone)$ is idempotent in $\Cat^\amalg$, which was already recalled above.

    To complete the proof, we now observe that the data in question is equivalent to a normed structure on $\ul\NMon_G$ (satisfying the above two axioms) that preserves the full subcategory $\Span(\mF_G)$, together with a lift of $\mF_{G,*}\to\Span(\mF_G)$ to a normed $G$-functor. The former is no data by the above, while Corollary~\ref{cor:borel-sm-ff} together with the idempotency of $\xF_*$ and $\Span(\xF)$ shows that also the latter is unique.
  \end{proof}
\end{propn}

\subsection[$G$-spectra and their symmetric monoidal structure]{\boldmath $G$-spectra and their symmetric monoidal structure}
Let us begin by giving two equivalent descriptions of the $G$-\icat{} of $G$-spectra:

\begin{construction}
  We define the $G$-\icat{} $\ul\Sp_G$ of $G$-spectra as the pointwise stabilization of $\ul\NMon_G$, i.e.~it is the $G$-\icat{} \[X\mapsto \Fun^\times(\Span((\xF_G)_{/X}),\Sp)\] with functoriality via restriction along pushforwards. This comes with a natural \emph{stabilization map} $\ell\colon\ul\NMon_G\to\ul\Sp_G$, induced by the usual stabilization/delooping map $\CMon(\Spc)\to\Sp$. We write $\mathbb S_G$ for the image of $\hom(\bfone,\blank)\in\ul\NMon_G(G/G)$ under $\ell$.
\end{construction}

\begin{construction}
  Write $\textsf{Sp}^\Sigma$ for the $1$-category of symmetric spectra in simplicial sets. For each finite group $H$, the category $\textsf{Sp}^\Sigma$ carries an \emph{equivariant flat model structure} \cite{hausmann-equivariant}*{4.7} whose weak equivalences are the so-called $H$-equivariant weak equivalences and whose cofibrations are the so-called \emph{flat cofibrations}; the latter are independent of the group $H$. We write $\textsf{Sp}^\Sigma_\textup{flat}$ for the full subcategory of flat spectra (i.e.~those $X$ for which $\emptyset\to X$ is a flat cofibration).

  We now consider the Borel $G$-category $(\textsf{Sp}^\Sigma_\textup{flat})^\flat$, and we equip each category $(\textsf{Sp}^\Sigma_\textup{flat})^\flat(G/H)=\Fun(BH,\textsf{Sp}^\Sigma_\textup{flat})$ with the $H$-equivariant weak equivalences. By \cite{hausmann-equivariant}*{\S5.2} these are preserved under restriction, so we can Dwyer--Kan localize this to obtain a $G$-\icat{} $\mSp_G$.
\end{construction}

\begin{remark}
  The inclusion $(\textsf{Sp}^\Sigma_\textup{flat})^\flat\hookrightarrow(\textsf{Sp}^\Sigma)^\flat$ induces an equivalence of Dwyer--Kan localizations (being pointwise the inclusion of the cofibrant objects of a model category), so we could equivalently have worked without restricting to flat spectra. However, flatness will come in handy below to define the symmetric monoidal and normed structures on $\mSp_G$.
\end{remark}

\begin{thm}
  There is a unique equivalence $\mSp_G\simeq\ul\Sp_G$ sending $\mathbb S_G$ to $\mathbb S_G$.
  \begin{proof}
    Combine \cite{CLL_Clefts}*{9.13} with \cite{CLL_Spans}*{9.9}.
  \end{proof}
\end{thm}

Our goal is to make both sides into normed \icats{} and then upgrade the above equivalence to a normed equivalence. As a stepping stone for this, we will first prove a comparison that does not take norms into account. We therefore introduce:

\begin{defn}
  A (na\"ive) \emph{symmetric monoidal $G$-\icat{}} is a functor $\xF_G^\op\to\CMon(\CatI)$.
\end{defn}

Equivalently, we can view a symmetric monoidal $G$-\icat{} as a commutative monoid in the (ordinary) \icat{} of $G$-\icats{} with respect to the cartesian product. Restricting along $\xF_G^\op\hookrightarrow\Span(\xF_G)$, every normed $G$-\icat{} has an underlying symmetric monoidal $G$-\icat{}.

We will be particularly interested in the case where the underlying $G$-\icat{} $\cC$ is presentable and the tensor product $-\otimes-\colon\cC\times\cC\to\cC$ preserves $G$-colimits in each variable, i.e.~for each $G/H$ the symmetric monoidal structure on $\cC(G/H)$ is closed, and for all $i\colon H\hookrightarrow K$ the \emph{projection map} \[i_!(X\otimes i^*Y)\to i_!X\otimes Y\] (the Beck--Chevalley map associated to $i^*X\otimes i^*Y\simeq i^*(X\otimes Y)$) is invertible. We call a symmetric monoidal $G$-\icat{} \emph{$G$-presentably symmetric monoidal} in this case. Similarly we say that a normed $G$-\icat{} is \emph{$G$-presentably normed} if the underlying symmetric monoidal $G$-\icat{} is $G$-presentably symmetric monoidal.

\begin{ex}\label{ex:pted-spaces-sym-mon}
  We have already seen that the normed structure on $\mS_{G,*}$ coming from the smash product is presentably symmetric monoidal in each degree. As all functors in sight are homotopical, the projection formula can be checked on the pointset level, where it is a trivial computation.\footnote{This isomorphism of pointed $G$-(simplicial) sets is sometimes referred to as the \emph{shearing isomorphism}.}
\end{ex}

\begin{ex}\label{ex:G-spectra-naive}
  The usual smash product of ($H$-)symmetric spectra is homotopical when restricted to flat spectra \cite{hausmann-equivariant}*{6.1}, making $\mSp_G$ into a symmetric monoidal $G$-\icat{}. This is again $G$-presentably symmetric monoidal: the statement for the levelwise tensor product is again clear, while for the projection map we observe that the corresponding non-derived map is again an isomorphism by direct inspection, and that all functors in sight are homotopical on flat spectra.
\end{ex}

\begin{ex}
  Also the normed structure on $\ul\NMon_G$ is
  $G$-presentably symmetric monoidal. For this note first that
  Proposition~\ref{prop:dayconvolution} shows that each
  $\ul\NMon_G(X)$ is presentably symmetric monoidal. For the
  projection formula $i_!(X\otimes i^*Y)\simeq i_!X\otimes Y$ we
  observe that both sides preserve colimits in $X$ and $Y$, so we may
  assume that $X$ and $Y$ are both in the image of the free functor $\mathbb{P}\colon \ul\mS_{G,*}\rightarrow \ul\NMon_G $. In this case, the claim
  follows by Proposition~\ref{prop:mathbb-P-sym-mon} together with
  Example~\ref{ex:pted-spaces-sym-mon}.
\end{ex}

\begin{ex}
  As $\Sp$ is idempotent, $\ul\Sp_G=\Sp\otimes\ul\NMon_G$ inherits a symmetric monoidal structure from $\ul\NMon_G$ such that each $\ul\Sp_G(X)$ is presentably symmetric monoidal, see \cite{GepnerGrothNikolaus}*{5.1}. This is again $G$-presentably symmetric monoidal: by the universal property of stabilization, the projection formula can be checked after restricting along $\ul\NMon_G\to\ul\Sp_G$, where this was verified in the previous example.
\end{ex}

In fact, the $G$-presentably symmetric monoidal structures considered in the above examples are unique:

\begin{thm}\label{thm:unique-naive}\
  \begin{enumerate}
    \item  The $G$-\icats{} $\ul\Spc_{G,*}$ and $\ul\mS_{G,*}$ admit unique $G$-\hskip0ptpresentably symmetric monoidal structures with unit $S^0$.
   \item The $G$-\icat{} $\ul\NMon_G$ admits a unique $G$-presentably symmetric mon\-oidal structure with unit $\hom(\bfone,\blank)$.
    \item The $G$-\icats{} $\ul\Sp_{G}$ and $\ul\mSp_{G}$ admit
      unique $G$-presentably symmetric mon\-oidal structures with unit
      $\mathbb S_G$.
  \end{enumerate}
  \noindent Moreover, the $G$-functors
  $\mS_{G,*}\to\ul\NMon_G\to\ul\Sp_G$ considered above enhance
  uniquely to maps of symmetric monoidal $G$-\icats{}, as do the
  equivalences $\ul\Spc_{G,*}\simeq\mS_{G,*}$ and
  $\ul\Sp_{G}\simeq\mSp_G$.
  \begin{proof}
    By \cite{MW_presentable}*{\S8.2}, the \icat{} of presentable $G$-\icats{} comes with a \emph{parametrized Lurie tensor product}, corepresenting bifunctors that preserve $G$-colimits in each variable. The unit is the $G$-\icat{} $\ul\Spc_G$, and the tensor product can be computed by the formula
    \[
      \cC\otimes\cD=\ul\Fun^\textup{R}_G(\cC^\op,\cD)\simeq\ul\Fun^\textup{L}_G(\cC,\cD^\op)^\op
    \]
    with functoriality in $\cC$ given via precomposition, see \cite{MW_presentable}*{8.2.11}.

    It now suffices to show that $(\ul\Spc_{G,*}, S^0)$, $(\ul\NMon_G,\hom(\bfone,\blank))$, and $(\mSp_{G},\mathbb S_G)$ are all idempotent with respect to this tensor product. By the above explicit formula for the tensor product, the statement for $\mSp_{G}$ is an instance of \cite{CLL_Clefts}*{9.13(2)}, while the statement for $\ul\NMon_G$ follows by combining \cite{CLL_Global}*{4.8.11} with \cite{CLL_Spans}*{9.9}.

    Finally, for $\ul\Spc_{G,*}=\ul\Spc_G\otimes\Spc_{*}$, we recall that $\ul\Spc_{G}\otimes\blank\colon\PrL\to\PrL_G$ admits a strong symmetric monoidal structure \cite{MW_presentable}*{end of 8.3.8}. In particular, it sends the idempotent $\Spc_*$ to an idempotent, finishing the proof.
  \end{proof}
\end{thm}

\subsection[Normed structures on $G$-spectra]{\boldmath Normed structures on
  $G$-spectra}
In this subsection we will finally construct the $G$-normed structure on $\ul\Sp_G$; in particular, we will show:

\begin{thm}\label{thm:G-sym-mon-unique}
  There exists a unique pair of a normed structure on $\ul\Sp_G$ together with a
  lift of $\ell\colon\ul\NMon_G\to\ul\Sp_G$ to a normed $G$-functor
  $\ell^\otimes\colon\ul\NMon_G^\otimes\to\ul\Sp_G^\otimes$ that satisfies the
  following two properties:
  \begin{enumerate}
    \item For each $X\in\xF_G$, $\ul\Sp_G^\otimes(X)$ is presentably symmetric monoidal.
    \item For each $f\colon X\to Y$ in $\xF_G$, the norm functor $f_\otimes\colon\ul\Sp_G^\otimes(X)\to\ul\Sp_G^\otimes(Y)$ preserves sifted colimits.
  \end{enumerate}
\end{thm}

This will require some further preparations.

\begin{lemma}\label{lemma:cmon-comp-prl}
  Let $\cI$ be a small \icat{} with finite coproducts equipped with a symmetric monoidal structure that preserves coproducts in each variable. Then the Day convolution on $\Fun(\cI^\op,\Spc)$ restricts to a symmetric monoidal structure on $\mathcal P_\Sigma(\cI)=\Fun^\times(\cI^\op,\Spc)$. Moreover, this is a presentably symmetric monoidal structure, and the tensor product of compact objects is compact again.
  \begin{proof}
    The fact that this restricts is the content of \cref{prop:dayconvolution}, where it is also shown that this is equivalently the localization of the Day convolution structure, whence in particular presentably symmetric monoidal.

    For the second statement, we now claim that any compact object in $\cP_\Sigma(\cI)$ is a retract of a finite colimit of representables. To prove the claim, consider any $X\in\cP_\Sigma(\cI)$, and write it as a colimit $\colim_{i\in I}x_i$ in $\cP(\cI)=\Fun(\cI^\op,\Spc)$ of representables. Applying the localization functor $\cP(\cI)\to\cP_\Sigma(\cI)$ we then also get such a colimit decomposition in $\cP_\Sigma(\cI)$. Restricting along a cofinal functor, we may assume that $I$ is a poset \cite{HTT}*{4.2.3.15}, and filtering it by its finite subsets, we can express $X$ as a filtered colimit in $\cP_\Sigma(\cI)$ of finite colimits of representables, i.e.~we have an equivalence
    \[
      \phi\colon X\xrightarrow{\;\simeq\;}\colim_{J\subseteq I\text{ finite}} \colim_{j\in J} x_j.
    \]
    By compactness of $X$, $\phi$ has to factor through a map $\psi\colon X\to\colim_{j\in J} x_j$ for some finite $J\subseteq I$. The composite $\colim_{j\in J} x_j\to\colim_{J\subseteq I\text{ finite}} \colim_{j\in J} x_j\simeq X$ is then a retraction of $\psi$, finishing the proof of the claim.

    With this, we can now easily prove the second statement: if $X$ and $Y$ are compact, then the above shows that $X\otimes Y$ is again a retract of a finite colimit of representables (using that $\otimes$ preserves colimits in each variable). As representables are compact in $\cP$, and hence also in $\cP_\Sigma$ (using that the latter is closed under filtered colimits), this immediately implies that $X\otimes Y$ is compact, as desired.
  \end{proof}
\end{lemma}

\begin{lemma}\label{lemma:sym-mon-stab}
  Let $\cD$ be presentably symmetric monoidal and pointed. Then the stabilization map $\Sigma^\infty\colon\cD\to\Sp(\cD)$ lifts uniquely to a map in $\CAlg(\PrL)$, and this lift is symmetric monoidal inversion of $\Sigma\bbone$.
  \begin{proof}
    The existence and uniqueness of the symmetric monoidal structure follows from idempotency of $(\Sp,\mathbb S)$ in $\PrL$, see~\cite{GepnerGrothNikolaus}*{5.1}. It therefore only remains to prove that this map is symmetric monoidal inversion.

    As the symmetric monoidal product $\otimes$ preserves colimits in each variable, $\Sigma\bbone\otimes\blank$ is equivalent to the suspension functor $\Sigma$. Expressing $\Sigma^\infty\colon\cD\to\Sp(\cD)$ as the sequential colimit in $\PrL$ along $\Sigma$, the claim is therefore an instance of \cite{gepner-meier}*{C.6} once we show that $\Sigma\bbone$ is a \emph{symmetric object}, in the sense that for some $n\ge2$ the automorphism of $(\Sigma\bbone)^{\otimes n}$ induced by the permutation $\sigma\coloneqq(1\;2\;\dots\;n)\in\Sigma_n$ is trivial. This is in fact true for any odd $n$ as in this case $\sigma$ has sign $+1$, so that the induced automorphism of $S^n$ also has degree $1$.
  \end{proof}
\end{lemma}

\begin{propn}\label{prop:stab-CMon-G}
  For any $X\in\xF_G$, the functor $\ell\colon{\ul\NMon}_G(X)\to{\ul\Sp}_G(X)$ is symmetric monoidal inversion of the object $\Sigma\bbone$ in both $\CAlg(\Pr^\textup{L})$ and $\CAlg(\Cat^\textup{sifted})$.
  \begin{proof}
    Observe first that $\bbone=\hom(\bfone,\blank)$ is compact, whence so is $\Sigma\bbone$. Moreover, we have seen in the proof of Lemma~\ref{lemma:sym-mon-stab} that $\Sigma\bbone$ is a symmetric object, while Lemma~\ref{lemma:cmon-comp-prl} shows that the the symmetric monoidal product preserves compact objects and colimits in each variable. Thus, \cite{norms}*{4.1} shows that the universal map in $\CAlg(\PrL)$ inverting $\Sigma\bbone$ agrees with the universal map in $\CAlg(\Cat^\textup{sifted})$. The claim therefore follows from the previous lemma.
  \end{proof}
\end{propn}

\begin{lemma}\label{lemma:rep-sphere}
  Let $f\colon X\to Y$ be any map in $\xF_G$. Then the symmetric monoidal functor $\ell\colon{\ul\NMon}_G(Y)\to{\ul\Sp}_G(Y)$ sends $f_\otimes(\Sigma\bbone)$ to an invertible object.
  \begin{proof}
    Consider first the special case that $f$ is the projection $G/H\to
    G/K$ for some $H\leqslant K$. By Theorem~\ref{thm:unique-naive},
    it will be enough to show that the composite
    \[\mathfrak L\colon\ul\NMon_G\xlongrightarrow{\ell} \ul\Sp_G\isoto\mSp_G\] sends
    $f_\otimes\Sigma\hom(\bfone,\blank)$ to an invertible object (with
    respect to the derived smash product of $K$-equivariant symmetric
    spectra). For this we compute
    \[
      \mathfrak L(f_\otimes\Sigma\hom(\bfone,\blank))=\mathfrak Lf_\otimes\Sigma\mathbb P(S^0)\simeq (\mathfrak L\mathbb P)(f_\otimes S^1)\simeq (\mathfrak L\mathbb P)(S^{K/H})
    \]
    where $\mathbb P\colon\mS_{G,*}\to\ul\NMon_G$ is the normed $G$-left adjoint from Proposition~\ref{prop:mathbb-P-sym-mon}, and the last equation uses the explicit description of the normed structure on $\mS_{G,*}$. Now $\mathfrak L\circ\mathbb P\colon\mS_{G,*}\to\mSp_{G}$ is a $G$-left adjoint sending $S^0$ to $\mathbb S_G$, so it is necessarily the suspension spectrum functor. But the representation sphere $\Sigma^\infty S^{K/H}$ is invertible with respect to the smash product of $K$-spectra by \cite{hausmann-equivariant}*{4.9(i)}, finishing the proof of the special case.

    In the case of a general map $f\colon X\to Y$ in $\xF_G$, we first note that an object is invertible in $\ul\Sp_G(Y)$ if and only if it is so after restricting to each orbit. By the double coset formula, we may therefore assume that $Y=G/H$. Decomposing $X$ into its orbits then provides a factorization of $X\to G/H$ as
    \[
      X=\coprod_{i=1}^r X_i\xrightarrow{\coprod f_i}\coprod_{i=1}^r G/H\xrightarrow{\nabla} G/H,
    \]
    whence $\ell(f_\otimes\Sigma\hom(\bfone,\blank))\simeq\ell\big({\bigotimes_{i=1}^r f_{i\otimes}\Sigma(\hom(\bfone,\blank))}\big)\simeq\bigotimes_{i=1}^r \ell f_{i\otimes}(\Sigma\hom(\bfone,\blank))$. As invertible elements are closed under tensor product, this completes the proof of the lemma.
  \end{proof}
\end{lemma}

\begin{proof}[Proof of Theorem~\ref{thm:G-sym-mon-unique}]
  By \cref{thm:unique-naive}, $\ell$ lifts uniquely to a natural transformation of functors into the \icat{} of presentably symmetric monoidal \icats{} and sifted-colimit-preserving functors, and by Proposition~\ref{prop:stab-CMon-G} this map is pointwise given by symmetric monoidal inversion of $\Sigma\bbone$. This then uniquely extends to the desired map of normed $G$-\icats{} (viewed as functors $\Span(\xF_G)\to\CMon(\CatI)$ as per our standing convention) by the universal property of symmetric monoidal inversion combined with the previous lemma.
\end{proof}

Let us now give alternative interpretations of this normed $G$-\icat:

\begin{propn}[cf.~\cite{norms}*{9.11 and 9.13}]\label{prop:comparison-BH}
  For every $H\leqslant G$, the symmetric monoidal functor $(\ell\circ\mathbb P)(G/H)\colon\ul\Spc_{G,*}^\otimes(G/H)\to\ul\Sp^\otimes_G(G/H)$ is given by universally inverting the objects of the form $f_\otimes\Sigma\bbone$ in $\CAlg(\PrL)$, or equivalently in $\CAlg(\Cat^\textup{sifted})$.
\end{propn}

Note that Nardin and Shah use this as the \emph{definition} of the normed structure on $\ul\Sp_G$ \cite{NardinShah}*{2.4.2}, following \cite{norms}*{\S9.2}. Thus, this result in particular shows that our approach agrees with their construction.

\begin{proof}
First note that this holds for $\Sigma^\infty\colon\mS_{G,*}(G/H)\to\mSp_G(G/H)$ by \cite{gepner-meier}*{C.7} together with \cite{hausmann-equivariant}*{7.5}.\footnote{Hausmann a priori just provides a Quillen equivalence without referring to the symmetric monoidal structures on both sides, but the left Quillen functor from symmetric to orthogonal spectra is strong symmetric monoidal with respect to Day convolution, so this automatically gives an equivalence of symmetric monoidal \icats.} Thus, it will suffice to lift the commutative triangle
\begin{equation}\label{diag:magic-square}
  \begin{tikzcd}[column sep=small]
    &\mS_{G,*}\arrow[dl,"\Sigma^\infty"', bend right=15pt]\arrow[dr, "\ell\circ\mathbb P", bend left=15pt]\\
    \mSp_{G}\arrow[rr, "\simeq"'] && \ul\Sp_{G}
  \end{tikzcd}
\end{equation}
observed in the proof of Theorem~\ref{thm:G-sym-mon-unique} to a commutative diagram of symmetric monoidal $G$-\icats{} with respect to the symmetric monoidal structures inherited from the normed structures considered above. This is however clear from Theorem~\ref{thm:unique-naive} (using that all of these structures are indeed $G$-presentably symmetric monoidal by the above).
\end{proof}

Finally, we can also describe the normed structure in terms of models:

\begin{propn}
  The normed structure on $(\Sp^\Sigma_\textup{flat})^\flat$ given by the smash product localizes to a $G$-presentably normed structure on $\mSp_{G}$ such that the norm $f_\otimes\colon\mSp^\otimes(X)\to\mSp^\otimes(Y)$ preserves sifted colimits for every $f\colon X\to Y$ in $\xF_G$.
  \begin{proof}
    In view of Example~\ref{ex:G-spectra-naive} it only remains to show that the symmetric monoidal norms $N^H_K\colon\Fun(BK,\Sp^\Sigma_\textup{flat})\to\Fun(BH,\Sp^\Sigma_\textup{flat})$ (known as the \emph{Hill--Hopkins--Ravenel norms}) are homotopical and that the resulting functors on localizations preserve sifted colimits.

    The first statement is \cite{hausmann-equivariant}*{6.8}. For the second statement we first observe that both filtered colimits and geometric realization in $\Sp^\Sigma$ are homotopical: namely, both are left Quillen (with respect to the projective model structures on the respective source functor categories) and moreover preserve levelwise weak equivalences (maps $f\colon X\to Y$ such that each $f(A)$ is a $(G\times\Sigma_A)$-weak equivalences) as observed in the proof of Proposition~\ref{prop:spc-*-sm-loc}. We moreover claim that both of these constructions preserve flat spectra. To see this, we recall that a symmetric spectrum $X$ is flat if and only if for each finite set $A$ a certain natural \emph{latching map} $L_A(X)\to X(A)$ is levelwise injective, see \cite{hausmann-equivariant}*{2.18}; all that we will need to know is that $L_A$ is defined as a certain colimit, and in particular commutes with geometric realization and all colimits. The two claims now immediately follow as injections of simplicial sets are preserved by filtered colimits and geometric realization.

    With this established, it will be enough to show that $N^G_H$ commutes with geometric realization and filtered colimits on the point-set level, up to \emph{isomorphism}. In particular, we can forget about all the actions and simply consider the endofunctor $X\mapsto X^{\wedge n}$ of $\Sp^\Sigma\supseteq\Sp^\Sigma_\textup{flat}$. The statement about filtered colimits is then clear as the smash product preserves colimits in each variable. Similarly, the statement about geometric realizations reduces to showing that for any $X_\bullet\colon\Delta^\op\to\Sp^\Sigma_\textup{flat}$ and $n\ge1$ the map
    \[
      \int^{[k]\in\Delta^\op} X_k\wedge\Delta^k_+\to\int^{[k_1],\dots,[k_n]\in(\Delta^\op)^n} X_{k_1}\wedge\cdots\wedge X_{k_n}\wedge \Delta^{k_1}_+\wedge\cdots\wedge\Delta^{k_n}_+
    \]
    induced by the diagonal embedding is an isomorphism. By induction, we reduce to proving this for the map
    \[
      \int^{[k]} Y_{k,k}\wedge\Delta^k_+\to \int^{[k_1]}\mskip-10mu\int^{[k_2]} Y_{k_1,k_2}\wedge\Delta^{k_1}_+\wedge\Delta^{k_2}_+
    \]
    for any bisimplicial object $Y$ in $\Sp^\Sigma$. Arguing levelwise, this follows at once from the fact that the geometric realization of a simplicial object in (pointed) simplicial sets is just given by its diagonal.
  \end{proof}
\end{propn}

\begin{thm}\label{thm:sp-normed-comparison}
  The equivalence $\mSp_G\simeq\ul\Sp_G$ of $G$-\icats{} upgrades canonically to an equivalence $\ul\mSp_{G}^\otimes\simeq\ul\Sp_{G}^\otimes$ of normed $G$-\icats{}.
  \begin{proof}
    Of the maps of $G$-\icats{} comprising the diagram (\ref{diag:magic-square}), all except for the lower one have been lifted to maps in $\NAlg_G(\CatI)$ above. As $\Sigma^\infty$ is given by universally inverting representation spheres, and we have shown that they become invertible in $\ul\Sp_G$, there is then a unique normed pointwise left adjoint $\ul\mSp_G^\otimes\to\ul\Sp_G^\otimes$ making the diagram commute. It only remains to show that this map forgets to our equivalence $\mSp_G\simeq\ul\Sp_G$.

    For this we simply note that after forgetting to symmetric monoidal $G$-\icats{} there is still a unique map making the diagram commute, and we have lifted the equivalence $\mSp_G\simeq\ul\Sp_G$ to such a map in the proof of Proposition~\ref{prop:comparison-BH}.
  \end{proof}
\end{thm}

\subsection{The multiplicative equivariant recognition theorem}
As an upshot of all the hard work done in the previous subsections, we can now easily prove Theorem~\ref{introthm:Gsymmeq} from the introduction:

\begin{thm}\label{thm:recognition}\
  \begin{enumerate}
    \item The normed structure on $\ul\NMon_G$ localizes to a normed structure on $\ul\NGrp_G$.
    \item The normed structure on $\mSp_G$ restricts to a normed structure on the full $G$-subcategory $\mSp_{G}^{\ge0}$ spanned by the connective equivariant spectra.
    \item The delooping functor $\ul\NMon_G\to\mSp_G$ acquires a canonical normed structure, and this restricts to a normed equivalence $\ul\NGrp_G^\otimes\simeq(\mSp_{G}^{\ge0})^\otimes$.
  \end{enumerate}
  \begin{proof}
    By Theorem~\ref{thm:sp-normed-comparison}, we may replace $\mSp_G$ by $\ul\Sp_G$; under this identification, the $G$-subcategory $\mSp_{G}^{\ge0}$ corresponds to $\ul\Sp_{G}^{\ge0}\coloneqq\Fun(\Span((\xF_G)_{/-}),\Sp^{\ge0})$.

    We now observe that the essential image of $\ell\colon\ul\NGrp_G\to\ul\Sp_G$ is precisely $\ul\Sp_{G}^{\ge0}$. As we have lifted $\ell$ to a normed functor in Theorem~\ref{thm:G-sym-mon-unique}, this shows that $\ul\Sp_{G}^{\ge0}$ is indeed a normed subcategory. Since $\ell$ factors as the localization functor $\ul\NMon_G\to\ul\NGrp_G$ followed by an equivalence $\ul\NGrp_G\simeq\ul\Sp_{G}^{\ge0}$, this then immediately implies the remaining statements.
  \end{proof}
\end{thm}

\begin{cor}
  The cocartesian fibration $\ul\NMon_G^\otimes\to\Span(\xF_G)$ restricts to a cocartesian fibration $\ul\NGrp_G^\otimes\to\Span(\xF_G)$, and the inclusion $\iota\colon\ul\NGrp_G^\otimes\hookrightarrow\ul\NMon_G^\otimes$ is lax normed.
  \begin{proof}
    For each $X\in\xF_G$, the functor
    $\ell\colon\ul\NMon_G(X)\to\ul\Sp_{G}^{\ge0}(X)$ has a fully faithful
    right adjoint, induced by the right adjoint $\Sp^{\ge0}\to\CMon$ of
    the delooping functor, and this induces a relative adjunction of
    cocartesian fibrations over $\xF_G^\op$. As a consequence of
    \cite{HA}*{7.3.2.6} (parallel to \cite{HA}*{7.3.2.8}), the right
    adjoint $\iota$ then canonically lifts to a lax normed functor
    $\iota^\otimes\colon(\ul\Sp_{G}^{\ge0})^\otimes\hookrightarrow\ul\NMon_G^\otimes$,
    which is fully faithful with essential image the subcategory $\ul\NGrp_G$.
  \end{proof}
\end{cor}

We now immediately obtain the following generalization of Theorem~\ref{introthm:tambara} from the introduction:

\begin{thm}
	\label{thm:Main_Theorem}
	For a weakly extensive subcategory $I \subseteq \xF_G$, there is an
	equivalence between
	\begin{itemize}
		\item the $\infty$-category $\NAlg_I\big((\mSp_{G}^{\ge0})^\otimes\big)\coloneqq \NAlg_{(\xF_G,I)}\big((\mSp_{G}^{\ge0})^\otimes\big)$ of connective \emph{\boldmath $I$-normed $G$-spectra}, and
		\item the $\infty$-category $\Tamb_{(\xF_G,I)}(\Spc) \subseteq \Fun^{\times}(\Bispan_I(\xF_G),\Spc)$ of space-valued $(\xF_G,I)$-Tambara functors.
	\end{itemize}
\end{thm}
\begin{proof}
	By \Cref{thm:recognition}, $\NAlg_{(\xF_G,I)}\big((\mSp_{G}^{\ge0})^\otimes\big)$ is equivalent to $\NAlg_{(\xF_G,I)}(\ul\NGrp_G^\otimes)$. The inclusion $\NAlg_{(\xF_G,I)}(\ul\NGrp_G^\otimes) \hookrightarrow \NAlg_{(\xF_G,I)}(\ul\NMon_G^\otimes) = \NRig_{(\xF_G,I)}(\Spc)$ identifies its source with the subcategory $\NRing_{(\xF_G,I)}(\Spc) \subseteq \NRig_{(\xF_G,I)}(\Spc)$ of normed $G$-rings in the \icat{} of spaces. By Theorem~\ref{thm:ring-comparison}, the equivalence between $\NRig_{(\xF_G,I)}(\Spc) \simeq \Fun^\times(\Bispan_I(\xF_G),\Spc)$ of Theorem~\ref{thm:cmon-bispan-model} restricts to an equivalence between $\NRing_{(\xF_G,I)}(\Spc)$ and $\Tamb_{(\xF_G,I)}(\Spc)$. Combining these three equivalences gives the result.
\end{proof}

If $I\subseteq\xF_G$ is even an extensive subcategory (i.e.~an
indexing system), then  
set-valued $(\xF_G,I)$-Tambara functors are known under the name \textit{incomplete Tambara functors} \cite[4.1]{blumberg-hill}; thus, we may think of the right-hand side of the theorem as ``higher'' incomplete Tambara functors.

\begin{remark}
  Let $I\subseteq\xF_G$ be an indexing system and $\cC^{\otimes}$
  an $I$-normed $G$-\icat{}. As we will now explain, the \icat{}
  $\NAlg_{{I}}(\cC)$ can be identified with that of algebras for
  the $G$-\iopd{} $\textup{Com}_{I}^\otimes$ as defined by Nardin and
  Shah~\cite{NardinShah}*{2.4.10}: By definition, the \icat{}
  $\NAlg_{{I}}(\cC^\otimes)$ is that of sections
  $\Span_{I}(\xF_G)\to \cC^{\otimes}$ that are cocartesian over
  $\xF_G^\op$. Here the inclusion
  $\Span_{I}(\xF_G) \to \Span(\xF_{G})$ exhibits $\Span_{I}(\xF_G)$ as
  an equivariant \iopd{} when these are defined over the base
  $\Span(\xF_{G})$ (see \cite{envelopes}*{\S 5.2}), and $I$-normed algebras in $\cC^\otimes$ are precisely
  algebras for this \iopd{}. In \cite{NardinShah} the theory of
  equivariant \iopds{} is instead developed over a different base
  $\ul\xF_{G,*}$ (a specific model of the cocartesian unstraightening of the functor $\ul{\xF}_{G,*}\colon\xF_G\to\CatI$ considered above), but these two versions of $G$-\iopds{} were shown to
  be equivalent under pullback along a certain functor
  $\ul\xF_{G,*} \to \Span(\xF_{G})$ in \cite{envelopes}*{5.2.14}. It
  is clear from the definitions that $\textup{Com}_{I}^\otimes$ is
  precisely the pullback
  $\ul\xF_{G,*} \times_{\Span(\xF_{G})} \Span_{I}(\xF_{G})$, so by
  \cite{envelopes}*{5.3.17} we get an equivalence between the \icat{}
  of $\Span_{I}(\xF_{G})$- and $\textup{Com}_{I}^\otimes$-algebras in
  $\cC^\otimes$. In particular, our $I$-normed $G$-spectra are equivalently $\textup{Com}_I^\otimes$-algebras in $\ul\Sp_G^\otimes$ in the sense of \cite{NardinShah}.
\end{remark}

\begin{remark}
  Recall from Remark~\ref{rk:mackey-functors} that any indexing system $I\subseteq\xF_G$ has an associated $N_\infty$-operad $\cO$ in $G$-spaces. It is generally expected that the \icat{} of $\textup{Com}^\otimes_{I}$-algebras in $\ul\Sp_{G}^\otimes$ is modelled by $\cO$-algebras in a good model category of $G$-spectra, like $G$-symmetric spectra; however, to our knowledge no rigorous proof of this comparison has appeared in the literature.
\end{remark}

\appendix
\section{The Borel construction}
In this appendix, we recall from \cite{borel-normed} how any \icat{} with $G$-action gives rise to a $G$-\icat{} and how similarly any symmetric monoidal \icat{} with $G$-action yields a normed $G$-\icat.

\subsection[Borel $G$-\icats{}]{\boldmath Borel $G$-\icats{}}

We start by constructing the functor
\[
	(-)^{\flat}\colon \Fun(BG,\CatI) \to \Fun^{\times}(\xF_G^{\op}, \CatI)
\]
from \icats{} with $G$-action to $G$-$\infty$-categories, which is
used, for instance, to define the $G$-\icats{} $\mS_G$ and $\mSp_G$.

\begin{construction}\label{constr:borel-g-cat}
  Write $k\colon (BG)^\op\hookrightarrow\xF_G$ for the inclusion of
  the full subcategory on the free $G$-set $G$. Then $k$ is fully faithful, so $k^*\colon\Fun(\xF_G^\op,\CatI)\to\Fun(BG,\CatI)$ has a fully faithful right adjoint $(-)^\flat$, which is uniquely characterized by demanding that we have a counit equivalence $\epsilon\colon k^*(-)^\flat\to\id$ and that each individual $\cC^\flat$ be right Kan extended.

  We will now give an explicit construction of $(-)^\flat$. For this we note that the inclusion $\Fun(BG,\Spc)\hookrightarrow\Fun(BG,\CatI)$ is cocontinuous, hence (by the universal property of presheaves) left Kan extended from the functor $(BG)^\op\to\Fun(BG,\CatI)$ classifying the right $G$-set $G$. Restricting to a full subcategory, we see that also the inclusion $i\colon\xF_G\hookrightarrow\Fun(BG,\CatI)$ is left Kan extended from the same functor. Thus, $\Fun_G(i(-),\cC)$ is right Kan extended, and we see that $(-)^\flat$ is given by the assignment $\cC\mapsto\Fun_G(i(-),\cC)$, where the right-hand side denotes the internal hom in $\Fun(BG,\CatI)$; the counit is the evident equivalence $\Fun_G(G,\cC)\simeq\cC$. Note that $(-)^\flat$ lands in the \icat{} $\Fun^\times(\xF_G^\op,\CatI)$ of $G$-\icats, so that we obtain an adjunction
  \[
  		k^*\colon \Fun^{\times}(\xF_G^{\op},\CatI) \rightleftarrows \Fun(BG,\CatI)\noloc (-)^\flat.
  \]
\end{construction}

\begin{defn}
  We will refer to $G$-\icats{} in the essential image of $(-)^\flat$ as \emph{Borel $G$-\icats{}}.
\end{defn}

\begin{remark}
	\label{rmk:Explicit_Description_Borel_Construction}
  In all of our examples, we apply the above right adjoint $(-)^\flat$ to an ordinary \icat{}, which is then to be understood as coming equipped with the trivial $G$-action. By adjointness, the resulting functor is given by
  \begin{align*}
    \CatI&\to\Fun^\times(\xF_G^\op,\CatI)\simeq\Fun(\textbf{O}_G^\op,\CatI)\\
    \cC&\longmapsto\Fun(i(-)_{hG},\cC).
  \end{align*}
  In particular, the value of $\cC^{\flat}$ on an orbit $G/H$ is the \icat{} $\cC^{BH} := \Fun(BH, \cC)$ of objects of $\cC$ with an $H$-action, with the evident restriction functoriality. This suggests the following alternative description of the Borel $G$-$\infty$-category $\cC^{\flat}$ that connects it to the constructions of \cite{CLL_Global, CLL_Clefts}:

  Write $\Orb$ for the \icat{} of finite connected groupoids and faithful functors: in other words, the objects are groupoids of the form $BH$ for a finite group $H$, and the morphisms $BK \to BH$ are those induced by \textit{injective} group homomorphisms $K \to H$. By \cite[5.10]{CLL_Clefts} there is an equivalence $\textbf{O}_G\simeq\Orb_{/BG}$ sending $G$ to the homomorphism $1\to BG$; postcomposing with the forgetful functor and the inclusion yields a functor $\upsilon\colon\textbf{O}_G\to\CatI$. We claim that $\upsilon$ agrees with $i(-)_{hG}$. For this we observe that both agree on the full subcategory spanned by the object $G$ (where they are constant with value the terminal object), so it suffices that both are left Kan extended from this subcategory.

  For $i(-)_{hG}$ this is clear since it is the restriction of a cocontinuous functor $\Fun(BG,\Spc)\to\CatI$. For $\upsilon$, it suffices that $\Orb_{/BG}\to\CatI$ is left Kan extended. But $\Orb_{/BG}$ is a full subcategory of $(\Spc)_{/BG}$, so it will be enough that the forgetful functor $\Spc_{/BG}\to\CatI$ is left Kan extended from the full subcategory spanned by $1\to BG$. However, straightening--unstraightening provides an equivalence $\Spc_{/BG}\simeq\Fun(BG,\Spc)$ sending $1\to BG$ to the corepresented functor $G=\Map_{BG}(*,\blank)$ so this follows again from cocontinuity.
\end{remark}

\subsection[Normed structures on Borel $G$-\icats{}]{\boldmath Normed structures on Borel $G$-\icats{}}
Recall from \Cref{def:Normed_Category} that a \textit{normed structure} on a $G$-$\infty$-category $\xF_G^{\op} \to \CatI$ is an extension to a product-preserving functor $\Span(\xF_G) \to \CatI$. Given a symmetric monoidal \icat{} $\cC$ with $G$-action, the Borel $G$-\icat{} $\cC^\flat$ comes equipped with a canonical normed structure that we will refer to as the \textit{Borel normed structure}:

\begin{propn}[\cite{global-picard}*{3.4 and~3.6}, \cite{borel-normed}*{3.3.3}]\label{prop:borel-G-sym-mon}
  The adjunction from \cref{constr:borel-g-cat} lifts to an adjunction
  \[\Fun^{\times}(\Span(\xF_G),\CatI) \rightleftarrows\Fun(BG,\CMon(\CatI))\noloc (-)^\flat,\]
  i.e.~the forgetful functor $\Fun^{\times}(\Span(\xF_G),\CatI)\to\Fun(BG,\CMon(\CatI))$ has a right adjoint $(-)^\flat$, and the Beck--Chevalley transformation
  \[
    \begin{tikzcd}
      \Fun(BG,\CMon(\CatI))\arrow[d, "\mathbb U"'] \arrow[r, "(-)^\flat"] & \Fun^{\times}(\Span(\xF_G),\CatI)\arrow[d, "\mathbb U"]\arrow[dl, Rightarrow]\\
      \Fun(BG,\CatI)\arrow[r, "(-)^\flat"'] & \Fun^\times(\xF_G^\op,\CatI)
    \end{tikzcd}
  \]
  is invertible.\qed
\end{propn}

\begin{cor}\label{cor:borel-sm-ff}
 The right adjoint
 \[
 	(-)^\flat\colon \Fun(BG,\CMon(\CatI)) \to \Fun^{\times}(\Span(\xF_G),\CatI)
 \]
 is fully faithful, with essential image those normed $G$-\icats{} whose underlying $G$-\icat{} is Borel.
  \begin{proof}
    As $\mathbb U$ is conservative, the Beck--Chevalley condition readily implies that the counit is invertible, as it is so for the original adjunction $\Fun^\times(\xF_G^\op,\CatI)\rightleftarrows\Fun(BG,\CatI)$, proving full faithfulness. Arguing in the same way about the units yields the characterization of the essential image.
  \end{proof}
\end{cor}

\begin{cor}\label{cor:borel-unique-family}
  Let $F\colon\Fun(BG,\CMon(\CatI))\to\Fun^{\times}(\Span(\xF_G),\CatI)$ be any functor equipped with a natural equivalence $\epsilon\colon\ev_{G}\circ F\to\id$, and assume that $F$ takes values in Borel $G$-\icats{}. Then $F$ is right adjoint to the evaluation functor, with counit given by $\epsilon$.
  \begin{proof}
    By adjointness, there is a unique natural transformation $F\to(-)^\flat$ that upon evaluation at $G\in\xF_G$ recovers $\epsilon$. As this evaluation functor is conservative on Borel $G$-\icats{}, this map is then an equivalence as desired.
  \end{proof}
\end{cor}

In the same way one shows the following pointwise version:

\begin{cor}\label{cor:borel-unique-single}
  If the $G$-\icat{} $\cC\colon\xF_G^\op\to\CatI$ is Borel, then any $G$-equivariant symmetric monoidal structure on $\cC(G)\in\Fun(BG,\CatI)$ lifts uniquely to a normed structure on $\cC$.\qed
\end{cor}

\begin{observation}\label{obs:pw-mon-borel}
  If $\cC$ is a symmetric monoidal \icat{} with $G$-action, then we get two natural symmetric monoidal structures on $\cC^{hH}\simeq\cC^\flat(G/H)$ for any $H\leqslant G$: on the one hand, we can equip $\cC^{hH}$ with the symmetric monoidal structure obtained from the one on $\cC$ by taking homotopy fixed points; on the other hand, we can restrict $\cC^\flat$ along the functor $\Span(\xF)\to\Span(\xF_G)$ induced by $G/H\times\blank$. The Eckmann--Hilton argument then shows that these two structures agree, i.e.~the covariant functoriality of $\cC^\flat$ in fold maps is induced by the given symmetric monoidal structure by passing to homotopy fixed points.
\end{observation}

\subsection{The classical case}
Let $\cC$ be a symmetric monoidal \icat{} with $G$-action. So far, we have completely described the contravariant functoriality of $\cC^\flat$, as well as the covariant functoriality with respect to fold maps. The goal of this final subsection is to provide the only missing piece of information, namely the covariant functoriality with respect to maps $G/K\to G/H$ for $K\leqslant H\leqslant G$, when $\cC$ is a $1$-category.

It turns out that the hardest part of this is actually not understanding the $\infty$-categorical side, but rather translating this back through the equivalence between (classical, biased) symmetric monoidal $1$-categories and the \icat{} of commutative monoids in $\Cat$ \cite{shimada-shimakawa,sharma}. We therefore take a somewhat different route here: namely, we will give a general result characterizing the structure maps $\cC^{hK}\to\cC^{hH}$ uniquely, and then give a (classical) construction satisfying these assumptions. As an upshot, we will never need to recall \emph{how} the functor $\Phi$ from symmetric monoidal $1$-categories to symmetric monoidal \icats{} actually works, except for the following basic facts:
\begin{itemize}
  \item $\Phi$ is fully faithful with essential image those symmetric monoidal $\infty$-{\hskip0pt}categories whose underlying category is a $1$-category.
  \item $\Phi$ preserves underlying categories, and for every $\cC$ the maps $*\to\cC$ and $\cC^{\times 2}\to\cC$ obtained via the functoriality of $\Phi(\cC)$ in $\bfn0\to\bfn1$ and $\bfn2\to\bfn1$ are given by the inclusion of the unit and the tensor product, respectively.
\end{itemize}
For definiteness, we fix $\Phi$ to be the equivalence from \cite{sharma}*{6.19} between
$\CMon(\Cat)$ and the \icat{} obtained from the $1$-category
$\textsf{PermCat}_1^\textup{strict}$ of \emph{permutative categories}
(i.e.~symmetric monoidal categories in which associativity and
unitality hold on the nose) and \emph{strict} symmetric monoidal
functors by Dwyer--Kan localizing at the underlying equivalences of categories. Below, we will frequently and
implicitly extend $\Phi$ to the analogous localizations of the $1$-categories
$\SymMonCat_1^\textup{strict}$ of symmetric monoidal categories
and \emph{strict} symmetric monoidal functors as well as
$\SymMonCat_1^\textup{stron\smash g}$ using the following
reformulation of Mac~Lane's coherence theorem, as refined symmetrically in \cite{may-perm}*{4.2}:

\begin{lemma}
  The inclusions
  \[
    \textup{\textsf{PermCat}}^\textup{strict}_1 \hookrightarrow\SymMonCat^\textup{strict}_1 \hookrightarrow\SymMonCat^\textup{stron\smash g}_1
  \]
  induce equivalences on Dwyer--Kan localizations.
  \begin{proof}
    For the composite $\textup{\textsf{PermCat}}^\textup{strict}_1 \hookrightarrow\SymMonCat^\textup{stron\smash g}_1$ this appears for example as \cite{parsum}*{1.19}. As part of the proof, the reference constructs a functor $\Pi\colon\SymMonCat^\textup{stron\smash g}_1\to  \textup{\textsf{PermCat}}^\textup{strict}_1$ together with a natural \emph{strong} symmetric monoidal equivalence $\nu\colon\cC\to\Pi\cC$ for any symmetric monoidal category $\cC$. It will therefore suffice to show that there exists a natural zig-zag of \emph{strict} symmetric monoidal equivalences between $\cC$ and $\Pi\cC$.

    This is actually an instance of a general construction: Define $\Xi\cC$ to be the category with objects triples of an object $X\in\cC$, an object $Y\in\Pi\cC$, and an isomorphism $\sigma\colon\nu(X)\isoto Y$. A morphism in $(X,Y,\sigma)\to(X',Y',\sigma')$ is given by a pair of a map $X\to X'$ and a map $Y\to Y'$ making the obvious diagram commute. This becomes a functor in $\cC$ in the obvious way, and the forgetful maps provide natural equivalences $\cC\isofrom\Xi\cC\isoto\Pi\cC$.

    We now make $\Xi\cC$ into a symmetric monoidal category as follows: the tensor product of objects is given by \[(X,Y,\sigma)\otimes(X',Y',\sigma')=\big(X\otimes X',Y\otimes Y',(\sigma\otimes\sigma')\circ\psi^{-1}\big)\] where $\psi\colon \nu(X)\otimes \nu(X')\isoto \nu(X\otimes X')$ denotes the structure isomorphism of the symmetric monoidal functor $\nu$. The unit is given by the inverse structure isomorphism $\nu(\bbone)\isoto\bbone$ of $\nu$, while the tensor product of morphisms as well as the associativity, unitality, and symmetry isomorphisms for $\Xi\cC$ are simply given pointwise. We omit the straightforward verification that this is well-defined and a symmetric monoidal category. It is then clear from the definitions that the projections $\cC\gets\Xi\cC\to\Pi\cC$ are strict symmetric monoidal. By direct inspection, they are still natural when considered as maps in $\SymMonCat_1^\textup{stron\smash g}$, which then completes the proof of the lemma.
  \end{proof}
\end{lemma}

We can now state our key technical lemma, whose proof will be given below after some preparations.

\begin{propn}\label{prop:norm-unique}
  Let $K\leqslant H\leqslant G$, and let $h_1,\dots,h_r$ be orbit representatives for $H/K$. Then there exists a \emph{unique} natural transformation
  \[
    \nu\colon\cC^{hK}\to\cC^{hH}
  \]
  of functors $\Fun(BG,\SymMonCat^\textup{strict}_1)\to\Cat$ lifting $\cC\to\cC, X\mapsto\bigotimes_{i=1}^rh_i.X$.
\end{propn}

\begin{ex}\label{ex:norm-us-an-example}
  If $\cC$ is any symmetric monoidal \icat{} with $G$-action, then the structure map $\cC^\flat(G/K=G/K\to G/H)\colon\cC^{hK}\to\cC^{hH}$ lifts the twisted $r$-fold tensor product $X\mapsto\bigotimes_{i=1}^rh_i.X$: this follows at once by computing the composition
  \[
    \begin{tikzcd}[column sep=tiny,row sep=small]
      &[-.75em] G/K\arrow[dl,equal]\arrow[dr] &[-.75em]& G\arrow[dl]\arrow[dr,equal]\\
      G/K &&[1.5em] G/H &[1.5em]& G
    \end{tikzcd}
  \]
  in $\Span(\xF_G)$, cf.~\cite{borel-normed}*{3.2.1}.
\end{ex}

\begin{construction}\label{constr:classical}
  Let $\cC$ be a symmetric monoidal $1$-category with $G$-action. Recall that the \emph{symmetric monoidal norm} $\Nm^H_K\colon\cC^{hK}\to\cC^{hH}$ is given as follows: we send a $K$-homotopy fixed point $X$ (with structure isomorphisms $\phi_k\colon X\to k.X$) of $\cC$ to $\bigotimes_{i=1}^r h_i.X$ with structure isomorphisms
  \begin{equation*}
    \psi_h\colon\bigotimes_{i=1}^r h_i.X \to \bigotimes_{i=1}^r hh_i.x
  \end{equation*}
  given as follows: if $\sigma\in\Sigma_n$ and $\ell_1,\dots,\ell_r\in K$ satisfy $hh_i=h_{\sigma(i)}\ell_i$ for $i=1,\dots,r$, then $\psi_h$ is given as the composite
  \[
    \bigotimes_{i=1}^r h_i.X
    \isoto
    \bigotimes_{i=1}^r h_{\sigma(i)}.X
    \xrightarrow{\bigotimes h_{\sigma(i)}.\phi_{\ell_i}}
    \bigotimes_{i=1}^r h_{\sigma(i)}\ell_i.X=
    \bigotimes_{i=1}^r hh_i.x
  \]
  where the unlabelled isomorphism is given by permuting the tensor factors according to $\sigma$; on morphisms, $\Nm^{K}_H$ is simply given by $f\mapsto\bigotimes_{i=1}^r h_i.f$.

  We omit the straightforward but rather lengthy verification that this is well-defined. Note that this is clearly natural in \emph{strict} symmetric monoidal functors, and hence also satisfies the assumptions of the proposition.
\end{construction}

\begin{observation}
  If $\cC$ carries the trivial $G$-action, the above construction simplifies as follows: the assignment $h\mapsto(\sigma;\ell_1,\dots,\ell_r)$ defines a homomorphism $\iota\colon H\to\Sigma_r\wr K$, and the $H$-object $\Nm^H_KX$ is given by equipping $\bigotimes_{i=1}^rX$ with the restriction of the natural $\Sigma_r\wr K$-action on $X$ (by permuting the factors and via the individual $K$-actions) along $\iota$.
\end{observation}

The uniqueness part of Proposition~\ref{prop:norm-unique} now immediately implies:

\begin{cor}\label{cor:borel-norms-classical}
  Let $\cC$ be a symmetric monoidal $1$-category with $G$-action. Then the structure map $\cC^{hK}=\cC^\flat(G/K)\to\cC^\flat(G/H)=\cC^{hH}$ of the associated Borel $G$-\icat{} is given by the classical symmetric monoidal norm of Construction~\ref{constr:classical}.\qed
\end{cor}

It remains to prove the proposition.

\begin{lemma}\label{lemma:free-sym-mon-cat}
  Let $\mathbb P\colon\Fun(BG,\CatI)\to\Fun(BG,\CMon(\CatI))$ denote
  the left adjoint of the forgetful functor. Then the restriction of
  $\mathbb P$ to $\xF_G$ is given by $X\mapsto(\xF_{/X})^\simeq$, with
  functoriality and $G$-action via postcomposition. The unit is given
  by $X\to (\xF_{/X})^\simeq,x\mapsto (x\colon\{*\}\to
  X)$.
  \begin{proof}
    Note that the claim is clear for $G=1$ and $X=\{*\}$. For trivial $G$ and general $X$, we now observe that since $\mathbb P$ is a left adjoint, it in particular preserves finite coproducts. As $\CMon(\CatI)$ is semiadditive, it therefore suffices that $\prod_{x\in X}\xF_{/\{x\}}\simeq\xF_{/X}$ via the coproduct functor, which is simply the statement that $\xF$ is extensive.

    This finishes the proof for trivial $G$. The lemma follows as the left adjoint for general $G$ is simply given by taking the non-equivariant left adjoint and equipping it with the induced $G$-action.
  \end{proof}
\end{lemma}

\begin{proof}[Proof of Proposition~\ref{prop:norm-unique}]
  The existence of such a map was observed in Example~\ref{ex:norm-us-an-example}, so it only remains to prove uniqueness.

  As both $(-)^{hK}$ and $(-)^{hH}$ preserve underlying equivalences of categories, we may replace the source by its Dwyer--Kan localization. Combining the above with \cite{g-global}*{4.1.36}, we see that this localization is given by $\Fun(BG, \CMon(\Cat))$ (with the evident localization functor).

  On the other hand, we observe that since the forgetful functor $\cC^{hH}\to\cC$ is faithful (here it is crucial that $\cC$ is a $1$-category!), a functor $f\colon\cC^{hK}\to\cC^{hH}$ is uniquely described by its effect on cores together with the composition $\cC^{hK}\to\cC^{hH}\to\cC$. Thus, we are altogether reduced to showing that there is at most one natural transformation $((-)^{hK})^\simeq\to((-)^{hH})^\simeq$ of functors $\CMon(\Cat)\to\Spc$ lifting the twisted $r$-fold tensor product.

  Combining the Yoneda lemma with the representability result proven in Lemma~\ref{lemma:free-sym-mon-cat}, this translates to saying that the object \[\bigotimes\limits_{i=1}^r h_i.(1\to G/K)\cong (H/K\hookrightarrow G/K)\] of $\xF_{/(G/K)}^\simeq$ admits at most one lift to an $H$-homotopy fixed point. But this is immediate since it admits no non-trivial automorphisms.
\end{proof}

\end{document}